\documentclass[11pt,english,reqno]{smfart}
\usepackage[T1]{fontenc}
\usepackage[english,francais]{babel}
\usepackage{amssymb,url,xspace,smfthm}

\usepackage{enumerate}
\usepackage[numbers, sort&compress]{natbib}
\usepackage[shortlabels]{enumitem}
\usepackage{mathrsfs}
\usepackage[a4paper, hmargin=2.8cm, vmargin=2.5cm, headheight=12pt, headsep=6mm]{geometry}

\newtheorem{theorem}{Theorem}[section]
\newtheorem{lemma}[theorem]{Lemma}

\newtheorem{proposition}[theorem]{Proposition}
\newtheorem{definition}[theorem]{Definition}
 \theoremstyle{definition}
\newtheorem{remark}[theorem]{Remark}

\numberwithin{equation}{section}

\newcommand{\eps}{\varepsilon}
\newcommand{\norm}[1]{\Vert#1\Vert}
\newcommand{\abs}[1]{\left\vert#1\right\vert}
\newcommand{\set}[1]{\left\{\,#1\,\right\}}
\newcommand{\inner}[1]{\left(#1\right)}
\newcommand{\comi}[1]{\langle#1\rangle}

\newcommand{\normm}[1]{{ \vert\kern-0.25ex \vert\kern-0.25ex \vert #1
		\vert\kern-0.25ex \vert\kern-0.25ex \vert}}

\makeatletter
\def\@setaddresses{\par
  \nobreak \begingroup
\normalsize
  \def\author##1{\nobreak\addvspace\bigskipamount}%
  \def\\{\unskip, \ignorespaces}%
  \interlinepenalty\@M
  \def\address##1##2{\begingroup
    \par\addvspace\bigskipamount\noindent
    \@ifnotempty{##1}{(\ignorespaces##1\unskip) }%
    {\ignorespaces##2}\par\endgroup}%
  \def\curraddr##1##2{\begingroup
    \@ifnotempty{##2}{\nobreak\indent{\itshape Current address}%
      \@ifnotempty{##1}{, \ignorespaces##1\unskip}\/:\space
      ##2\par}\endgroup}%
  \def\email##1##2{\begingroup
    \@ifnotempty{##2}{\nobreak\noindent{\itshape E-mail address}%
      \@ifnotempty{##1}{, \ignorespaces##1\unskip}\/:
       ##2\par}\endgroup}%
   \def\urladdr##1##2{\begingroup
    \@ifnotempty{##2}{\nobreak\indent{\itshape URL}%
      \@ifnotempty{##1}{, \ignorespaces##1\unskip}\/:\space
      \ttfamily##2\par}\endgroup}%
  \addresses
  \endgroup
}
\makeatother


\newcommand{\BibTeX}{{\scshape Bib}\kern-.08em\TeX}
\newcommand{\T}{\S\kern .15em\relax }
\newcommand{\AMS}{$\mathcal{A}$\kern-.1667em\lower.5ex\hbox
        {$\mathcal{M}$}\kern-.125em$\mathcal{S}$}

\title[]{Local well-posedness for the Boltzmann equation
 with hard potentials}
\date {}
\author[]{Hao-Guang Li, Wei-Xi Li \and  Chao-Jiang Xu}

\address[H.-G. Li]{School of Mathematics and Statistics, South-Central Minzu University, Wuhan 430074, China}
\email{lihaoguang@scuec.edu.cn}

\address[W.-X. Li]{School of Mathematics and Statistics, \& Hubei Key Laboratory of Computational Science, Wuhan University, Wuhan 430072, China}
\email{wei-xi.li@whu.edu.cn}

\address[C.-J. Xu]{School of Mathematics and Key Laboratory of Mathematics MIIT,
Nanjing University of Aeronautics and Astronautics, Nanjing 210016, China}
\email{xuchaojiang@nuaa.edu.cn}


\begin{document}

\def\smfbyname{}

\begin{abstract}
We consider the spatially inhomogeneous non-cutoff Boltzmann equation with hard potentials in the non-perturbative setting. For  initial data with polynomial decay in the velocity variable, we establish the local-in-time existence and uniqueness of weak solutions, conditional to pointwise bounds on the hydrodynamic quantities (mass, energy, and entropy). Compared to the soft potential case, the key challenge for full-range hard potentials lies in the more severe loss of velocity moments. The proof combines a hypoelliptic estimate with interpolation inequalities to handle the moment-loss terms.
\end{abstract}


\maketitle

\tableofcontents

 \section{Introduction}
For general initial data, the existence of renormalized solutions was established in the seminal work of DiPerna-Lions \cite{MR1014927} (see Alexandre-Villani's work \cite{MR1857879} for the non-cutoff counterpart). However,  the regularity  and uniqueness  of such renormalized solutions remain   a challenging open problem.  
In recent years, the theory of conditional regularity has been well developed through a series of works by Imbert and Silvestre \cite{MR4433077, MR4229202,MR4195746,MR4049224} and their joint work with Mouhot \cite{MR4112729,MR4033752}.
This in turn motivates the investigation into the existence and uniqueness of weak solutions with minimal  regularity for large data. 
While such well-posedness has been established very recently for soft potentials by Henderson-Snelson-Tarfulea  \cite{MR4895262}, the hard potential case presents a major difficulty due to severe   loss of velocity moments (or velocity weights in the $L_{x,v}^2$ setting). To the best of our knowledge, only a few studies have addressed this problem, and only in specific contexts such as the Landau equation or the Boltzmann equation with moderately soft potentials.
  The purpose of this work is to establish  the local-in-time existence  and uniqueness for  full-range hard potentials conditional to point-wise bounds on the hydrodynamic quantities.

We study the  following   spatially inhomogeneous Boltzmann equation without angular cut-off:
\begin{equation}\label{eq-1}
\left\{
\begin{array}{ll}
   \partial_t f+v\cdot\partial_xf=Q(f,f),\\
  f|_{t=0}=f_{in},
\end{array}
\right.
\end{equation}
where $f(t,x,v)\geq 0$  represents the probability density function at time $t\geq0$, position $x\in\mathbb{T}^3$ with
velocity $v\in\mathbb{R}^3$.  The Boltzmann collision operator on the right-hand side of (\ref{eq-1}) is a bilinear operator defined as
\begin{align}\label{Q}
Q(g,f)(t,x,v)=\int_{\mathbb{R}^3\times\mathbb{S}^2}B(v-v_*,\sigma)(g'_*f'-g_*f)dv_*d\sigma
\end{align}
where and throughout the paper we use the standard shorthand $f'=f(t,x,v')$, $f=f(t,x,v)$, $g'_*=g(t,x,v'_*)$ and $g_*=g(t,x,v_*)$.  Here  $(v,v_*)$ and $(v',v'_*)$ denote the  post- and pre-collisional velocities, respectively, satisfying the conservation of momentum and energy:
$$v'+v'_*=v+v_*,\quad |v'|^2+|v'_*|^2=|v|^2+|v_*|^2.$$
These relations lead to the so-called  $\sigma$-representation with $\sigma\in \mathbb{S}^2$,
\begin{equation*}
\left\{
  \begin{aligned}
 & v'=\frac{v+v_*}{2}+\frac{|v-v_*|}{2}\sigma,\\
 & v'_*=\frac{v+v_*}{2}-\frac{|v-v_*|}{2}\sigma.
  \end{aligned}
\right.
\end{equation*}
The collision kernel $B(v-v_*,\sigma)$ in \eqref{Q} depends on the relative velocity  $|v-v_*|$ and the deviation angle $\theta$ defined by
\begin{equation}
	\label{angle}
	\cos\theta=\frac{v-v_*}{|v-v_*|}\cdot\sigma.
\end{equation}
Without loss of generality, we assume $B(v - v_*, \sigma)$ is supported on $0 \leq \theta \leq \pi/2$ (so that $\cos \theta \geq 0$) and takes the form
\begin{equation}\label{kern}
B(v - v_*, \sigma) = |v - v_*|^\gamma  b(\cos \theta) \textrm{ with }   0 \leq \sin \theta \, b(\cos \theta) \approx \theta^{-1 - 2s},
\end{equation}
where $|v - v_*|^\gamma$ is the kinetic part with $  \gamma >-3$, and $b(\cos \theta)$ is the angular part with singularity parameter $s\in (0,1)$.  Here and throughout the paper, $p \approx q$ (resp. $p \lesssim q$) means that $C^{-1} q \leq p \leq C q$ (resp. $p \leq C q$) for some generic constant $C \geq 1$. 
The angular part $b(\cos \theta)$ thus has a non-integrable singularity at $\theta = 0$:
\begin{equation*}
\int_0^{\frac{\pi}{2}} \sin \theta \, b(\cos \theta)  d\theta = +\infty.
\end{equation*}
The cases $-3 < \gamma < 0$, $\gamma = 0$, and $\gamma>0$ are referred to as soft potential, Maxwellian molecules, and hard potential, respectively.  In some literatures,  the range $0<\gamma+2s<2$  is often referred to as moderately soft potentials.  In this work, we consider the parameter regime that  $\gamma\geq 0 $ and $0<s<1,$  covering all hard potentials (including Maxwellian molecules) and mild-to-strong angular singularities.

The Landau equation arises as the grazing-collision limit of the Boltzmann equation and can be regarded as its diffusive model. It takes the form 
\begin{equation*}
	\partial_t f+v\cdot\partial_x f=Q_L(f,\ f),\quad 
f|_{t=0}=f_{in},
\end{equation*}
where the collision 
\begin{equation}\label{colli}
Q_L(g,\, f)=\sum_{1\leq i, j\leq 3}\partial_{v_i} \set{\int_{\mathbb R^3}a_{i,j}(v-v_*)\big[g(v_*)\partial_{v_j}
f(v)-f(v)(\partial_{v_j}g)(v_*)\big]dv_*},	
\end{equation}
with $(a_{i,j})_{1\leq i,j\leq 3}$  a symmetric nonnegative   matrix given by 
\begin{equation*} 
a_{i,j}(v)=\inner{\delta_{ij}|v|^2-{v_iv_j}}\abs v^{\gamma}.
\end{equation*}
Here   $\delta_{ij}$ denotes the Kronecker delta function.  
Note that  the Landau collision operator in \eqref{colli} takes a differential form, in contrast to the pseudo-differential nature of the Boltzmann operator. This structural difference generally simplifies the analysis of the Landau equation.

 \subsection{Notations and  Main Result}

 We introduce the following conventions and notations that will be used throughout the paper. The commutator of two operators $T$ and $S$ is denoted by 
 $[T,S]=TS-ST$.  The notation $\langle\, \cdot\, \rangle$ stands for $ (1+|\cdot|^2)^{\frac 12}$; in particluar,   $\langle v \rangle = (1+|v|^2)^{\frac 12}$ for $v\in\mathbb{R}^3$. For $r\in\mathbb{R},$ the operators $\langle D_x \rangle^r$ and $\langle D_v \rangle^r$  are the Fourier multipliers with symbols $\langle k \rangle^r$ and $\langle \eta \rangle^r$, respectively, that is,
\begin{equation}\label{fxfv}
\mathcal F_x\big( \langle D_x \rangle^r h\big)(k, v)=\langle k \rangle^r (\mathcal F_x h)(k, v), \    \mathcal F_v\big( \langle D_v \rangle^r h\big)(x, \eta)=\langle \eta \rangle^r (\mathcal F_v h)(x, \eta),
\end{equation}
where  $\mathcal F_x$ and $\mathcal F_v$ denote the partial Fourier transform  with respect to $x$ and $v$, with $(k,\eta)\in\mathbb Z^3\times\mathbb R^3$ being the dual variables of $(x,v)\in \mathbb T^3\times\mathbb R^3$. Similarly,  $\mathcal F_{x,v}$ stands for the full Fourier transform in $(x,v)$. 

The norm and inner product of $L^2(\mathbb T^3\times \mathbb R^3)$ are written as $\|\cdot\|_{L^2_{x, v}}$ and $(\cdot,\cdot)_{L^2_{x, v}}$. 
We  use the notation $\|\cdot\|_{L^2_x}$ (resp. $\|\cdot\|_{L^2_v}$)  when the variable $x$ (resp. $v$) is specified.  In addition, we use 
$$
H^p_x L^q_v=H_x^p(\mathbb T^{3}; L_v^q(\mathbb{{R}}^3))
$$
for the classical Sobolev space. The weighted Lebesgue space $L_r^1(\mathbb T^3\times\mathbb R^3)$
 is defined by 
 \begin{equation}\label{ltau}
 	L_r ^{1}(\mathbb T^3\times\mathbb R^3)=\Big\{h;\ \norm{h}_{L^1_r}:= \int_{\mathbb T^3\times\mathbb R^3} \comi v^r |h(x,v)| dxdv<+\infty\Big\}.
 \end{equation}
 
\begin{definition}\label{xel}
We define  the triple norm $\normm{\cdot} $ by 
\begin{equation}\label{trinorm}
\normm{h}^2=\norm{\comi{D_v}^{s}\comi v^{\frac{\gamma}{2}}h}_{L^2_v}^2+\norm{(-\triangle_{\mathbb S^2})^{\frac{s}{2}}\comi v^{\frac{\gamma}{2}} h}_{L_v^2}^2,
\end{equation}
with 
\[\Delta_{\mathbb{S}^2}=\sum_{\substack{ 1\le j,k\le3\\ j\neq k}}(v_j\partial_{v_k}-v_k\partial_{v_j})^2.\]
 Here $ (-\triangle_{\mathbb S^2})^{\frac{s}{2}}$ denotes  the fractional Laplacian on the sphere. Moreover, we set
\begin{equation}\label{nhx}
\begin{aligned}
\normm{h}^2_{L^2_x}=\int_{\mathbb{T}^3}\normm{h}^2dx=\norm{\comi{D_v}^{s}\comi v^{\frac{\gamma}{2}}h}_{L^2_{x,v}}^2+\norm{(-\triangle_{\mathbb S^2})^{\frac{s}{2}}\comi v^{\frac{\gamma}{2}} h}_{L_{x,v}^2}^2.
\end{aligned}
\end{equation}
Let 
$ \rho=  1+ \frac{14s+7\gamma}{6}$ with $s,\gamma$ the parameters in \eqref{kern}. 
For each $l\geq 3\rho,$   the Hilbert space  $\mathscr{X}_{l}$ is defined by 
\begin{equation*}
	\mathscr{X}_{l}=\big\{h\in H_x^3 L_v^2;\  \norm{h}_{\mathscr{X}_l}<+\infty \big\},
\end{equation*}
where
\begin{equation}\label{xn}
\|h\|^2_{\mathscr{X}_{l}}=\sum_{|\alpha|=3}\|\langle v\rangle^{l-\rho|\alpha|} \partial^{\alpha}_{x}f\|^2_{L^2_{x, v}}
+\|\langle v\rangle^{l} f\|^2_{L^2_{x, v}}.
\end{equation}
Furthermore, we define the  space $\mathcal{Y}_{l}$, equipped with the norm
\begin{equation}\label{yn}
\|h\|^2_{\mathcal{Y}_{l}}=\sum_{|\alpha|=3}\normm{\langle v\rangle^{l-\rho|\alpha|} \partial^{\alpha}_{x}f}^2_{L_x^2}+\normm{\langle v\rangle^{l} f}^2_{L_x^2}\,.
\end{equation}
\end{definition}

 \begin{remark}
The parameter $\rho$ is determined in Lemma \ref{inter-inequality}.  The key property of  space $\mathscr{X}_l$ is that  
 it controls terms involving higher-order spatial derivatives by means of a weaker weighted norm. This elegant idea, used in  Chaturvedi \cite{MR4646863} for the Landau equation with hard potentials, motivates our use of the hypoelliptic technique to overcome the loss of moments in velocity (see Subsection \ref{subsec:dm} for details). 
 \end{remark}

In this paper, we study the existence of solutions compatible with the conditional regularity established by Imbert and Silvestre \cite{MR4433077}.  Precisely, we seek solutions
 that are bounded away from vacuum and have finite mass, energy, and entropy densities.  Specifically, we assume the initial datum $f_{{in}}$ in \eqref{eq-1} satisfies $f_{{in}} \geq 0$ and for all $ x\in\mathbb T^3,$
 \begin{equation}\label{finite}
\left\{
\begin{aligned}
	&0<m_0\leq \int_{\mathbb R^3} f_{in}(x,v)dv\leq M_0, \\
& \int_{\mathbb R^3}  f_{in}(x,v) \abs{v}^2 dv\leq E_0,\\
&  \int_{\mathbb R^3}   f_{in}(x,v)\log f_{in}(x,v) dv\leq H_0,
\end{aligned}
	\right.
	\end{equation}
where, here and throughout the paper,  $m_0, M_0, E_0$, and $H_0$ are given positive constants.  We then prove that the solution $f$ to \eqref{eq-1} remains non-negative and moreover satisfies condition $\boldsymbol{(H)}$ defined as below:  
for all $(t, x)\in [0, T]\times\mathbb{T}^3$ with given $T$, it holds that
\begin{equation*}
\boldsymbol{(H)}\ \ \  \  \  \  \ \ \ \ \ \ \ \ \ \ \  \left\{
\begin{aligned}
	&\frac{m_0}{2}\leq \int_{\mathbb R^3} f(t,x,v)dv  \leq 2M_0, \\
& \int_{\mathbb R^3}  f(t,x,v) \abs{ v}^2 dv\leq 2E_0,\\
&  \int_{\mathbb R^3}   f(t,x,v)\log f(t,x,v) dv\leq 2H_0.
\end{aligned}
	\right.
\end{equation*}
 
\begin{theorem}\label{local-exsitence}
Suppose that the collision kernel $B$ takes the form \eqref{kern}  with parameters
  $0 < s < 1$ and $ \gamma\geq 0$.  Then there exist two constants $\ell_0>\ell$,  
  depending only on $s$ and $\gamma$, such that  the following existence and uniqueness results hold.
  \begin{enumerate}[label={(\roman*)}, leftmargin=*, widest=i]
  \item Suppose  the initial datum $f_{in}\ge 0$ in \eqref{eq-1} satisfies  assumption \eqref{finite} and that
\begin{equation}\label{initial-datum}
\|f_{in}\|_{L^1_{\ell_0}}+\|f_{in}\|_{\mathscr{X}_{\ell}}<+\infty,
\end{equation}
then   the Boltzmann equation \eqref{eq-1} admits a  local-in-time solution 
$$
f\in L^\infty\inner{[0,T], \, L^1_{\ell_0}\cap \mathscr{X}_{\ell}}\cap L^2\inner{[0,T], \, L^1_{\ell_0+\gamma}\cap \mathcal{Y}_{\ell}}$$
for some $T>0.$  Moreover  $f$ is non-negative and satisfies condition $\boldsymbol{(H)}$. Recall the function spaces here are given in \eqref{ltau} and Definition \ref{xel}.
\item Let $g\in L^\infty([0,T], \mathscr{X}_{\ell})$ be any non-negative solution to equation \eqref{eq-1}, satisfying condition $\boldsymbol{(H)}$. Then $f\equiv g$ on $[0,T].$
  \end{enumerate}
\end{theorem}

\begin{remark}
	 The results of Theorem \ref{local-exsitence},  with the parameter $s$ replaced by $1$,  remain valid for the Landau equation.  Compared to the prior works \cite{MR4646863, MR4970362}, we do not require the initial data to have exponential or sub-exponential decay. 
     Moreover, the argument is applicable to the soft potential case ( $-3<\gamma<0$), thereby weakening the regularity requirement in  \cite{MR4433077,MR4112183,MR4416998}.
\end{remark}

\begin{remark}
	The choice of the  constants $\ell$ and $\ell_0$   will specified in Subsection \ref{subsec:ell}. 
\end{remark}

\begin{remark}
We do not address regularization effects in this paper.   
Here we work with initial data with polynomial decay, a setting compatible with the conditional regularity framework.   This enables us to apply the regularity results  \cite{MR4433077, MR4927804} to conclude the instantaneous $C^\infty$-smoothness of the solution constructed in Theorem \ref{local-exsitence}.  However, the  smoothing effect in a finer space, saying analytic or  Gevrey class, remains out of reach. The reason is that our argument relies crucially on the non-negativity of the solution, a property that persists for the solution itself but not for its higher derivatives. Consequently, the methods used in \cite{MR4930523, MR4612704} to establish analytic or Gevrey regularity  do not apply directly here.   
\end{remark}

\begin{remark}
Establishing uniqueness for the Cauchy problem \eqref{eq-1} with rough initial data is  a challenging problem.  For soft potentials, uniqueness has been proven  by Alexandre-Morimoto-Ukai-Xu-Yang(AMUXY) \cite{MR2861580} assuming $H_v^{2s}$-
regularity in velocity, and by Henderson-Snelson-Tarfulea \cite{MR4895262} for H\"older continuous initial data. In the moderately soft potential case, however, higher-order Sobolev regularity is required for uniqueness, as shown in  AMUXY \cite{MR2679369}.  For the Landau equation with hard potentials, uniqueness was obtained  by Chaturvedi \cite{MR4646863} for  initial data with high-order Sobolev regularity,  and by Snelson and Taylor \cite{MR4970362} for H\"older continuous data with no vacuum region  (a condition analogous to that in  \cite{MR4895262}).  We  also mention the recent work \cite{2025sun} of Alonso-Gualdani-Sun on  uniqueness  for certain related diffusive models.  
\end{remark}

\begin{remark}
To handle the highly  nonlinear collision operator,    \(L^\infty_xL^2_v\) appears  to be a natural  function space for the well-posedness theory,   or the larger space $H^{\frac 32 +}_x L_v^2$ if the energy method applies.   In this paper, we  restrict the spatial regularity to \(H^3_x\) rather than \(H^{\frac 32+}_x\) in order to derive a lower bound on the density function, following the argument as in \cite{preprint-JHL}. We believe that the  $H_x^3$-regularity requirement could potentially be relaxed to a weaker one.  
\end{remark}


\subsection{Related works}

There is extensive literature on the well-posedness of the Boltzmann equation; for the cutoff case, we refer to Villani's survey \cite{MR1942465}. 
For the non-cutoff and close-to-equilibrium settings, global well-posedness  and smoothing effect are well established. As this is not our main focus, we refer to the  recent works \cite{MR2863853,MR2784329,MR4930523,MR4230064} and the references therein.

 Here, we briefly review related works on the spatially inhomogeneous  non-cutoff Boltzmann equation and the Landau equation for the non-perturbative setting. For general initial data, the global existence of renormalized solutions was established by Alexandre-Villani \cite{MR1857879} (cf. DiPerna and Lions \cite{MR1014927} for the cutoff Boltzmann equation). However, the global regularity and uniqueness of such renormalized solutions remain challenging open problems.
Our work is strongly motivated by the theory of regularity conditional to the macroscopic bounds as in condition $\boldsymbol{(H)}$. This line of research, often called the conditional regularity, was initiated by Silvestre \cite{MR3582250,MR3551261} and further developed in collaboration with Imbert and Mouhot \cite{MR4433077, MR4229202,MR4195746,MR4049224,MR4112729,MR4033752}.  The theory is now well understood owing  to Imbert and Silvestre \cite{MR4433077}, who proved that any solution satisfying the macroscopic bounds as in condition $\boldsymbol{(H)}$ becomes instantaneously $C^\infty$-smooth in all variables. See also the work \cite{MR4927804} of Fern\'{a}ndez-Real,  Ros-Oton and Weidner   for the recent improvement. 
Since the aforementioned works mainly concern regularity,  a natural follow-up   question is the existence of solutions that are compatible with the assumptions of the conditional regularity theory. While it is straightforward to verify that the results of Imbert-Silvestre \cite{MR4433077} apply to close-to-quilibrium solutions, the existence of solutions for rough initial data in a non-perturbative setting remains far from settled.  The   
 local  existence and the $C^\infty$-regularization effect were first established  by   AMUXY   \cite{MR2679369, MR3177640}  for initial data with high-order Sobolev regularity  and  Gaussian decay in velocity,  under the conditions    $s<\frac12$ and $\gamma+2s<1.$   For the soft potentials,   see also \cite{preprint-JHL} for the local existence and uniqueness  of solutions that instantly belongs to analytic or a sharp Gevrey space, even for  initial data with low regular and  Gaussian decay. We refer to \cite{MR3906224,MR4168919} for the related works on the Landau equation with soft potentials. 
 However, Gaussian decay is not fully compatible with the conditional regularity framework, because the uniform estimates in \cite{MR4433077} do not guarantee the persistence of  the Gaussian decay.  From the point view of the conditional regularity program, it is therefore more natural to seek solutions with merely  polynomial decay.  For soft potentials, solutions with polynomial decay were first established by Morimoto and Yang \cite{MR3376931} for high-order Sobolev  regular initial data under the restrictions
$0<s<\frac12$ and  $\gamma +
2s>- \frac{3}{2}$. This was later extended by Henderson-Snelson-Tarfulea \cite{MR4112183} to  $0<s<1 $ (with the same bound on   $\gamma +
2s$),  and finally by Henderson-Wang \cite{MR4416998} for the full range $-3<\gamma<0$ and $0<s<1.$  A significant improvement in this direction is due to Henderson-Snelson-Tarfulea  \cite{MR4895262}. They established local existence  for $-3<\gamma<0$ and $0<s<1$ under minimal hypotheses: the initial data need only belong to a weighted $L^\infty$ space and be strictly positive on a small ball in phase space, thereby dispensing with all regularity requirements.  For uniqueness, however, they required stronger conditions on the initial data: H\"older continuity and the absence of vacuum regions.     
In contrast to the soft (or moderately soft) potential cases, the treatment of hard potentials is more delicate due to the severe moment loss issue. The literature on this subject remains sparse. Currently, only two results are available for the Landau equation with hard potentials. Chaturvedi \cite{MR4646863} first proved local well-posedness for initial data with sufficient Sobolev regularity  and exponential decay.  Very recently, Snelson-Taylor \cite{MR4970362} extended the prior result to the case of irregular initial data with sub-exponential decay.   Our results of Theorem \ref{local-exsitence} provide a step toward understanding  the Boltzmann equation with full-range hard potentials. Finally we mention that  conditional regularity  is closely related to the convergence-to-equilibrium problem investigated  by Desvillettes and Villani \cite{MR2116276, MR1787105}.

 In the literature cited above,   one can
identify at least two methods for handling regularity,  one
referred to  the De Giorgi-Nash-Moser theory and another one to H\"ormander's hypoelliptic technique.
The classical De Giorgi-Nash-Moser theory provides an effective method for establishing H\"older regularity for elliptic or parabolic equations with irregular coefficients in divergence form. On the other hand, H\"ormander's hypoelliptic theory offers a powerful framework for dealing with degenerate equations, even though its bracket conditions typically require the coefficients to possess a certain regularity. However, the original De Giorgi-Nash-Moser  theory does not apply directly to the degenerate Boltzmann equation.  Significant progress has been made by extending this framework to hypoelliptic kinetic equations with rough coefficients,  since the work  of Golse-Imbert-Mouhot-Vasseur  \cite{MR3923847}  and that  of Cameron-Silvestre-Snelson \cite{MR3778645}.  For a comprehensive overview and further references, we refer to Mouhot's survey \cite{MR3966858} and the very recent lectures by Brigati-Mouhot \cite{BC-25} and by Imbert \cite{2026imbert}.

Finally, the aforementioned works mainly concern local-in-time results. The global stability near vacuum was first studied by Luk \cite{MR3948345} for the Landau equation with moderately soft potentials. This result was subsequently extended by Chaturvedi \cite{MR4270852,MR4552228} to both the Boltzmann equation with moderately soft potentials (i.e., $0<s<1$ and $0<\gamma+2s<2$) and  the Landau equation with hard potentials.

\subsection{Difficulty and methodology for non-perturbative setting}\label{subsec:dm}

The lower and upper bounds of the collision operator play a crucial role in  the well-posedness of the Boltzmann equation. In the close-to-equilibrium setting, the linearized operator around the global Maxwellian dominates  and   essentially behaves as  (see  \cite{MR3950012} for instance) 
\begin{equation}\label{bl}
\comi v^{ \gamma }\comi{D_v}^{2s}+\comi v^{ \gamma}(-\triangle_{\mathbb S^2})^{s}+\comi v^{2s+\gamma}.
\end{equation}
This give a sharp coercivity which allows one to control the non-linear term under a smallness assumption. Consequently, global well-posedness can be established via the micro-macro decomposition method, see \cite{MR2863853,MR2793203,MR2784329} for details.  

However, in the non-perturbative setting, the coercivity estimates become scarce, as the weighted estimate in \eqref{bl}
is no longer available. If $g=g(t,x,v)$ is a non-negative function  satisfying condition $\boldsymbol{(H)}$, then  there exist two constants $c_0,   C_0>0$, depending only on the quantities $m_0, E_0$ and $ H_0$ in condition $\boldsymbol{(H)}$, such that (see \cite{MR3942041,MR2959943} for instance) 
\begin{equation}\label{lower+B}
(Q(g,f),f)_{L^2_v}\le - c_0\normm{f}^2+C_0\|\langle v\rangle^{s+\frac{\gamma}{2}}f\|^2_{L^2_{v}},
\end{equation}
and   
\begin{equation}\label{uppbound+B}
\big|\big( Q(g,f), h\big)_{L^2_{v}}\big|\lesssim \big(\|\langle v\rangle^{\gamma+2s}g\|_{L^1_v}+\|g\|_{L^2_v}\big) \normm{f} \big(\normm{h}+\|\langle v\rangle^{s+\frac{\gamma}{2}}h\|_{L^2_v}\big),
\end{equation}
where the trip norm $\normm{\, \cdot\, }$ is defined in \eqref{trinorm}. Thus the main difficulty lies in controlling terms involving weights like $\langle v \rangle^{s+\frac{\gamma}{2}}$ that appear in \eqref{lower+B} and \eqref{uppbound+B}.  One approach is  to work with initial data with Gaussian decay. This allows us to carry out $L^2$-energy estimates using a time-dependent Gaussian weight    
$e^{-(c-at)\comi v^2}$ with $a,c>0$. Differentiating this weight produces a damping term $a\comi v^2$, which can compensate for the weight loss when  $2s+ \gamma  \leq 2$. Consequently, well‑posedness can be established for soft potentials $(\gamma<0)$ and moderately soft potentials $ \gamma\le 2-2s)$; see \cite{MR2679369} for details. However, the weight-loss issue becomes essential when $2s+\gamma> 2.$ 
   
The novelty of our approach lies in overcoming the difficulty addressed above by combining an intrinsic {\em hypoelliptic} structure with delicate interpolation inequalities.  The key observation is the following estimate: for any $\eps>0,$
 \begin{equation}\label{hypo-3A}
 	\|\langle v\rangle^{s+\frac{\gamma}{2}}h\|_{L^2_{x, v}}  \leq  \eps  \|\langle v\rangle^{\frac{\gamma}{2(1+2s)}}\comi{D_x}^{\frac{s}{1+2s}}h\|_{L^2_{x, v}} +\eps\norm{\comi{D_v}^{s}\langle v\rangle^{\frac{\gamma}{2}}h}_{L^2_{x, v}}  +C_{\eps}\|\langle v\rangle^{3+7s+2\gamma}h\|_{L^1_{x, v}}.
 \end{equation}
 Here the first term on the right-hand side is provided by the hypoelliptic estimate, the second by the coercivity estimate, the third by the moment estimate. In this way, we can bound the weight-loss terms by the {\em spatial hypoelliptic estimate, the velocity coercivity estimate, and the large-moment estimate}. This explains why we require the initial data belong to  $L^1_{\ell_0}\cap \mathscr{X}_{\ell}$ with well-chosen $\ell_0>\ell.$   
The interplay between these two spaces compensates for the weight loss that occurs in the energy estimates.
Remark that hypoelliptic estimates have been extensively used to explore the instantaneous regularity and spectral property  of  the Boltzmann equation and related diffusive models. Here, we apply this technique for the first time to study the existence of solutions.

\subsection{Proof Scheme for Theorem \ref{local-exsitence}}

We prove the existence of weak solutions using an energy method combined with Picard iteration. Starting from  $f_0= c\,\comi v^{-2\ell_0},$ we define the iterative sequence by
\begin{equation*}
\big(\partial_t + v\cdot \partial_x\big) f_{n+1}
= Q(f_n, f_{n+1}),  \quad f_{n+1}|_{t=0}=f_{\text{in}}.
\end{equation*}
The proof is divided into the following four main steps, designed to close the energy estimates in  weighted function spaces.

\begin{enumerate}[leftmargin=*, widest=i]
\item 
Assume $f_n\ge 0$ satisfy condition $\boldsymbol{(H)}$  and that the   commutator 
$$
\Big | (\langle v\rangle^\ell  Q(f_n, f_{n+1}), h)_{L^2_{x,v}}-( Q(f_n,\langle v\rangle^\ell  f_{n+1}), h)_{L^2_{x,v}}\Big|
$$
is under control. Then, using the coercivity \eqref{lower+B} and the upper bound \eqref{uppbound+B},  we obtain that,    for $0<t\le T$ sufficiently small,
$$
\norm{f_{n+1}(t)}_{\mathscr{X}_{\ell}}^2  +    \int_0^T \norm{f_{n+1}}_{\mathcal{Y}_{\ell}}^2dt \lesssim \norm{f_{in}}_{\mathscr{X}_{\ell}}^2+ \int_0^T 
  \norm{\comi v^{\ell+s+\frac{\gamma}{2}}f_{n+1}(t)}_{L_{x,v}^2}^2dt+\cdots.
$$
The critical novelty lies in controlling the term 
$ \norm{\comi v^{\ell+s+\frac{\gamma}{2}}f_{n+1}}_{L_{x,v}^2}^2 $ when $s+\frac{\gamma}{2}>1$. This is the major difficulty for the hard potential case. To overcome it, we use the interpolation inequality \eqref{hypo-3A}, in which the second term can be absorbed by the coercivity. It therefore remains to estimate the first and the last terms in that inequality. 
\item To estimate the first term on the right-hand side of \eqref{hypo-3A}, we establish the following hypoelliptic estimate for the Boltzmann operator $Q(f_n, f_{n+1}).$ Under the hypothesis that   $f_n\ge 0$ and satisfies condition $\boldsymbol{(H)}$,
\begin{multline*}
    \int_0^T \norm{\comi v^{\frac{\gamma}{2(1+2s)}+\ell} \comi {D_x}^{\frac{s}{1+2s}} f_{n+1}}_{L^2_{x, v}}^2 dt
    \lesssim  \Big(\sup_{t\leq T}\norm{f_{n+1}}_{\mathscr{X}_\ell}\Big)^2 \\
      + \int_0^T  \norm{f_{n+1}}_{\mathcal{Y}_\ell}^2dt+  \int_0^T\big| \big(  \langle v\rangle ^\ell Q(f_n, f_{n+1}),\  P\langle v\rangle ^\ell  f_{n+1}\big)_{L^2_{x, v}}\big|dt
\end{multline*}
where $P: L_{x,v}^2\mapsto L_{x,v}^2$ is a well-chosen bounded operator.  
 
\item  To handle the last term in \eqref{hypo-3A}, we first prove that the non-negativity of   
 $f_n$	yields the non-negativity of  
$f_{n+1}$. This then enables us to  establish the desired moment estimate, based on    the following inequality
\begin{align*}
&\int Q(f_n, f_{n+1}) \langle v\rangle^{ \ell_0 } dx dv \lesssim - c_1 \|\langle v\rangle^{ \ell_0 +\gamma}f_{n+1}\|_{L^1_{x, v}}+ c_2\|  f_{n+1}\|_{L^{\infty}_xL^1_v}\|\langle v\rangle^{ \ell_0 +\gamma}f_n\|_{L^1_{x, v}}\\
 &\qquad\qquad + \|\comi v^{4+\gamma}f_n\|_{L_x^\infty L^1_v}\|\langle v\rangle^{ \ell_0}f_{n+1}\|_{L^1_{x, v}}+ \|\langle v\rangle^{4+\gamma}f_{n+1}\|_{L_x^\infty L^1_v} \norm{\comi v^{ \ell_0 }f_n}_{L^1_{x, v}},
\end{align*}
with $0<c_2\ll c_1$ and  $\ell_0=\ell+3+7s+2\gamma.$

\item  By combining the arguments from the previous steps, we obtain an  {\em \`a priori} estimate for the linear Boltzmann equation, based on a careful choice of the parameters $\ell$ and $\ell_0$. The existence of solutions to the original nonlinear equation \eqref{eq-1} then follows from a standard regularization argument, since the analysis provides an $\epsilon$-uniform estimate for the following regularized Cauchy problem (with $0<\epsilon\ll 1$):    
\begin{equation*}
	\big(\partial_t  +v\cdot\partial_x\big)f_{n+1}^\epsilon+\epsilon \comi v^{s+2\gamma} f_{n+1}^\epsilon=Q(f_n^\epsilon, f_{n+1}^\epsilon),\quad f_{n+1}^\epsilon|_{t=0}=f_{in}.
\end{equation*}
More details can be found at Sections   \ref{app:sub}-\ref{s:main}.
\end{enumerate}
While the proof of the main result involves delicate microlocal analysis techniques,  the main idea is displayed in the four steps outlined above.

\subsection{Choice of parameters in weight functions}\label{subsec:ell}
Throughout the paper,  we let  $\ell, \ell_0$ and $\boldsymbol{\ell_1}$ be three fixed  exponents  for weight functions, which is to be specified below.  First we define several 
constants related to the angular integration. 
For every $l\geq 2$ we introduce the following associated constants:
\begin{equation}\label{lamome}
\left\{
\begin{aligned}
   \boldsymbol{\lambda}_ l &:=2^{ \gamma}\int_{\mathbb S^2} b(\cos\theta) \Big(1- \cos^ l \frac{\theta}{2} \Big) \,d\sigma= 2^{\gamma}(2\pi) \int_0^{\frac{\pi}{2}}\sin\theta b(\cos\theta) \Big(1- \cos^ l \frac{\theta}{2} \Big)\,d\theta,\\
     \boldsymbol{\omega}_ l &:=2^\gamma\int_{\mathbb S^2} b(\cos\theta) \sin^ l \frac{\theta}{2}  \,d\sigma=2^\gamma (2\pi) \int_0^{\frac{\pi}{2}} \sin\theta b(\cos\theta) \sin^ l \frac{\theta}{2} \,d\theta.
\end{aligned}
\right.
\end{equation}
Note that
$\boldsymbol{\lambda}_l$ increases in $l$,  while $ \boldsymbol{\omega}_ l$ decreases in $l.$ Moreover,  under condition \eqref{kern}, one has (cf. \cite[(A.17) and (A.18)]{MR3125270})
 \begin{equation}\label{approximate}
 \boldsymbol{\lambda}_l\approx l^s, 
\quad  \boldsymbol{\omega}_l\approx \int^{\frac{\pi}{4}}_0 (\sin\theta)^{l-1-2s}d\theta\approx  \int^{\frac{\sqrt2}{2}}_0 t^{l-1-2s}dt\approx  l^{-1} 2^{-\frac{l}{2}}.
 \end{equation}
Denoting 
\begin{equation}\label{agams}
\boldsymbol{A}_{\gamma,s}= \int^{\frac{\pi}{2}}_0\sin\theta b(\cos\theta)  \Big(\frac{1}{\cos^{3+\gamma}\frac{\theta}{2}}-1\Big) d\theta \geq 0,
\end{equation}
which  depends only on $s$ and $\gamma,$ the cancellation lemma (cf. \cite[Lemma 1]{MR1765272}) reads
\begin{equation}
    \label{cance}
    \begin{split}
         \int_{\mathbb{R}^3\times\mathbb{S}^2}  \abs{v-v_*}^\gamma b(\cos\theta) (f'-f) \,d\sigma\,dv  =\boldsymbol{A}_{\gamma,s} \int_{\mathbb{R}^3}f(v)|v-v_*|^{\gamma}dv.
    \end{split}
\end{equation}
Recall $m_0,M_0$ are the quantities given in \eqref{finite} and let $\tilde c_0>0$ be  a given   constant (determined in   Lemma \ref{lem:te}).   
By \eqref{approximate}, there exists a constant ${\boldsymbol{\ell_1}}$  such that 
\begin{equation}\label{loc-l}
 {\boldsymbol{\ell_1}}>\frac{13}{2}+\gamma,\quad \boldsymbol{\lambda}_{2{\boldsymbol{\ell_1}}}\geq \frac{2^{4+3\gamma}M_0}{m_0}\boldsymbol{A}_{\gamma,s},\quad    \boldsymbol{\omega}_{\boldsymbol{\ell_1}-2-\gamma}     \leq  \frac{\tilde c_0}{32M_0},\quad \frac{\boldsymbol{\omega}_{\boldsymbol{\ell_1}}}{\boldsymbol{\lambda}_{\boldsymbol{\ell_1}}} \leq \frac{m_0}{4^{4+\gamma}M_0}.
\end{equation}
Letting $\boldsymbol{\ell_1}$ be given above,   we require the parameters  $\ell, \ell_0$ to satisfy that (see   Lemmas \ref{lem:small} and \ref{lem:te})
\begin{equation}\label{rhod}
\ell= 2\boldsymbol{\ell_1} +2 +\gamma+3\rho,\quad \ell_0=\ell+3+7s+2\gamma.
\end{equation}
Throughout the paper we will fix  $\ell, \ell_0 $ and $\boldsymbol{\ell_1}$ as specified above and denote
 \begin{equation}\label{A-0}
\boldsymbol{L}=\max\Big\{1, \ 4\|f_{in}\|_{\mathscr{X}_{\ell}}
+4\|f_{in}\|_{L^1_{\ell_0}}\Big\},
\end{equation}
which is a constant depending only on the initial datum $f_{in}$.

\subsection{Organization of the paper and conventions for constants}

The paper is organized as follows. Section \ref{S1} lists several useful estimates that will be frequently used later. Section \ref{app:sub} and Section \ref{s:priori} are devoted to establishing a hypoelliptic estimate and an {\it  \`a priori} estimate, respectively. The proof of the main results is presented in Section \ref{s:main}.  In  Appendix \ref{Appendix}, we present two types of change of variables.

To simplify notation, we adopt the convention that the capital letter 
 $C\ge 1$ denotes a generic positive constant which may change from line to line. This constant depends only on the parameter 
 $\ell,\ell_0,\boldsymbol{\ell_1}$ in \eqref{rhod}  and on the quantities 
 $m_0,M_0, E_0,H_0$. Moreover, we denote by $C_\tau$  a generic constant that additionally depends on the parameter $\tau$.

 \section{Preliminary Estimates}\label{S1}

This section is devoted to presenting several  estimates on the collision operator and weighted interpolation inequalities, which will be used throughout the subsequent analysis.  In this part  we suppose $f,g,h$ are  sufficiently regular functions so that all norms and integrals appearing subsequently are well-defined.


\subsection{Trilinear bound and coercivity estimate}\label{subsec:tri}

We recall the well-known lower and upper bounds for the Boltzmann collision operator $Q$.  For the proof of these estimates, interested readers may refer to  \cite{MR3942041, MR2847536, MR2959943} for instance. 

The first estimate is the trilinear upper bound (see, e.g., \cite[Theorem 1.4]{MR3942041}):  
	\begin{equation}\label{uppbound}
	\big|\big( Q(g,f), h\big)_{L^2_{v}}\big|\leq C_1 \big(\|\langle v\rangle^{\gamma+2s}g\|_{L^1_v}+\|g\|_{L^2_v}\big) \normm{f} \big(\normm{h}+\|\langle v\rangle^{s+\frac{\gamma}{2}}h\|_{L^2_v}\big),
	\end{equation}
where the triple norm $\normm{\cdot}$ is defined as in \eqref{trinorm},  and  $C_1$ is  a   constant depending only on the parameters $s$ and $\gamma$ in  \eqref{kern}. The above upper bound is used to establish the existence of weak solutions. To obtain uniqueness we will apply the following more flexible upper bounds (see, e.g. \cite[Theorem 1.1]{MR3942041} or \cite[Proposition 2.9]{MR2847536}): 
\begin{equation}\label{uppbound2}
	\big|\big( Q(g,f), h\big)_{L^2_{v}}\big|\leq C_1 \big(\|\langle v\rangle^{\gamma+2s}g\|_{L^1_v}+\|g\|_{L^2_v}\big)\|\comi{D_v}^s\comi v^{a_1} f \|_{L_v^2}\|\comi{D_v}^s\comi v^{a_2}h\|_{L_v^2},
	\end{equation}
where $a_1, a_2\geq 0$ are arbitrarily given constants satisfying $a_1+a_2=\gamma+2s.$ 
    
Next, we recall the coercivity estimate for the Boltzmann collision operator.  Assume that  $g=g(t,x,v)$ is a non-negative function  satisfying condition $\boldsymbol{(H)}$. Then  there exist two constants $c_0,   C_0>0$, depending only on  the quantities $m_0, M_0, E_0$ and $ H_0$ in condition $\boldsymbol{(H)}$, such that the following lower bound holds (see \cite[Theorem 1.2]{MR3942041} or \cite{MR2959943} for instance):
\begin{equation}\label{lower+}
(-Q(g,f),f)_{L^2_v}\geq c_0\normm{f}^2-C_0\|\langle v\rangle^{s+\frac{\gamma}{2}}f\|^2_{L^2_{v}}.
\end{equation}
A variant of this estimate, which does not involve the last term in \eqref{trinorm}, takes the following form (cf. \cite[Proposition 2.1]{MR2959943}):
\begin{equation}\label{lower+2}
(-Q(g,f),f)_{L^2_v}\geq c_0\|\comi{D_v}^s\langle v\rangle^{\frac{\gamma}{2}}f\|_{L^2_v}^2-C_0\|\langle v\rangle^{\frac{\gamma}{2}}f\|^2_{L^2_{v}}.
\end{equation}
 Note that  the constants $c_0$ and $C_0$ above are independent of $x$. Integrating \eqref{lower+} and \eqref{lower+2} with respect to  $x\in \mathbb T^3$ yields
  \begin{equation}
  	\label{lower}
  		(-Q(g,f),f)_{L^2_{x, v}}\geq c_0\normm{f}_{L^2_x}^2-  C_0 \|\comi v^{s+\frac{\gamma}{2}}f\|^2_{L^2_{x, v}},
  \end{equation}   
  and
  \begin{equation*}
      (-Q(g,f),f)_{L^2_{x, v}}\geq c_0\|\comi{D_v}^s\langle v\rangle^{\frac{\gamma}{2}}f\|_{L^2_{x, v}}^2
      -C_0\|\langle v\rangle^{\frac{\gamma}{2}}f\|^2_{L^2_{x, v}}.
  \end{equation*}
The following norm equivalence, as established in \cite[Lemma 2.2]{MR2425608}, 
    \begin{equation}\label{equal-norm}
        \|\langle v\rangle^{a}\langle D_v\rangle^{b}h\|^2_{L^2_v} \approx \|\langle D_v\rangle^{b}\langle v\rangle^{a}h\|^2_{L^2_v} 
     \end{equation}
will be used in the following discussion. 
By this and the fact that 
\begin{equation*}
    \norm{(-\Delta_{\mathbb S^2})^{\frac{s}{2}} h}_{L_v^2}\leq C_s\norm{\comi{D_v}^s\langle v\rangle^{s} h}_{L_v^2}
\end{equation*}
for some constant $C_s$ depending only on $s$,   we have 
\begin{equation}
    \label{upbd}
    \normm{h}\leq C_{\gamma,s}\|\comi{D_v}^s\langle v\rangle^{s+\frac{\gamma}{2}}h\|_{L^2_v},
\end{equation}
where $C_{\gamma,s}$ is a constant depending only on $s,\gamma,$ and $\normm{h}$ is defined as in \eqref{trinorm}.

\subsection{Interpolation  inequalities}\label{interpolation}

The upper bound   \eqref{uppbound} and the last term in the coercivity estimate \eqref{lower}  account for the loss of weights.  The following lemma  is crucial for overcoming this problem, as it
  provides control of the weight-loss term  in terms of  the   intrinsic diffusion and  
 the moment.

\begin{lemma}\label{embedding}
For any $\varepsilon>0$ it holds that
\begin{equation*}
\begin{aligned}
&\|\langle v\rangle^{s+\frac{\gamma}{2}}h\|_{L^2_{x, v}}^2 \leq \varepsilon\big(\|\langle v\rangle^{\frac{\gamma}{2(1+2s)}}\comi{D_x}^{\frac{s}{1+2s}}h\|_{L^2_{x, v}}^2+\normm{h}_{L_x^2}^2\big)+C \varepsilon^{-\frac{3+6s}{s}} \|\langle v\rangle^{3+7s+2\gamma}h\|_{L^1_{x, v}}^2\, ,
\end{aligned}
\end{equation*}
recalling   $\comi{D_x}^r$ is the Fourier multiplier defined as in \eqref{fxfv}  and the norm $\normm{h}_{L_x^2}$ is defined in \eqref{nhx}.
\end{lemma}

\begin{proof}
In this proof, we denote 
\begin{equation}\label{pq}
	p=\frac{3+7s}{3+5s},\quad q=\frac{3+7s}{2s}.
\end{equation}
Note that $p^{-1}+q^{-1}=1$. Then
it follows from H\"older's inequality that
\begin{equation}\label{sgh}
\begin{aligned}
 \|\langle v\rangle^{s+\frac{\gamma}{2}}h\|^2_{L^2_{v}}&=\int_{\mathbb{R}^3} \langle v\rangle^{ 2s+\gamma} h^2  dv=\int_{\mathbb{R}^3} \langle v\rangle^{ \frac{3+7s+2\gamma}{p}-(5s+\gamma+3)} \comi v^{\frac{3+7s+2\gamma}{q}} |h|^{1+\frac{1}{p}} |h|^{\frac{1}{q}}  dv\\
&\leq \left(\int_{\mathbb{R}^3} \langle v\rangle^{3+7s+2\gamma-p(5s+\gamma+3)} |h|^{1+p}dv\right)^{\frac{1}{p}}\left(\int_{\mathbb{R}^3}\langle v\rangle^{3+7s+2\gamma}|h|dv\right)^{\frac{1}{q}}\\
&\leq \Big(\|\langle v\rangle^{\frac{\gamma(1+s)}{2(1+2s)}}h\|_{L^{1+p}_{v}}\Big)^{\frac{1+p}{p}}\Big(\|\langle v\rangle^{3+7s+2\gamma}h\|_{L^1_v}\Big)^{\frac{1}{q}},
\end{aligned}
\end{equation}
where
the last inequality holds since it follows from \eqref{pq} that  \begin{equation*}
	\frac{3+7s+2\gamma-p(5s+\gamma+3)}{1+p}=\frac{\gamma(1+s)}{2(1+2s)}.
\end{equation*}
In view of \eqref{pq},  we  apply  the following  Sobolev inequality (see, e.g.,  Theorem 1.1 in \S 1.1 of   \cite[Chapter V]{MR290095}):
\begin{equation}\label{GNS}
   \| g\|_{L^a(\Omega)}\leq C \|\inner{1-\Delta}^{\frac{b}{2}}g\|_{L^2(\Omega)}, \quad \Omega=\mathbb{R}^3  \textrm{ or } \mathbb T^3\ \textrm{ and}\    \frac{1}{a}=\frac{1}{2}-\frac{b}{3}, 
\end{equation}
with the choice $a=1+p$ and  $b=\frac{s}{2(1+2s)}$, to obtain that 
\begin{align*}
 \|\langle v\rangle^{\frac{\gamma(1+s)}{2(1+2s)}}h\|_{L^{1+p}_{v}}\leq C \|\comi{D_v}^{\frac{s}{2(1+2s)}} \langle v\rangle^{\frac{\gamma(1+s)}{2(1+2s)}}  h\|_{L^2_{v}}.
\end{align*}
 Consequently, combining the above inequality with \eqref{sgh} yields 
\begin{equation*}
\|\langle v\rangle^{s+\frac{\gamma}{2}}h\|^2_{L^2_{v}}\leq C \Big( \| \comi{D_v}^{\frac{s}{2(1+2s)}} \langle v\rangle^{\frac{\gamma(1+s)}{2(1+2s)}}  h\|_{L^2_{v}}\Big)^{\frac{1+p}{p}}\Big(\|\langle v\rangle^{3+7s+2\gamma}h\|_{L^1_v}\Big)^{\frac{1}{q}}.
\end{equation*}
Hence
\begin{equation}\label{rs}
\begin{aligned}
&\|\langle v\rangle^{s+\frac{\gamma}{2}}h\|^2_{L^2_{x, v}}
 \leq C\int_{\mathbb{T}^3}\Big( \|\comi{D_v}^{\frac{s}{2(1+2s)}}  \langle v\rangle^{\frac{\gamma(1+s)}{2(1+2s)}} h\|_{L^2_{v}}\Big)^{\frac{1+p}{p}}\Big(\|\langle v\rangle^{3+7s+2\gamma}h\|_{L^1_v}\Big)^{\frac{1}{q}}dx\\
&\leq C\left(\int_{\mathbb{T}^3}\|\langle D_v\rangle^{\frac{s}{2(1+2s)}}\langle v\rangle^{\frac{\gamma(1+s)}{2(1+2s)}}h\|^{1+p}_{L^2_{v}}dx\right)^{\frac{1}{p}}\Big(\int_{\mathbb T^3}\|\langle v\rangle^{3+7s+2\gamma}h\|_{L_v^1}dx\Big)^{\frac{1}{q}}\\
& \leq C\left(\int_{\mathbb{R}^3}\|\langle D_v\rangle^{\frac{s}{2(1+2s)}}\langle v\rangle^{\frac{\gamma(1+s)}{2(1+2s)}}h\|^{2}_{L^{1+p}_{x}}dv\right)^{\frac{1+p}{2p}}\Big(\|\langle v\rangle^{3+7s+2\gamma}h\|_{L^1_{x, v}}\Big)^{\frac{1}{q}},
\end{aligned}
\end{equation}
where the last inequality follows from Minkowski's inequality. 
Using again the Sobolev inequality \eqref{GNS} with $a=1+p$ and $b=\frac{s}{2(1+2s)},$ we conclude
$$\|g\|_{L^{1+p}_{x}}\leq C \|\langle D_x\rangle^{\frac{s}{2(1+2s)}}g\|_{L^2_{x}},$$
and thus, applying the above estimate for $g=\langle D_v\rangle^{\frac{s}{2(1+2s)}}\langle v\rangle^{\frac{\gamma(1+s)}{2(1+2s)}}h,$
\begin{multline*}
	 \left(\int_{\mathbb{R}^3}\|\langle D_v\rangle^{\frac{s}{2(1+2s)}}\langle v\rangle^{\frac{\gamma(1+s)}{2(1+2s)}}h\|^{2}_{L^{1+p}_{x}}dv\right)^{\frac{1+p}{2p}}\\
 	\leq C\left(\int_{\mathbb{R}^3}\|\langle D_x\rangle^{\frac{s}{2(1+2s)}} \langle D_v\rangle^{\frac{s}{2(1+2s)}} \langle v\rangle^{\frac{\gamma(1+s)}{2(1+2s)}}h\|^{2}_{L^{2}_{x}}dv\right)^{\frac{1+p}{2p}}\\
 \leq C \Big(\|\langle D_x\rangle^{\frac{s}{2(1+2s)}} \langle D_v\rangle^{\frac{s}{2(1+2s)}} \langle v\rangle^{\frac{\gamma(1+s)}{2(1+2s)}}h\|_{L^2_{x, v}}^2\Big)^{\frac{1+p}{2p}}.
\end{multline*}
For the last term above, we compute 
\begin{align*}
	&\|\langle D_x\rangle^{\frac{s}{2(1+2s)}} \langle D_v\rangle^{\frac{s}{2(1+2s)}} \langle v\rangle^{\frac{\gamma(1+s)}{2(1+2s)}}h\|_{L^2_{x, v}}^2 =\Big(\langle D_x\rangle^{\frac{s}{(1+2s)}}h,\   \langle v\rangle^{\frac{\gamma(1+s)}{2( 1+2s) }}\langle D_v\rangle^{\frac{s}{ 1+2s}} \langle v\rangle^{\frac{\gamma(1+s)}{2( 1+2s) }}h\Big)_{L^2_{x, v}}\\
	&\leq  \|\langle v\rangle^{\frac{\gamma}{2(1+2s)}} \langle D_x\rangle^{\frac{s}{1+2s}}h\|_{L^2_{x, v}}  \|\langle v\rangle^{-\frac{\gamma}{2(1+2s)}}  \langle v\rangle^{\frac{\gamma(1+s)}{2( 1+2s) }}\langle D_v\rangle^{\frac{s}{ 1+2s}} \langle v\rangle^{\frac{\gamma(1+s)}{2( 1+2s) }}h\|_{L^2_{x, v}}\\
	&\leq C  \|\langle v\rangle^{\frac{\gamma}{2(1+2s)}} \langle D_x\rangle^{\frac{s}{1+2s}}h\|_{L^2_{x, v}}  \| \langle D_v\rangle^{\frac{s}{ 1+2s }} \langle v\rangle^{ \frac{\gamma}{2}} h\|_{L^2_{x, v}},
\end{align*}
where the last inequality follows from \eqref{equal-norm}. Now we combine these estimates to get that
\begin{align*}
   	&\left(\int_{\mathbb{R}^3}\|\langle D_v\rangle^{\frac{s}{2(1+2s)}}\langle v\rangle^{\frac{\gamma(1+s)}{2(1+2s)}}h\|^{2}_{L^{1+p}_{x}}dv\right)^{\frac{1+p}{2p}}\\
 &\leq C \Big( \|\langle v\rangle^{\frac{\gamma}{2(1+2s)}} \langle D_x\rangle^{\frac{s}{1+2s}}h\|_{L_{x,v}^2}  \|\langle D_v\rangle^{ s} \langle v\rangle^{\frac{\gamma }{2}}h\|_{L^2_{x, v}}\Big)^{\frac{1+p}{2p}}.
   \end{align*}
 Substituting the above estimate into \eqref{rs} yields 
\begin{align*}
&\|\langle v\rangle^{s+\frac{\gamma}{2}}h\|^2_{L_{x,v}^2}
\leq C \Big( \|\langle v\rangle^{\frac{\gamma}{2(1+2s)}} \langle D_x\rangle^{\frac{s}{1+2s}}h\|_{L^2_{x, v}}  \|\langle D_v\rangle^{ s}\langle v\rangle^{\frac{\gamma }{2}} h\|_{L^2_{x, v}}\Big)^{\frac{1+p}{2p}}\Big(\|\langle v\rangle^{3+7s+2\gamma}h\|_{L^1_{x, v}}\Big)^{\frac{1}{q}}\\
&  \leq \varepsilon \|\langle v\rangle^{\frac{\gamma}{2(1+2s)}} \langle D_x\rangle^{\frac{s}{1+2s}}h\|_{L^2_{x, v}}\|\langle D_v\rangle^{ s}\langle v\rangle^{\frac{\gamma }{2}} h\|_{L^2_{x, v}} + C\varepsilon^{-\frac{q(1+p)}{p}}\|\langle v\rangle^{3+7s+2\gamma}h\|_{L^1_{x, v}}^2\\
&  \leq \varepsilon \|\langle v\rangle^{\frac{\gamma}{2(1+2s)}} \langle D_x\rangle^{\frac{s}{1+2s}}h\|_{L^2_{x, v}}^2+\varepsilon\|\langle D_v\rangle^{ s}\langle v\rangle^{\frac{\gamma }{2}} h\|_{L^2_{x, v}}^2 + C\varepsilon^{-\frac{q(1+p)}{p}}\|\langle v\rangle^{3+7s+2\gamma}h\|_{L^1_{x, v}}^2.
\end{align*}
Note $\frac{q(1+p)}{p}=\frac{3+6s}{s}$ and $\|\langle D_v\rangle^{ s}\langle v\rangle^{\frac{\gamma }{2}} h\|_{L^2_{x, v}}\leq \normm{h}_{L_x^2}$. Then
 the proof of Lemma \ref{embedding} is completed.
\end{proof}

\begin{lemma}\label{inter-inequality}
Let   $\tau \in \mathbb{R}$. For any $\eps>0$ it holds that  
	\begin{equation*}
    \begin{split}
        \sum_{|\alpha|=3}\norm{\comi v^{\tau+ s+\frac{\gamma}{2}}\partial_x^{\alpha} h}^2_{L^2_{x, v}}  &\leq \varepsilon \sum_{|\alpha|=3}\norm{\langle v\rangle^{\tau+\frac{\gamma}{2(1+2s)}} \langle D_x\rangle^{\frac{s}{1+2s}}\partial_x^{\alpha} h}^2_{L^2_{x, v}} +C\varepsilon^{-\frac {3+6s}{s}} \norm{\comi v^{\tau+3\rho}h}^2_{L^2_{x, v}}\, ,
    \end{split}
	\end{equation*}
	 recalling $\rho= 1+\frac{14s+7\gamma}{6}$.
\end{lemma}

\begin{proof}
 Recall that  $\mathcal{F}_xh$ is the Fourier transform of $h$ with respect to the space variable $x\in \mathbb{T}^3,$  with $k\in\mathbb Z^3$ the Fourier dual of $x$.
 It follows from Parseval's identity   that
\begin{equation}\label{par}
    \sum_{|\alpha|=3}\norm{\langle v\rangle^{\tau+ s+\frac{\gamma}{2}}\partial^{\alpha}_xh}^2_{L^2_{x, v}}
   = \sum_{|\alpha|=3}\,\sum_{k\in\mathbb Z^3} \norm{\langle v\rangle^{\tau+ s+\frac{\gamma}{2}} \mathcal F_x\inner{\partial_x^\alpha h}(k,\cdot)}^2_{L^2_{v}}.
\end{equation}
 Using  H\"{o}lder's inequality with the choice of $p$ and $q$ that 
\begin{equation}\label{pq+}
	\frac{1}{p}=\frac {3+6s}{3+7s}, \quad  \frac{1}{q}=\frac{s}{3+7s},
\end{equation}
 we get 
\begin{align*}
&\sum_{|\alpha|=3}\norm{\langle v\rangle^{\tau+ s+\frac{\gamma}{2}} \mathcal F_x\inner{\partial_x^\alpha h}(k,\cdot)}^2_{L^2_{v}}\leq \int_{\mathbb{R}^3} \langle v\rangle^{ 2 \tau+\gamma +2s}|k|^{6}|\mathcal F_x{h}(k,v)|^2 dv\\
&\leq \int_{\mathbb{R}^3}  \langle v\rangle^{ \frac{2 \tau}{p}+ \frac{\gamma}{ p(1+2s)}}\langle v\rangle^{\frac{2 \tau}{q}+ \gamma +2s- \frac{\gamma}{ p(1+2s)}}|k|^{6}|\mathcal F_x{h}(k,v)|^{\frac{2}{p}}|\mathcal F_x{h}(k,v)|^{\frac{2}{q}} dv\\
&\leq\bigg(\int_{\mathbb{R}^3}  \langle v\rangle^{  2 \tau + \frac{\gamma}{ (1+2s)}} |k|^{6p}|\mathcal F_x{h}(k,v)|^2 dv\bigg)^{\frac{1}{p}}  
\bigg(\int_{\mathbb{R}^3}  \langle v\rangle^{ 2 \tau +( \gamma +2s)q- \frac{ q\gamma}{ p(1+2s)}} |\mathcal F_x{h}(k,v)|^2 dv\bigg)^{\frac{1}{q}}\\
&\leq \norm{\langle v\rangle^{\tau+\frac{\gamma}{2(1+2s)}}|k|^{3+\frac{s}{1+2s}}\mathcal F_x{h}(k,\cdot)}_{L_v^2}^{\frac{2}{p}}
\norm{\langle v\rangle^{\tau+3+7s+ \frac{7\gamma}{2}}\mathcal F_x{h}(k,\cdot)}^{\frac{2}{q}}_{L_v^2},
\end{align*}
the last inequality holds since it follows from \eqref{pq+} that
\begin{equation*}
	6p=6+\frac{2s}{1+2s}, \quad ( \gamma +2s)q- \frac{ q\gamma}{ p(1+2s)}=6+14s+7\gamma.
\end{equation*}
Taking summation over $k\in\mathbb Z^3$  in the above estimate yields that, for any $\eps>0,$
\begin{equation*}
	\begin{aligned}
	&	\sum_{|\alpha|=3}\,\sum_{k\in\mathbb Z^3} \norm{\langle v\rangle^{\tau+ s+\frac{\gamma}{2}} \mathcal F_x\inner{\partial_x^\alpha h}(k,\cdot)}^2_{L^2_{v}}\\
	&\leq \sum_{k\in\mathbb Z^3} \Big(\varepsilon  \norm{\langle v\rangle^{\tau+\frac{\gamma}{2(1+2s)}}|k|^{3+\frac{s}{1+2s}}\mathcal F_x{h}(k,\cdot)}_{L_v^2}^{2}
+C \varepsilon^{-\frac{q}{p}} \norm{\langle v\rangle^{\tau+3+7s+ \frac{7\gamma}{2}}\mathcal F_x{h}(k,\cdot)}_{L_v^2}^2\Big)\\
&\leq \varepsilon \sum_{\abs\beta=3} \norm{\langle v\rangle^{\tau+\frac{\gamma}{2(1+2s)}}\comi{D_x}^{ \frac{s}{1+2s}}\partial_x^\beta h}_{L^2_{x, v}}^{2}
+C \varepsilon^{-\frac {3+6s}{s}} \norm{\langle v\rangle^{\tau+3+7s+ \frac{7\gamma}{2}}h}_{L^2_{x, v}}^2,
	\end{aligned}
\end{equation*}
where the last inequality holds since it follows from \eqref{pq+} that  $\frac qp= \frac{(3+6s)}s.$  This, together with \eqref{par} as well as the definition of $\rho$, completes the proof of Lemma \ref{inter-inequality}.
\end{proof}

\begin{lemma}\label{lem:small}
Recall $\ell,\boldsymbol{\ell_1}$ are determined as in \eqref{rhod}.  Then for any $\eps>0,$ it holds that 
 \begin{align*}
      \|\comi {D_v}^s   \comi v^{\boldsymbol{\ell_1}+2s+\frac{\gamma}{2}}  h \|_{H_x^2 L_v^2}^2\leq 
     \eps  \norm{ h}_{\mathcal{Y}_\ell}^2+C \eps^{-5}  \| \comi {D_v}^s\comi v^{\boldsymbol{\ell_1}+ \frac{\gamma}{2}}  h \|_{L^2_{x, v}}^2,
 \end{align*}
 where $\norm{\cdot}_{\mathcal{Y}_\ell}$ is defined in \eqref{yn}.
\end{lemma}

\begin{remark}
    This lemma will be used to overcome the lack of a smallness assumption on   initial data. 
\end{remark}

\begin{proof}[Proof of  Lemma \ref{lem:small}] Using \eqref{equal-norm}
 and Parseval's identity gives
    \begin{align*}
       & \|\comi {D_v}^s   \comi v^{\boldsymbol{\ell_1}+2s+\frac{\gamma}{2}}  h \|_{H_x^2 L_v^2}^2  \leq C \|   \comi v^{\boldsymbol{\ell_1}+2s+\frac{\gamma}{2}} \comi {D_v}^s h \|_{H_x^2 L_v^2}^2\\
        &\leq C \sum_{k\in \mathbb Z^3} \|\abs{k}^2   \comi v^{\boldsymbol{\ell_1}+2s+\frac{\gamma}{2}} \comi {D_v}^s F_x h(k,\cdot) \|_{L_v^2}^2 + C \sum_{k\in \mathbb Z^3} \|   \comi v^{\boldsymbol{\ell_1}+2s+\frac{\gamma}{2}} \comi {D_v}^s \mathcal F_x h(k,\cdot) \|_{L_v^2}^2.
    \end{align*}
    On the other hand, for any $\eps>0,$ we have 
    \begin{align*}
       & \|\abs{k}^2   \comi v^{\boldsymbol{\ell_1}+2s+\frac{\gamma}{2}} \comi {D_v}^s F_x h(k,\cdot) \|_{L_v^2}^2\\
       & \leq \eps \|\abs{k}^3   \comi v^{\boldsymbol{\ell_1}-s} \comi {D_v}^s F_x h(k,\cdot) \|_{L_v^2}^2+C\eps^{-2} \|    \comi v^{\boldsymbol{\ell_1}+8s+\frac{3\gamma}{2}} \comi {D_v}^s F_x h(k,\cdot) \|_{L_v^2}^2.
    \end{align*}
 Combining the above estimates yields, for any $\eps>0,$
 \begin{align*}
    & \|\comi {D_v}^s   \comi v^{\boldsymbol{\ell_1}+2s+\frac{\gamma}{2}}  h \|_{H_x^2 L_v^2}^2  \\
    & \leq  \eps \sum_{k\in\mathbb Z^3} \|\abs{k}^3   \comi v^{\boldsymbol{\ell_1}-s} \comi {D_v}^s F_x h(k,\cdot) \|_{L_v^2}^2+C\sum_{k\in\mathbb Z^3}\eps^{-2} \|    \comi v^{\boldsymbol{\ell_1}+8s+\frac{3\gamma}{2}} \comi {D_v}^s F_x h(k,\cdot) \|_{L_v^2}^2\\
    & \leq  \eps  \sum_{\abs\alpha=3}\|   \comi v^{\boldsymbol{\ell_1}-s} \comi {D_v}^s \partial_x^\alpha h  \|_{L^2_{x, v}}^2+C \eps^{-2} \|    \comi v^{\boldsymbol{\ell_1}+8s+\frac{3\gamma}{2}} \comi {D_v}^s h \|_{L^2_{x, v}}^2\\
    & \leq  \eps  \sum_{\abs\alpha=3}\|   \comi v^{\boldsymbol{\ell_1}-s} \comi {D_v}^s \partial_x^\alpha h  \|_{L^2_{x, v}}^2+ \eps \|    \comi v^{\boldsymbol{\ell_1}+16s+\frac{5\gamma}{2}} \comi {D_v}^s h \|_{L^2_{x, v}}^2 +C \eps^{-5} \|    \comi v^{\boldsymbol{\ell_1}+ \frac{\gamma}{2}} \comi {D_v}^s h \|_{L^2_{x, v}}^2\\
    & \leq  \eps  \sum_{\abs\alpha=3}\|   \comi v^{\ell-\rho\abs\alpha+\frac{\gamma}{2}} \comi {D_v}^s \partial_x^\alpha h  \|_{L^2_{x, v}}^2+\eps \|    \comi v^{\ell+\frac{\gamma}{2}} \comi {D_v}^s h \|_{L^2_{x, v}}^2+C \eps^{-5}  \|    \comi v^{\boldsymbol{\ell_1}+ \frac{\gamma}{2}} \comi {D_v}^s h \|_{L^2_{x, v}}^2,
 \end{align*}
 where the last inequality follows from \eqref{rhod}. This with \eqref{equal-norm} and the definition of $\norm{\cdot}_{\mathcal{Y}_\ell}$ in \eqref{yn} yields the assertion in Lemma \ref{lem:small}. The proof is completed. 
\end{proof}


\subsection{Moment estimate for the collision operator}

This subsection is devoted to establishing moment estimates for the collision operator. The main result of this part is stated as below.  

\begin{proposition}\label{momentQ}
Suppose the collision kernel  $B$ takes the form \eqref{kern}  with $0<s<1$ and $\gamma\geq 0$. For any $l\geq 5$, there exists a constant $C_l$ depending on $l,$ such that,      for any sufficiently regular functions $f,g\geq 0$ with condition $\boldsymbol{(H)}$ fulfilled by $g$,  
\begin{align*}
 \int_{\mathbb T^3\times\mathbb R^3} Q(g, f) \langle v\rangle^{ l } \,dxdv \leq& -\frac{m_0\boldsymbol{\lambda}_ l }{4^{\gamma+1 }}   \|\langle v\rangle^{ l +\gamma}f\|_{L^1_{x, v}}+2  \boldsymbol{\omega}_{ l }\|  f\|_{L^{\infty}_xL^1_v}\|\langle v\rangle^{ l +\gamma}g\|_{L^1_{x, v}}\\
 &+C_{l}\|\comi v^{4+\gamma}g\|_{L_x^\infty L^1_v}\|\langle v\rangle^{ l}f\|_{L^1_{x, v}}+C_{ l}\|\langle v\rangle^{4+\gamma}f\|_{L_x^\infty L^1_v} \norm{\comi v^{ l }g}_{L^1_{x, v}}\, ,
\end{align*}
where $\boldsymbol{\lambda}_ l $ and $\boldsymbol{\omega}_ l $ are the positive constants defined in \eqref{lamome}, and  $m_0$ is the   constant in \eqref{finite}.
\end{proposition}

We first list some useful inequalities to be used frequently in the subsequent argument.  Note that 
\begin{equation}\label{elementaryinequality}
2^{-\gamma}\langle v\rangle^{\gamma}-\langle v_*\rangle^{\gamma}\leq |v-v_*|^{\gamma}\leq  2^{ \gamma} \langle v\rangle^{\gamma} \langle v_*\rangle^{\gamma},\quad |v-v_*|^{\gamma}\leq  2^{ \gamma } \Big(\langle v\rangle^{\gamma}+ \langle v_*\rangle^{\gamma}\Big).
\end{equation} 
Thus, for $f,g\geq 0,$ 
 \begin{equation}\label{tlow}
\begin{aligned}
     \int_{\mathbb{R}^6} |v-v_*|^{\gamma} g_*f\langle v\rangle^{ l }  dv_*\,dv &\geq 2^{-\gamma} \int_{\mathbb{R}^6}    g_*f\langle v\rangle^{ l+\gamma }  dv_*\,dv-\int_{\mathbb{R}^6}  \comi{v_*}^{\gamma} g_*f\langle v\rangle^{ l }  dv_*\,dv\\
     &\geq 2^{-\gamma} \norm{g}_{L_v^1}\norm{\langle v\rangle^{ l+\gamma } f}_{L_v^1}-\norm{\comi v^{\gamma}g}_{L_v^1}\norm{\langle v\rangle^{ l} f}_{L_v^1}.
\end{aligned}
\end{equation}
Consequently, the following inequality 
\begin{equation*}
     \int_{\mathbb{R}^6} |v-v_*|^{\gamma} g_*f\langle v\rangle^{ l }  dv_*\,dv \geq \frac{m_0}{2^{1+\gamma}}\norm{\langle v\rangle^{ l+\gamma } f}_{L_v^1}- \norm{\comi v^{\gamma}g}_{L_v^1}\norm{\langle v\rangle^{ l} f}_{L_v^1}  
\end{equation*}
holds for $f,g\geq 0$ with $g$ satisfying condition $\boldsymbol{(H)}$ .  
Using the standard change of variables between pre- and post-collisional  velocities
\begin{equation}
	\label{prepost}
	(v,v_*,\sigma) \rightarrow(v',v_*', (v-v_*)/|v-v_*|),
\end{equation}
 whose Jacobian equals one (cf. \cite[Subsection 4.5]{MR1942465} for instance), we rewrite  
 \begin{equation}\label{moest}
  \begin{aligned}
 \int_{\mathbb{R}^3}  Q(g, f)  \langle v\rangle^{ l } dv  
 &=\int_{\mathbb{R}^6\times\mathbb{S}^2} |v-v_*|^{\gamma}b(\cos\theta) (g_*'f'-g_*f)\langle v\rangle^{ l }\,d\sigma\, dv_*\, dv\\
 &=\int_{\mathbb{R}^6\times\mathbb{S}^2} |v-v_*|^{\gamma}b(\cos\theta)g_*f\big(\langle v'\rangle^{ l } -\langle v\rangle^{ l }\big)\,d\sigma\, dv_*\, dv.
 \end{aligned}
 \end{equation}
Therefore, the proof of Proposition \ref{momentQ} reduces to establishing the following lemma.

\begin{lemma}
    \label{lem:cru}
    Suppose $f,g\geq 0$. Then  for any $l\geq5$ there exists a constant $C_{l},$ depending  on $l$, such that  
    \begin{align*}
       & \int_{\mathbb{R}^6\times\mathbb{S}^2} |v-v_*|^{\gamma}b(\cos\theta)g_*f\big(\langle v'\rangle^{ l } -\langle v\rangle^{ l }\big)\,d\sigma\, dv_*\, dv\\
       &\qquad\leq  - \frac{\boldsymbol{\lambda}_l}{2^{1+2\gamma}}  \norm{g}_{L_v^1}\norm{\langle v\rangle^{ l+\gamma } f}_{L_v^1}+2\boldsymbol{\omega}_ l  \norm{f}_{L_v^1}\norm{\comi v^{ l +\gamma}g}_{L_v^1}  \\
     &\qquad\qquad +C_{l}  \norm{\comi v^{4+\gamma}f}_{L_v^1}\norm{\comi v^{l}g}_{L_v^1}+C_{l}\norm{\comi v^{4+\gamma}g}_{L_v^1}\norm{\comi v^{l}f}_{L_v^1},
    \end{align*}
  with       $ \boldsymbol{\lambda}_ l, \boldsymbol{\omega}_{ l } $ defined as in \eqref{lamome}.
\end{lemma}

\begin{proof} To simplify notations, here and throughout paper we denote by $C_l$  a generic constant depending  on $l.$ 
We first recall the following result established in \cite[Lemma 2.5]{MR4201411}:    
  \begin{equation}\label{diffweight}
  \langle v'\rangle^{ l }=\langle v\rangle^{ l }\cos^{ l }{\frac{\theta}{2}}+\langle v_*\rangle^{ l }\sin^{ l }{\frac{\theta}{2}} + l \langle v\rangle^{ l -2}|v-v_*|(v\cdot \kappa^\perp)\cos^{ l -1}{\frac{\theta}{2}}\sin{\frac{\theta}{2}}+R_ l ,
\end{equation}
where  $\theta$ is defined by \eqref{angle}, and 
\begin{equation}
    \label{deperp}
  \kappa^\perp:=\frac{\sigma-(\sigma\cdot\kappa)\kappa}{\abs{\sigma-(\sigma\cdot\kappa)\kappa}} \  \textrm{ with } \
\kappa=\frac{v-v_*}{|v-v_*|},
\end{equation}
and the error term $R_ l $   in \eqref{diffweight} satisfies that, for  $l\geq 5$, 
\begin{equation}\label{error}
\begin{aligned}
|R_ l | & \leq C_l\langle v\rangle\langle v_*\rangle^{ l -1}  \sin^{ l -3}{\frac{\theta}{2}}+C_l\langle v\rangle^{ l -2} \langle v_*\rangle^2\sin^{2}{\frac{\theta}{2}}+C_l\langle v\rangle^{ l -4} \langle v_*\rangle^4\sin^{2}{\frac{\theta}{2}}.
\end{aligned}
\end{equation}
Interested readers may refer to \cite[Lemma 2.5]{MR4201411} for the proof of \eqref{diffweight} and \eqref{error}.    
Consequently,  using  \eqref{diffweight} and \eqref{error} as well as the fact that 
\begin{equation*}
\forall\ j\geq 2, \quad 	\int_{\mathbb S^2}b(\cos\theta) \sin^j{\frac{\theta}{2}} d\sigma=2\pi\int_{0}^{\frac{\pi}{2}} \sin\theta b(\cos\theta)  \sin^j{\frac{\theta}{2}} d\theta \leq C,
\end{equation*}
we obtain, for  $f,g\geq0$ and $l\geq 5,$
\begin{equation}
\label{i14}
\begin{aligned}
 &\int_{\mathbb{R}^6\times\mathbb{S}^2} |v-v_*|^{\gamma}b(\cos\theta)g_*f\big(\langle v'\rangle^{ l } -\langle v\rangle^{ l }\big)\,d\sigma\, dv_*\, dv \\
&\leq \int_{\mathbb{R}^6\times\mathbb{S}^2}  |v-v_*|^{\gamma}b(\cos\theta)g_*f\langle v\rangle^{ l }\big(\cos^{ l }{\frac{\theta}{2}}-1\big)d\sigma dv_* dv\\
&\quad+\int_{\mathbb{R}^6\times\mathbb{S}^2} |v-v_*|^{\gamma}b(\cos\theta)g_*f\langle v_*\rangle^{ l }\sin^{ l }{\frac{\theta}{2}}\,d\sigma\, dv_*\, dv\\
&\quad+ l \int_{\mathbb{R}^6\times\mathbb{S}^2} |v-v_*|^{\gamma+1}b(\cos\theta) g_* f   \langle v\rangle^{ l -2}(v\cdot \kappa^\perp)\cos^{ l -1}{\frac{\theta}{2}}\sin{\frac{\theta}{2}}\,d\sigma\, dv_*\, dv\\
&\quad+C_l\int_{\mathbb{R}^6\times\mathbb{S}^2} |v-v_*|^{\gamma} g_*f  \left(\langle v\rangle \langle v_*\rangle^{ l -1}+\langle v\rangle^{ l -2}\langle v_*\rangle^{2} +\langle v\rangle^{ l -4}\langle v_*\rangle^{4}\right)dv_*\,dv\\
&:=I_1+I_2+I_3+I_4.
\end{aligned}
\end{equation}
For the last term $I_4$ above, using  \eqref{elementaryinequality} and the estimate 
\begin{equation*}
  \forall\ \tilde\eps>0,\quad \langle v\rangle \langle v_*\rangle^{ l -1+\gamma}\leq \tilde\eps \langle v_*\rangle^{ l+\gamma}+\tilde\eps^{-\gamma} \langle v\rangle^{1+\gamma} \langle v_*\rangle^{ l -1}, 
\end{equation*} 
we compute, for any $\eps>0,$ 
\begin{equation}\label{sl}
\begin{aligned}
    &C_l\int_{\mathbb{R}^6} |v-v_*|^{\gamma} g_*f   \langle v\rangle \langle v_*\rangle^{ l -1} dv_*\,dv\\
    &\leq 2^\gamma C_l\int_{\mathbb{R}^6}  g_*f   \langle v\rangle \langle v_*\rangle^{ l -1+\gamma} dv_*\,dv+2^\gamma C_l\int_{\mathbb{R}^6}   g_*f   \langle v\rangle^{1+\gamma} \langle v_*\rangle^{ l -1} dv_*\,dv\\
    &\leq \eps \int_{\mathbb{R}^6}  g_*f   \langle v_*\rangle^{ l+  \gamma} dv_*\,dv+\Big[\big(2^\gamma C_l\big)^{1+\gamma} \eps^{-\gamma} +2^\gamma C_l\Big] \int_{\mathbb{R}^6}  g_*f   \langle v\rangle^{1+\gamma} \langle v_*\rangle^{ l -1} dv_*\,dv\\
    &\leq \eps \norm{f}_{L_v^1}\norm{\comi v^{l+\gamma}g}_{L_v^1}+C_{\eps,l}  \norm{\comi v^{1+\gamma}f}_{L_v^1}\norm{\comi v^{l}g}_{L_v^1}.
\end{aligned}
\end{equation}  
Similarly, for any  $\tilde\eps>0,$
\begin{multline*}
    C_l\int_{\mathbb{R}^6} |v-v_*|^{\gamma} g_*f  \left( \langle v\rangle^{ l -2}\langle v_*\rangle^{2} +\langle v\rangle^{ l -4}\langle v_*\rangle^{4}\right)dv_*\,dv\\
     \leq \tilde\eps \norm{g}_{L_v^1}\norm{\comi v^{l+\gamma}f}_{L_v^1}+C_{\tilde\eps,l}\norm{\comi v^{4+\gamma}g}_{L_v^1}\norm{\comi v^{l}f}_{L_v^1}.
\end{multline*}
Consequently, combining the two estimates above yields the upper bound of $I_4$ in \eqref{i14}. Namely,  for any $\eps,\tilde\eps>0,$ 
\begin{multline*}
     I_4\leq   \eps \norm{f}_{L_v^1}\norm{\comi v^{l+\gamma}g}_{L_v^1}+\tilde\eps \norm{g}_{L_v^1} \norm{\comi v^{l+\gamma}f}_{L_v^1}\\
     +C_{\eps,l}  \norm{\comi v^{4+\gamma}f}_{L_v^1}\norm{\comi v^{l}g}_{L_v^1}+C_{\tilde\eps,l}\norm{\comi v^{4+\gamma}g}_{L_v^1}\norm{\comi v^{l}f}_{L_v^1}.
\end{multline*} 
In particular, letting $\eps=\boldsymbol{\omega}_l$ and   $\tilde\eps=\frac{ \boldsymbol{\lambda}_ l }{4^{1+ \gamma}}$ with $\boldsymbol{\omega}_l, \boldsymbol{\lambda}_ l $  defined by  \eqref{lamome}, we obtain 
\begin{equation}\label{i4}
\begin{aligned}
    I_4  \leq &  \boldsymbol{\omega}_l \norm{f}_{L_v^1}\norm{\comi v^{l+\gamma}g}_{L_v^1}+ \frac{ \boldsymbol{\lambda}_ l }{4^{1+ \gamma}}\norm{g}_{L_v^1} \norm{\comi v^{l+\gamma}f}_{L_v^1}\\
     &+C_{l}  \norm{\comi v^{4+\gamma}f}_{L_v^1}\norm{\comi v^{l}g}_{L_v^1}+C_{l}\norm{\comi v^{4+\gamma}g}_{L_v^1}\norm{\comi v^{l}f}_{L_v^1}.
\end{aligned}
\end{equation}
To estimate $I_1$ in \eqref{i14}, using the definition of $\boldsymbol{\lambda}_ l $ in \eqref{lamome},  we compute  
\begin{equation}\label{i1+}
    \begin{aligned}
    I_1 &=\int_{\mathbb{R}^6}  |v-v_*|^{\gamma} g_*f\langle v\rangle^{ l }\Big(2\pi\int_0^{\frac{\pi}{2}} \sin\theta b(\cos\theta) \big(\cos^{ l }{\frac{\theta}{2}}-1\big)d\theta \Big) dv_*\,dv\\
     &=-\frac{ \boldsymbol{\lambda}_ l }{2^{ \gamma}} \int_{\mathbb{R}^6}  |v-v_*|^{\gamma} g_*f\langle v\rangle^{ l }  dv_*\,dv\\
     &\leq -\frac{ \boldsymbol{\lambda}_ l }{4^{ \gamma}}\norm{g}_{L_v^1}\norm{\langle v\rangle^{ l+\gamma } f}_{L_v^1}+C_l\norm{\comi v^{\gamma}g}_{L_v^1}\norm{\langle v\rangle^{ l} f}_{L_v^1},
    \end{aligned}
\end{equation}
the last inequality using \eqref{tlow}. 
Recalling  $\boldsymbol{\omega}_ l $ is defined by
\eqref{lamome}, we use \eqref{elementaryinequality} to estimate the term $I_2$ in \eqref{i14} as follows. 
\begin{equation}\label{i2}
\begin{split}
      I_2&=  \int_{\mathbb{R}^6} |v-v_*|^{\gamma}g_*f\langle v_*\rangle^{ l }\Big(2\pi\int_0^{\frac{\pi}{2}} \sin\theta b(\cos\theta)\sin^{ l }{\frac{\theta}{2}}d\theta\Big) dv_*\,dv\\
      &= 2^{-\gamma}\boldsymbol{\omega}_ l  \int_{\mathbb{R}^6} |v-v_*|^{\gamma}g_*f\langle v_*\rangle^{ l }  dv_*\,dv\\
      &\leq   \boldsymbol{\omega}_ l  \norm{f}_{L_v^1}\norm{\comi v^{ l +\gamma}g}_{L_v^1}+  \boldsymbol{\omega}_ l  \norm{\comi v^\gamma f}_{L_v^1}\norm{\comi v^{ l }g}_{L_v^1}.
\end{split}
\end{equation}
It remains to handle $I_3$ in \eqref{i14}. We claim that, for any function $H(\cos\theta,\sin\theta)$ of $\cos\theta$ and $\sin\theta,$    
\begin{equation}\label{sdf}
	\int_{\mathbb S^2} H(\cos\theta,\sin\theta)  (v\cdot \kappa^\perp)\,d\sigma=0,
\end{equation} 
recalling $\cos\theta=\kappa\cdot\sigma$ with $\kappa=\frac{v-v_*}{\abs{v-v_*}}$ and $\kappa^\perp$ is defined in \eqref{deperp}.  
To verify this,  we use \eqref{deperp} to write 
\begin{equation}\label{orth}
\begin{aligned}
  \sigma&=(\sigma\cdot\kappa)\kappa+\kappa^\perp \abs{\sigma-(\sigma\cdot\kappa)\kappa}\\
  &= \kappa\cos\theta+\kappa^\perp\sin\theta=\kappa\cos\theta+\big(\underbrace{\kappa_1 \cos\phi+\kappa_2 \sin\phi}_{=\kappa^\perp}\big)\sin\theta,   
\end{aligned}
\end{equation}
where $\phi\in [0,2\pi]$ and $\kappa_1,\kappa_2\in\mathbb S^2$ such that $\{\kappa,\kappa_1,\kappa_2\}$  forms an orthonormal basis of $\mathbb R^3$. In the polar coordinates given by \eqref{orth}, the integral in \eqref{sdf} can be expressed as 
\begin{align*}
	&\int_{\mathbb S^2} H(\cos\theta,\sin\theta)  (v\cdot \kappa^\perp)\,d\sigma\\
	&= 	\int_0^{\pi}\sin\theta H(\cos\theta, \sin\theta)   \Big(\int_0^{2\pi} \big(v\cdot \kappa_1 \cos\phi+v\cdot\kappa_2\sin\phi\big)d\phi \Big)d\theta=0,
\end{align*}
where the last identity holds because $\kappa_1,\kappa_2$ are independent of $\phi$ and $\int_0^{2\pi}   \cos\phi \,d\phi=\int_0^{2\pi}\sin\phi\, d\phi=0.$ 
This gives \eqref{sdf}, and thus 
\begin{equation*}
    I_3= l \int_{\mathbb{R}^6\times\mathbb{S}^2} |v-v_*|^{\gamma+1}b(\cos\theta) g_* f   \langle v\rangle^{ l -2} (v\cdot \kappa^\perp)\cos^{ l -1}{\frac{\theta}{2}}\sin{\frac{\theta}{2}}\,d\sigma\, dv_*\, dv=0.
\end{equation*}
Substituting this and   the estimates \eqref{i4}, \eqref{i1+} and  \eqref{i2}   into \eqref{i14}, we  get the desired estimate in Lemma \ref{lem:cru}, completing the proof. 
\end{proof}

\begin{proof}[Completing the proof of Proposition \ref{momentQ}]
In the proof, we always assume the functions $f$ and $g$ are non-negative. 
 Using \eqref{moest} and Lemma \ref{lem:cru}, we conclude that
 \begin{equation*}
 \begin{aligned}
  \int_{\mathbb{R}^3}  Q(g, f)  \langle v\rangle^{ l } dv    \leq &  -\frac{\boldsymbol{\lambda}_l}{2^{1+2\gamma}} \norm{g}_{L_v^1}\norm{\langle v\rangle^{ l+\gamma } f}_{L_v^1} +2\boldsymbol{\omega}_ l  \norm{f}_{L_v^1}\norm{\comi v^{ l +\gamma}g}_{L_v^1}  \\
     &  +C_{l}  \norm{\comi v^{4+\gamma}f}_{L_v^1}\norm{\comi v^{l}g}_{L_v^1}+C_{l}\norm{\comi v^{4+\gamma}g}_{L_v^1}\norm{\comi v^{l}f}_{L_v^1}. 
        \end{aligned}
 \end{equation*}
Moreover,  since $g\geq 0$ satisfies condition $\boldsymbol{(H)}$, then 
  \begin{equation*}
  \frac{m_0}{2}\leq  \norm{g}_{L_v^1}.
  \end{equation*}
 Combining the above estimates  yields   
 \begin{multline*}
 \int_{\mathbb{R}^3}  Q(g, f)  \langle v\rangle^{ l } dv  
 \leq -\frac{m_0\boldsymbol{\lambda}_ l }{4^{\gamma  +1}}\norm{\comi v^{ l +\gamma}f}_{L_v^1}
        +2  \boldsymbol{\omega}_ l  \norm{  f}_{L_v^1}\norm{\comi v^{ l +\gamma}g}_{L_v^1}\\ +C_{l}\|\comi v^{4+\gamma}g\|_{L^1_v}\|\langle v\rangle^{ l}f\|_{L_v^1}+C_{ l}\|\langle v\rangle^{4+\gamma}f\|_{L^1_v} \norm{\comi v^{ l }g}_{L_v^1}. 
\end{multline*}
 Integrating the above inequality over $x\in\mathbb T^3$, we obtain the desired estimate in Proposition \ref{momentQ}. The proof is completed.
\end{proof}


\subsection{$L^2$-estimates for the collision operator}

Lower and upper bounds for the collision operator in the $L_v^2$-setting are given in Subsection \ref{subsec:tri}. The present part aims to derive  their counterparts in the weighted $L_v^2$-space, and the main result is as follows. 

\begin{proposition}\label{upbounded}
For any $ l \geq \frac{13}{2}+\gamma$, there exists a constant  $C_{ l }>0$ depending on $ l $, such that  
\begin{equation}
	\label{upbound-v}
\begin{aligned} 
 \big|\big ( \langle v\rangle^{ l }Q(g,f), h\big)_{L^2_{v}}\big| \leq & C_l \|\langle v\rangle^{6+\gamma}g\|_{L^2_v}\normm{\langle v\rangle^{l}f}\left(\normm{h}+\|\langle v\rangle^{ s+\frac{\gamma}{2}}h\|_{L^2_v}\right)\\
& +C_{l}\|\langle v\rangle^{6+\gamma}f\|_{L^2_v}\|\langle v\rangle^{l+\frac{\gamma}{2}}g\|_{L^2_v}\|\langle v\rangle^{\frac{\gamma}{2}}h\|_{L^2_v}.
\end{aligned}
\end{equation}
Moreover, if $g\geq 0$, then 
\begin{equation}\label{coer-v}
\begin{split}
  \big(\langle v\rangle^{ l } Q(g,f), \langle v\rangle^{ l } f\big)_{L^2_{v}} &\leq   -\frac{c_0}{2}\normm{ \langle v\rangle^{ l }f}^2
+ C_0\|\langle v\rangle^{ l + s+\frac{\gamma}{2}}f\|^2_{L^2_v}\\
&\   +C_{ l } \|\langle v\rangle^{6+\gamma}f\|_{L^2_v}^2\|\langle v\rangle^{ l +\frac{\gamma}{2}}g\|_{L^2_v}^2+C_{ l } \|\langle v\rangle^{6+\gamma} g\|^2_{L^2_v}\| \langle v\rangle^{ l +\frac{\gamma}{2}}f\|^2_{L^2_v},
\end{split}
\end{equation}
and 
\begin{equation}\label{+coercivity-com}
	\begin{aligned}
	&\big(\langle v\rangle^{ l } Q(g,f),\  \langle v\rangle^{ l } f\big)_{L^2_{v}} \\
	& \leq  -\frac{c_0}{2}\norm{\comi{D_v}^s\comi v^{l+\frac{\gamma}{2}}  f}_{L_v^2}^2+ 2C_0 \norm{\comi v^{l+\frac{\gamma}{2}}f}_{L_v^2}^2+2^{1+\gamma} \big(\boldsymbol{\lambda}_l+\boldsymbol{A}_{\gamma,s}\big) \| g\|_{L^1_v}   \|\langle v\rangle^{l+\frac{\gamma}{2}}f\|_{L^2_v}^2\\
		&\quad  +C_l \|\langle v\rangle^{6+\gamma}f\|_{L^2_v}^2\|\langle v\rangle^{l+\frac{\gamma}{2}}g\|_{L^2_v}^2+C_{ l }\Big(1+\|\langle v\rangle^{6+\gamma}g\|_{L^2_v}^2\Big)   \|\langle v\rangle^{l}f\|_{L^2_v}^2,	
	\end{aligned}
\end{equation}
where $c_0$ is the constant given  in \eqref{lower+2}. 
\end{proposition}

 The proof of Proposition \ref{upbounded} relies crucially on commutator estimates between weight functions and the collision operator. Although there has been extensive literature  on such estimates, they are not  applicable in our setting. To this end, we adapt the  argument in  \cite{MR4201411} to establish the required estimates.

\begin{lemma}[Commutator estimate]\label{commutator}
For any $ l \geq \frac{13}{2}+\gamma$,   there exists  a constant $C_{ l }>0$ depending  on $ l $, such that
 \begin{align*}
 &\Big | (\langle v\rangle^ l  Q(g,f), h)_{L^2_v}-( Q(g,\langle v\rangle^ l  f), h)_{L^2_v}\Big|\\
 	&\leq  	 C_l \|\langle v\rangle^{6+\gamma}g\|_{L^2_v}\|\langle v\rangle^{\frac{\gamma}{2}}h\|_{L^2_v}\|\comi{D_v}^s \langle v\rangle^{l+\frac{\gamma}{2}}f\|_{L^2_v}+	2 \boldsymbol{\omega}_{l-2-\gamma} \|f\|_{L^1_v} \|\langle v\rangle^{l+\frac{\gamma}{2}}g\|_{L^2_v} \|\langle v\rangle^{ \frac{\gamma}{2}}h\|_{L^2_v}\\
 		&\quad +C_l  \|\langle v\rangle^{ 6+\gamma}f\|_{L^2_v}\|\langle v\rangle^{l+\frac{\gamma}{2}}g\|_{L^2_v}  \| h\|_{L^2_v}+C_l  \|\langle v\rangle^{ 6+\gamma}f\|_{L^2_v}\|\langle v\rangle^{l}g\|_{L^2_v}  \| \comi v^{\frac{\gamma}{2}}h\|_{L^2_v},
 \end{align*}
 where $\boldsymbol{\lambda}_{l},\boldsymbol{  \omega}_{l}$ with $l\geq 2$  are the constants defined in \eqref{lamome}. As a result, 
  \begin{equation}\label{sa}
    \begin{aligned}
        &|(\langle v\rangle^ l  Q(g,f), h)_{L^2_v}-( Q(g,\langle v\rangle^ l  f), h)_{L^2_v}|\\
        &\leq  	C_l \|\langle v\rangle^{6+\gamma}f\|_{L^2_v} \|\langle v\rangle^{l+\frac{\gamma}{2}}g\|_{L^2_v}    \|\langle v\rangle^{ \frac{\gamma}{2}}h\|_{L^2_v} +C_{ l } \norm{\comi v^{6+\gamma} g}_{L^2_v}  \norm{\comi v^{ \frac{\gamma}{2}}h}_{L_v^2}\normm{\comi v^l f}.
    \end{aligned}
\end{equation}
\end{lemma}
\begin{proof} 
It suffices to prove the first assertion,   as the second   follows immediately from it by the inequality that   
\begin{equation*}
 \|\comi{D_v}^s\langle v\rangle^{l+\frac{\gamma}{2}} f\|_{L^2_v}\leq  \normm{\comi v^l f}. 
\end{equation*} 
Similar to \eqref{moest}, we use 
  the pre- and post-collisional  change of variables \eqref{prepost}   to write 
\begin{align*}
&(\langle v\rangle^ l  Q(g,f), h)_{L^2_v}-( Q(g,\langle v\rangle^ l  f), h)_{L^2_v}\\
 &=\int_{\mathbb{R}^6\times\mathbb{S}^2}|v-v_*|^{\gamma}b(\cos\theta) (g_*'f'-g_*f)\langle v\rangle^{l} h \,d\sigma\, dv_*\, dv\\
 &\quad-\int_{\mathbb{R}^6\times\mathbb{S}^2}  |v-v_*|^{\gamma}b(\cos\theta) \big(g_*'f'\langle v'\rangle^ l-g_*f\langle v\rangle^{l}\big )  h \,d\sigma\, dv_*\, dv \\
 &=\int_{\mathbb{R}^6\times\mathbb{S}^2}  |v-v_*|^{\gamma}b(\cos\theta) g_*'f' \big(\langle v\rangle^{l} - \langle v'\rangle^ l\big)h \,d\sigma\, dv_*\, dv \\
&=\int_{\mathbb{R}^6\times\mathbb{S}^2}  |v-v_*|^{\gamma}b(\cos\theta)g_*f\big(\langle v'\rangle^ l -\langle v\rangle^ l \big) h'\,d\sigma\, dv_*\, dv.
\end{align*}
Moreover,  substituting  \eqref{diffweight}  into the above integral and then using   \eqref{error},  one has     
\begin{equation}\label{j}
\begin{aligned}
&(\langle v\rangle^ l  Q(g,f), h)_{L^2_v}-( Q(g,\langle v\rangle^ l  f), h)_{L^2_v}\\
&\leq
\int_{\mathbb{R}^6\times\mathbb{S}^2} |v-v_*|^{\gamma}b(\cos\theta)\Big(\cos^{ l }\frac{\theta}{2}-1\Big) \langle v\rangle^{ l } g_* f h' \,d\sigma\, dv_*\, dv\\
&\quad  +\int_{\mathbb{R}^6\times\mathbb{S}^2} |v-v_*|^{\gamma}b(\cos\theta)\sin^{ l }\frac{\theta}{2} \langle v_*\rangle^{ l } g_* f h' \,d\sigma\, dv_*\, dv\\
&\quad  + l \int_{\mathbb{R}^6\times\mathbb{S}^2} |v-v_*|^{\gamma+1}b(\cos\theta)\langle v\rangle^{ l -2}(v\cdot \kappa^\perp) \sin\frac{\theta}{2}\cos^{l -1}\frac{\theta}{2}\,g_*f h' \,d \sigma\, dv_*\, dv\\
&\quad  +C_l\int_{\mathbb{R}^6\times\mathbb{S}^2}  |v-v_*|^{\gamma} b(\cos\theta) \langle v\rangle \langle v_*\rangle^{ l -1}\sin^{l-3}{\frac{\theta}{2}}  \,  |g_*  f h'| \, d\sigma\, dv_*\, dv\\
&\quad  +C_l\int_{\mathbb{R}^6\times\mathbb{S}^2}  |v-v_*|^{\gamma} b(\cos\theta)\Big( \langle v\rangle^{l -2}\langle v_*\rangle^{2}+\langle v\rangle^{l -4}\langle v_*\rangle^{4} \Big) \sin^2{\frac{\theta}{2}} \, |g_*  f h'| \, d \sigma \,dv_*\, dv\\
&:=J_{1}+J_{2}+J_{3}+J_{4}+J_5.
\end{aligned}
\end{equation}
We will proceed to estimate the right-hand terms. 

{\it Step 1 (Estimate on $J_1$).} It follows from  Cauchy-Schwarz inequality that
\begin{multline}\label{j11}
	|J_1| \leq  \Big(\int_{\mathbb{R}^6\times\mathbb{S}^2} |v-v_*|^{\gamma}b(\cos\theta)\Big(1-\cos^{ l }\frac{\theta}{2}\Big) \langle v\rangle^{ 2l }       |f|^2 |g_*| \,d\sigma\, dv_*\, dv \Big)^{\frac12} \\
	\times\Big(\int_{\mathbb{R}^6\times\mathbb{S}^2} |v-v_*|^{\gamma}b(\cos\theta)\Big(1-\cos^{ l }\frac{\theta}{2}\Big)        |h'|^2 |g_*|  \,d\sigma\, dv_*\, dv\Big)^{\frac12}.
\end{multline}
For the first integral on the right-hand side, using the second estimate  in \eqref{elementaryinequality} as well as the definition of $\boldsymbol{\lambda}_l$  in  \eqref{lamome}  gives
\begin{align*}
	&  \int_{\mathbb{R}^6\times\mathbb{S}^2} |v-v_*|^{\gamma}b(\cos\theta)\Big(1-\cos^{ l }\frac{\theta}{2}\Big) \langle v\rangle^{ 2l }     |f|^2|g_*|  \,d\sigma\, dv_*\, dv\\
    &=2^{-\gamma}\boldsymbol{\lambda}_l \int_{\mathbb{R}^6} |v-v_*|^{\gamma}  \langle v\rangle^{ 2l }    |f|^2|g_*|  \, dv_*\, dv \leq   \boldsymbol{\lambda}_l  \| g\|_{L^1_v} \|\langle v\rangle^{l+\frac{\gamma}{2}}f\|_{L^2_v}^2+  \boldsymbol{\lambda}_l  \|\langle v\rangle^{\gamma}g\|_{L^1_v} \|\langle v\rangle^{l}f\|_{L^2_v}^2.
\end{align*}
For the second integral  in \eqref{j11}, we use   formula  \eqref{rech} in Appendix \ref{Appendix}  to  obtain 
  \begin{align*}
  	&  \int_{\mathbb{R}^6\times\mathbb{S}^2} |v-v_*|^{\gamma}b(\cos\theta)\Big(1-\cos^{ l }\frac{\theta}{2} \Big)      |h'|^2 |g_*|\,d\sigma\, dv_*\, dv  \\
  	&=  \int_{\mathbb{R}^6\times\mathbb{S}^2} |v-v_*|^{\gamma}b(\cos\theta)\Big(1-\cos^{ l }\frac{\theta}{2} \Big)\frac{1}{\cos^{3+\gamma}\frac{\theta}{2}}       |h|^2 |g_*| \,d\sigma\, dv_*\, dv\\
  	&  \leq  2^{ \frac{3+\gamma}{2}} \int_{\mathbb{R}^6\times\mathbb{S}^2} |v-v_*|^{\gamma}b(\cos\theta)\Big(1-\cos^{ l }\frac{\theta}{2} \Big)      |h|^2|g_*| \,d\sigma\, dv_*\, dv\\
  	&   \leq 2^{ \frac{3+\gamma}{2}}\boldsymbol{\lambda}_l\| g\|_{L^1_v}  \|\langle v\rangle^{ \frac{\gamma}{2}}h\|_{L^2_v}^{2}+2^{ \frac{3+\gamma}{2}}\boldsymbol{\lambda}_l\|\langle v\rangle^{\gamma}g\|_{L^1_v}  \| h\|_{L^2_v}^{2},
  \end{align*}
  where the first inequality holds because $ \cos{\frac{\theta}{2}} \geq 2^{-\frac12}$  on the support of the collision kernel (i.e.  $\theta\in[0,\frac{\pi}{2}]$),    the last line follows from   the definition of $\boldsymbol{\lambda}_l$ in  \eqref{lamome}  together with \eqref{elementaryinequality}.    Substituting these estimates into \eqref{j11} yields
  \begin{multline*}
 	|J_1| \leq  2^{ \frac{3+\gamma}{4}} \boldsymbol{\lambda}_l \Big(  \| g\|_{L^1_v} \|\langle v\rangle^{l+\frac{\gamma}{2}}f\|_{L^2_v}^2+ \|\langle v\rangle^{\gamma}g\|_{L^1_v} \|\langle v\rangle^{l}f\|_{L^2_v}^2\Big)^{\frac12}\\ 
 	\times  \Big(\| g\|_{L^1_v}  \|\langle v\rangle^{ \frac{\gamma}{2}}h\|_{L^2_v}^{2}+\|\langle v\rangle^{\gamma}g\|_{L^1_v}  \| h\|_{L^2_v}^{2}\Big)^{\frac12}.
 \end{multline*}
 This implies 
 \begin{equation}\label{crj1}
 	\begin{aligned}
 		|J_1| 	\leq &2^{1+\gamma} \boldsymbol{\lambda}_l \| g\|_{L^1_v} \|\langle v\rangle^{l+\frac{\gamma}{2}}f\|_{L^2_v}  \norm{\comi v^{\frac{\gamma}{2}}h}_{L_v^2}\\
 		&+C_l \|\langle v\rangle^{\gamma}g\|_{L^1_v}\|\langle v\rangle^{l+\frac{\gamma}{2}}f\|_{L^2_v}  \norm{ h}_{L_v^2}+C_l \|\langle v\rangle^{\gamma}g\|_{L^1_v}\|\langle v\rangle^{l}f\|_{L^2_v}  \norm{\comi v^{\frac{\gamma}{2}} h}_{L_v^2},
 	\end{aligned}
 \end{equation}  
 and thus
 \begin{equation}
 	\label{j1}
 	\begin{aligned}
 	|J_1| & \leq 	C_l \|\langle v\rangle^{\gamma}g\|_{L^1_v}\|\langle v\rangle^{l+\frac{\gamma}{2}}f\|_{L^2_v}  \norm{\comi v^{\frac{\gamma}{2}}  h}_{L_v^2}\\
 	&\leq 	C_l \|\langle v\rangle^{6+\gamma}g\|_{L^2_v}\|\comi{D_v}^s\langle v\rangle^{l+\frac{\gamma}{2}}f\|_{L^2_v}  \norm{\comi v^{\frac{\gamma}{2}}  h}_{L_v^2}. 
   \end{aligned}
 \end{equation}
 	
 {\it Step 2 (Estimate on $J_2$).}
 For  $J_2$ in \eqref{j},  we use Cauchy-Schwarz inequality to obtain   \begin{equation}\label{j2+++}
 \begin{aligned}
 	|J_2| =&\Big|\int_{\mathbb{R}^6\times\mathbb{S}^2} |v-v_*|^{\gamma}b(\cos\theta)\sin^{ l }\frac{\theta}{2} \langle v_*\rangle^{ l } g_* f h' \,d\sigma\, dv_*\, dv\Big|\\
 	\leq &  \Big(\int_{\mathbb{R}^6\times\mathbb{S}^2} |v-v_*|^{\gamma}b(\cos\theta)\sin^{ l-\frac{3+\gamma}{2} }\frac{\theta}{2}\,\langle v_*\rangle^{ 2l } g_*^2 |f|  \,d\sigma\, dv_*\, dv \Big)^{\frac12}\\
 	&\qquad \qquad\times\Big(\int_{\mathbb{R}^6\times\mathbb{S}^2} |v-v_*|^{\gamma}b(\cos\theta)\sin^{ l +\frac{3+\gamma}{2}}\frac{\theta}{2} \, |f| |h'|^2 \,d\sigma\, dv_*\, dv\Big)^{\frac12}.
 	\end{aligned}
 \end{equation}
 To estimate the last factor in \eqref{j2+++}, we use 
  \eqref{inch} in Appendix \ref{Appendix}  to write 
  that  
  \begin{align*}
 &\int_{\mathbb{R}^3\times\mathbb{S}^2} |v-v_*|^{\gamma}b(\cos\theta)\sin^{ l+\frac{3+\gamma}{2} }\frac{\theta}{2}    |h'|^2 d\sigma   dv_* \\ 
 & = \int_{\mathbb{R}^3\times\mathbb{S}^2} |v-v_*|^{\gamma}b(\cos\theta)\frac{\sin^{ l+\frac{3+\gamma}{2} }\frac{\theta}{2}}{\sin^{ 3+\gamma }\frac{\theta}{2}}\, |h_*|^2 d\sigma  dv_* \\
 &=\int_{\mathbb{R}^3\times\mathbb{S}^2} |v-v_*|^{\gamma}b(\cos\theta) \sin^{ l-\frac{3+\gamma}{2}}\frac{\theta}{2}\, |h_*|^2 d\sigma  dv_*,
 \end{align*}
 and thus,
 recalling  $\boldsymbol{\omega}_l$ defined as in \eqref{lamome} is decreasing in $l$ and  using  \eqref{elementaryinequality},  
 \begin{align*}
 	&\int_{\mathbb{R}^6\times\mathbb{S}^2} |v-v_*|^{\gamma}b(\cos\theta)\sin^{ l +\frac{3+\gamma}{2}}\frac{\theta}{2} \, |f| |h'|^2 \,d\sigma\, dv_*\, dv\\
 	&= \int_{\mathbb{R}^6\times\mathbb{S}^2} |v-v_*|^{\gamma}b(\cos\theta)\sin^{ l-\frac{3+\gamma}{2} }\frac{\theta}{2}\,   |h_*|^2\, |f|  \,d\sigma dv_*dv\\ 
 	& \leq \frac{\boldsymbol{\omega}_{l-2-\gamma} }{2^\gamma} \int_{\mathbb{R}^6} |v-v_*|^{\gamma}   |h_*|^2\, |f|  \, dv_* dv  \leq   \boldsymbol{\omega}_{l-2-\gamma} \Big(\|f\|_{L^1_v} \|\langle v\rangle^{ \frac{\gamma}{2}}h\|_{L^2_v}^2  +  \|\langle v\rangle^{ \gamma}f\|_{L^1_v} \| h\|_{L^2_v}^2\Big).
 \end{align*}
 Similarly, 
 \begin{multline*}
 	\int_{\mathbb{R}^6\times\mathbb{S}^2} |v-v_*|^{\gamma}b(\cos\theta)\sin^{ l-\frac{3+\gamma}{2} }\frac{\theta}{2}\,\langle v_*\rangle^{ 2l } g_*^2 |f|  \,d\sigma\, dv_*\, dv \\
 	\leq     \boldsymbol{\omega}_{l-2-\gamma} \Big(\|f\|_{L^1_v} \|\langle v\rangle^{l+\frac{\gamma}{2}}g\|_{L^2_v}^2 + \|\langle v\rangle^{ \gamma}f\|_{L^1_v}\|\langle v\rangle^{l}g\|_{L^2_v}^2\Big) 
 \end{multline*}
 Substituting  the two estimates above  into  \eqref{j2+++} yields 
 \begin{equation}\label{j2}
 \begin{aligned}
 	|J_2| \leq &   \boldsymbol{\omega}_{l-2-\gamma}   \Big(\|f\|_{L^1_v} \|\langle v\rangle^{l+\frac{\gamma}{2}}g\|_{L^2_v}^2 + \|\langle v\rangle^{ \gamma}f\|_{L^1_v}\|\langle v\rangle^{l}g\|_{L^2_v}^2\Big)^{\frac12}\\
 	&\qquad\qquad\qquad\times \Big(\|f\|_{L^1_v} \|\langle v\rangle^{ \frac{\gamma}{2}}h\|_{L^2_v}^2  +  \|\langle v\rangle^{ \gamma}f\|_{L^1_v} \| h\|_{L^2_v}^2\Big)^{\frac12}\\
 	\leq &  \boldsymbol{\omega}_{l-2-\gamma} \|f\|_{L^1_v} \|\langle v\rangle^{l+\frac{\gamma}{2}}g\|_{L^2_v} \|\langle v\rangle^{ \frac{\gamma}{2}}h\|_{L^2_v}\\
 	&+C_l  \|\langle v\rangle^{ \gamma}f\|_{L^1_v}\|\langle v\rangle^{l+\frac{\gamma}{2}}g\|_{L^2_v}  \| h\|_{L^2_v}+C_l  \|\langle v\rangle^{ \gamma}f\|_{L^1_v}\|\langle v\rangle^{l}g\|_{L^2_v}  \| \comi v^{\frac{\gamma}{2}}h\|_{L^2_v}.
 	\end{aligned}
 \end{equation}
 
 {\it Step 3 (Estimate on $J_4$).}
 For the term $J_4$ in \eqref{j}, we follow the argument in \eqref{sl} to compute, for any $\eps>0,$
 \begin{equation}\label{j4142}
 	\begin{aligned}
 	J_{4} &=C_l\int_{\mathbb{R}^6\times\mathbb{S}^2}  |v-v_*|^{\gamma} b(\cos\theta) \sin^{l-3}{\frac{\theta}{2}}   \langle v\rangle \langle v_*\rangle^{ l -1}   |g_*  f h'| \,d\sigma\, dv_*\, dv\\
 	&\leq 2^\gamma C_l\int_{\mathbb{R}^6\times\mathbb{S}^2}  b(\cos\theta) \sin^{l-3}{\frac{\theta}{2}}   \langle v\rangle \langle v_*\rangle^{ l -1 +\gamma}   |g_*  f h'| \,d\sigma\, dv_*\, dv\\
 	&\quad + 2^\gamma C_l\int_{\mathbb{R}^6\times\mathbb{S}^2}  b(\cos\theta) \sin^{l-3}{\frac{\theta}{2}}   \langle v\rangle^{1+\gamma} \langle v_*\rangle^{ l -1 }   |g_*  f h'| \,d\sigma\, dv_*\, dv\\
 	&\leq \eps \int_{\mathbb{R}^6\times\mathbb{S}^2}  b(\cos\theta) \sin^{l-3}{\frac{\theta}{2}}   \langle v_*\rangle^{ l   +\gamma}   |g_*  f h'| \,d\sigma\, dv_*\, dv\\
 	&\quad +C_{\eps,l} \int_{\mathbb{R}^6\times\mathbb{S}^2}  b(\cos\theta) \sin^{l-3}{\frac{\theta}{2}}   \langle v\rangle^{1+\gamma} \langle v_*\rangle^{ l -1 }   |g_*  f h'| \,d\sigma\, dv_*\, dv\\	
 &:=J_{4,1}+J_{4,2}.
 	\end{aligned}
 \end{equation}
 Following a similar argument as in the proof of \eqref{j2} we have,  for $l\geq  \frac{13}{2}+\gamma,$
\begin{align*}
  J_{4,1}  \leq & \eps \left(\int_{\mathbb{R}^6\times\mathbb{S}^2}  \frac{b(\cos\theta)}{\sin^{\frac{3}{2}}\frac{\theta}{2}}\sin^{l-3}\frac{\theta}{2}\, |g _* |^2 \langle v_*\rangle^{2l+2\gamma}   |f|\,d\sigma\, dv_*\, dv\right)^{\frac{1}{2}}\\
&\qquad\times \left(\int_{\mathbb{R}^6\times\mathbb{S}^2}   b(\cos\theta) \sin^{\frac{3}{2}}\frac{\theta}{2}\,\sin^{l-3}\frac{\theta}{2} \,  |f|\,|h'|^2\,d\sigma\, dv_*\, dv\right)^{\frac{1}{2}}\\
 \leq & \eps C_l \| f\|_{L^1_v}\|\langle v\rangle^{l+\frac{\gamma}{2}}g\|_{L^2_v}\| h\|_{L^2_v},
\end{align*} 
where the last line follows from \eqref{inch} in Appendix \ref{Appendix}. 
Similarly, 
\begin{align*}
 J_{4,2} &\leq C_{\eps, l} \left(\int_{\mathbb{R}^6\times\mathbb{S}^2}  \frac{b(\cos\theta)}{\sin^{\frac{3}{2}}\frac{\theta}{2}}\sin^{l-3}\frac{\theta}{2}\, |g _* |^2 \langle v_*\rangle^{2l-2}  \comi v^{1+\gamma}  |f|\,d\sigma\, dv_*\, dv\right)^{\frac{1}{2}}\\
&\qquad\times \left(\int_{\mathbb{R}^6\times\mathbb{S}^2}   b(\cos\theta) \sin^{\frac{3}{2}}\frac{\theta}{2}\,\sin^{l-3}\frac{\theta}{2} \, \comi v^{1+\gamma} |f|\,|h'|^2\,d\sigma\, dv_*\, dv\right)^{\frac{1}{2}}\\
&\leq C_{\eps,l} \|\langle v\rangle^{1+\gamma}f\|_{L^1_v}\|\langle v\rangle^{l}g\|_{L^2_v}\| h\|_{L^2_v}.
\end{align*}
Combining these estimates yields, for any $\eps>0,$ 
\begin{equation*}
	J_{4}\leq J_{4,1}+J_{4,2} \leq \eps  \| f\|_{L^1_v}\|\langle v\rangle^{l+\frac{\gamma}{2}}g\|_{L^2_v}\| h\|_{L^2_v}+C_{\eps,l} \|\langle v\rangle^{1+\gamma}f\|_{L^1_v}\|\langle v\rangle^{l}g\|_{L^2_v}\| h\|_{L^2_v},
\end{equation*}
and thus, letting $\eps=\boldsymbol{\omega}_{l-2-\gamma}$ in the above inequality, 
\begin{equation}\label{j4}
	J_4\leq   \boldsymbol{\omega}_{l-2-\gamma} \|f\|_{L^1_v} \|\langle v\rangle^{l+\frac{\gamma}{2}}g\|_{L^2_v} \|\langle v\rangle^{ \frac{\gamma}{2}}h\|_{L^2_v}+C_l\|\langle v\rangle^{1+\gamma}f\|_{L^1_v}\|\langle v\rangle^{l}g\|_{L^2_v}\| h\|_{L^2_v}.
\end{equation}

{\it Step 4 (Estimate on $J_5$).} For the term $J_5$ in \eqref{j}, we split it as 
\begin{equation*}
\begin{aligned}
    J_{5}=&C_l\int_{\mathbb{R}^6\times\mathbb{S}^2}  |v-v_*|^{\gamma} b(\cos\theta) \langle v\rangle^{l -2} \langle v_*\rangle^{2}  \sin^2{\frac{\theta}{2}} \, |g_*  f h'| \, d \sigma \,dv_*\, dv \\
    &+C_l\int_{\mathbb{R}^6\times\mathbb{S}^2}  |v-v_*|^{\gamma} b(\cos\theta) \langle v\rangle^{l -4} \langle v_*\rangle^{4}  \sin^2{\frac{\theta}{2}} \, |g_*  f h'| \, d \sigma \,dv_*\, dv \\
    :=&J_{5,1}+J_{5,2}.
\end{aligned}
\end{equation*}
Following the argument in \eqref{j4142} we have, for any $\eps>0,$
\begin{equation*}
    \begin{aligned}
       J_{5,1}\leq& \eps \int_{\mathbb{R}^6\times\mathbb{S}^2}    b(\cos\theta) \langle v\rangle^{l +\gamma}   \sin^2{\frac{\theta}{2}} \, |g_*  f h'| \, d \sigma \,dv_*\, dv\\
       &+C_{\eps,l}\int_{\mathbb{R}^6\times\mathbb{S}^2}    b(\cos\theta) \langle v\rangle^{l -2} \langle v_*\rangle^{2+\gamma}  \sin^2{\frac{\theta}{2}} \, |g_*  f h'| \, d \sigma \,dv_*\, dv.
    \end{aligned}
\end{equation*}
Moreover, using \eqref{rech} in Appendix \ref{Appendix}, we compute 
\begin{align*}
 & \int_{\mathbb{R}^6\times\mathbb{S}^2}    b(\cos\theta) \langle v\rangle^{l +\gamma}   \sin^2{\frac{\theta}{2}} \, |g_*  f h'| \, d \sigma \,dv_*\, dv\\
 &\leq    \left(\int_{\mathbb{R}^6\times\mathbb{S}^2}    b(\cos\theta)\sin^2\frac{\theta}{2} \, |g_*|   \langle v\rangle^{2l+2\gamma}|f|^2\,d\sigma  dv_*  dv\right)^{\frac{1}{2}}  \\
 &\ \ \ \ \ \ \ \ \ \ \  \times \left(\int_{\mathbb{R}^6\times\mathbb{S}^2}  b(\cos\theta)\sin^2\frac{\theta}{2}\, | g_*|   |h'|^2 \,d \sigma  dv_*  dv \right)^{\frac{1}{2}}\\
&\leq    C_{l}  \| g\|_{L^1_v}\|\langle v\rangle^{l+\frac{\gamma}{2}}f\|_{L^2_v}\| h\|_{L^2_v}.
\end{align*}
Similarly, 
\begin{align*}
 \int_{\mathbb{R}^6\times\mathbb{S}^2}    b(\cos\theta) \langle v\rangle^{l -2} \langle v_*\rangle^{2+\gamma}  \sin^2{\frac{\theta}{2}} \, |g_*  f h'| \, d \sigma \,dv_*\, dv  
 \leq   C_{l}  \|\langle v\rangle^{2+\gamma}g\|_{L^1_v}\|\langle v\rangle^{l }f\|_{L^2_v}\|\langle v\rangle^{\frac{\gamma}{2}}h\|_{L^2_v}.
\end{align*}
As a result, combining these estimates yields, for any $\eps>0,$ 
\begin{equation}\label{j51}
    J_{5,1}\leq  \eps  \| g\|_{L^1_v}\|\langle v\rangle^{l+\frac{\gamma}{2}}f\|_{L^2_v}\| h\|_{L^2_v}+C_{\eps, l}  \|\langle v\rangle^{4+\gamma}g\|_{L^1_v}\|\langle v\rangle^{l }f\|_{L^2_v}\|\langle v\rangle^{\frac{\gamma}{2}}h\|_{L^2_v}.
\end{equation}
The same argument shows that  $J_{5,2}$ admits the same upper bound as $J_{5,1}.$ 
This, together  with 
  the estimate 
 \begin{equation}\label{fre}
	\|\langle v\rangle^{4+\gamma}g\|_{L^1_v} \leq C\|\langle v\rangle^{6+\gamma}g\|_{L^2_v},
\end{equation}
yields 
\begin{equation}
\label{j5+}
|J_5|= J_{5,1}+J_{5,2}\leq C_{l}  \| g\|_{L^1_v}\|\langle v\rangle^{l+\frac{\gamma}{2}}f\|_{L^2_v}\| h\|_{L^2_v}+C_{l}  \|\langle v\rangle^{6+\gamma}g\|_{L^2_v}\|\langle v\rangle^{l }f\|_{L^2_v}\|\langle v\rangle^{\frac{\gamma}{2}}h\|_{L^2_v},	
\end{equation}
and thus
\begin{equation}
\label{j5}
|J_5| \leq C_{l}\|\langle v\rangle^{6+\gamma}g\|_{L^2_v}\|\langle v\rangle^{l+\frac{\gamma}{2}}f\|_{L^2_v} \|\langle v\rangle^{\frac{\gamma}{2}}h\|_{L^2_v}.	
\end{equation}

{\it Step 5 (Estimate on $J_3$).}
It remains to estimate $J_3$ in \eqref{j}. From the definition of $\kappa^{\perp}$ in \eqref{deperp}, it follows that 
\begin{equation}\label{vvstar}
	v\cdot \kappa^\perp=v_*\cdot\kappa^\perp,
\end{equation}
and thus $J_3$ in \eqref{j} may be   rewrite   as
\begin{equation}\label{j3j31}
\begin{aligned}
J_{3}&=l \int_{\mathbb{R}^6\times\mathbb{S}^2}  |v -v_*|^{\gamma+1} b(\cos\theta)\cos^{l-1}\frac{\theta}{2}\sin\frac{\theta}{2}\big(v_*\cdot \kappa^\perp \big)g_*\langle v'\rangle^{l-2}f' h'\,d\sigma\, dv_*\, dv\\
&\quad+l\int_{\mathbb{R}^6\times\mathbb{S}^2} |v-v_*|^{\gamma+1}b(\cos\theta)\cos^{l-1}\frac{\theta}{2}\sin\frac{\theta}{2}(v_*\cdot \kappa^\perp)  g_*\Big(\langle v\rangle^{l-2}f-\langle v'\rangle^{l-2}f'\Big) h'd\sigma  dv_*  dv\\
&:=J_{3,1}+J_{3,2}.
\end{aligned}
\end{equation}
To estimate $J_{3,1},$ we set, analogously to   $\kappa$ and $\kappa^{\perp}$  in  \eqref{deperp},
\begin{equation*}
	 \kappa'=\frac{v'-v_*}{|v'-v_*|} , \quad \kappa'^{\perp} =\frac{\sigma-(\sigma\cdot\kappa')\kappa'}{\abs{\sigma-(\sigma\cdot\kappa')\kappa'}}.
\end{equation*} 
Then  $\kappa'^{\perp}$ is orthogonal to $\kappa'$, and  moreover we can verify directly that 
\begin{equation*}
	\kappa^{\bot}=\kappa' \sin\frac{\theta}{2}+\kappa'^\bot \cos\frac{\theta}{2},\quad |v'-v_*|=|v-v_*|\cos\frac{\theta}{2}.
\end{equation*}
Consequently, using the above identities and formula \eqref{rech}  in Appendix \ref{Appendix},
 we estimate the term $J_{3,1}$ in \eqref{j3j31} as follows.
\begin{align*}
&J_{3,1} =l\int_{\mathbb{R}^6\times\mathbb{S}^2} |v -v_*|^{\gamma+1} b(\cos\theta)\cos^{l-1}\frac{\theta}{2}\sin\frac{\theta}{2}\\
&\qquad\qquad\qquad \times \bigg (v_*\cdot \Big(\kappa' \sin\frac{\theta}{2}+\kappa'^\perp \cos\frac{\theta}{2}\Big)\bigg)g_*\langle v'\rangle^{l-2}f'h'\,d\sigma\, dv_*\, dv\\
&=l \int  |v-v_*|^{\gamma+1} b(\cos\theta)  \cos^{l- 5-\gamma}\frac{\theta}{2}\sin\frac{\theta}{2} \bigg (v_*\cdot \Big(\kappa  \sin\frac{\theta}{2}+\kappa^\perp \cos\frac{\theta}{2}\Big)\bigg)g_*\langle v\rangle^{l-2}f h\,d\sigma  dv_*  dv,
\end{align*}
where in the second  identity we used \eqref{rech}.  Moreover, using \eqref{vvstar} and \eqref{sdf} gives  
\begin{align*}
 \int_{\mathbb{R}^6\times\mathbb{S}^2} |v-v_*|^{\gamma+1} b(\cos\theta)\cos^{l- 5-\gamma}\frac{\theta}{2} \sin\frac{\theta}{2}  
  \Big  (v_*\cdot \kappa^\perp \cos\frac{\theta}{2} \Big) g_*\langle v\rangle^{\ell-2}f h \,d\sigma\, dv_*\, dv=0.
\end{align*}
Hence
\begin{equation}\label{dj31}
\begin{aligned}
    |J_{3,1}|&= \Big| l \int_{\mathbb{R}^6\times\mathbb{S}^2} |v-v_*|^{\gamma+1} b(\cos\theta)  \cos^{l- 5-\gamma}\frac{\theta}{2}  \sin^2\frac{\theta}{2}  (v_*\cdot  \kappa ) g_*\langle v\rangle^{l-2}f h \,d\sigma\, dv_*\, dv\Big|\\
    &\leq C_l \int_{\mathbb{R}^6} \langle v_*\rangle^{2+\gamma}|g_*|\langle v\rangle^{l-2}|fh|dv dv_*+C_l \int_{\mathbb{R}^6} \langle v_*\rangle |g_*|\langle v\rangle^{l-1+\gamma}|fh|dv dv_*\\
    & \leq C_l\|\langle v\rangle^{2+\gamma}g\|_{L^1_v}\|\langle v\rangle^{l+\frac{\gamma}{2}}f\|_{L^2_v}\|\comi v^{\frac{\gamma}{2}}h\|_{L^2_v}\\
    &\leq C_{l}  \|\langle v\rangle^{6+\gamma}g\|_{L^2_v}\|\comi{D_v}^s\langle v\rangle^{l+\frac{\gamma}{2}}f\|_{L^2_v}\|\langle v\rangle^{\frac{\gamma}{2}}h\|_{L^2_v}.
\end{aligned}
\end{equation}
 For the term $J_{3,2}$  in \eqref{j3j31},  one has
 \begin{equation}\label{j321}
\begin{aligned}
 |J_{3,2}| \leq&    C_l\Big(\int_{\mathbb{R}^6\times\mathbb{S}^2} |v-v_*|^{\gamma}b(\cos\theta) \cos^{2 l -2}\frac{\theta}{2}\sin^2\frac{\theta}{2}\langle v_*\rangle^{2}|h'|^2|g_*|\,d\sigma\, dv_*\, dv\Big)^{\frac{1}{2}}\\
&\quad \times\Big(\int_{\mathbb{R}^6\times\mathbb{S}^2} |v-v_*|^{2+\gamma}b(\cos\theta)|g_*|\big(\langle v\rangle^{ l -2}f-\langle v'\rangle^{ l -2}f'\big)^2\,d\sigma\, dv_*\, dv\Big)^{\frac{1}{2}}\\
 \leq & C_l\|\langle v\rangle^{2+\gamma}g\|_{L^1_v}^{\frac12}\|\langle v\rangle^{\frac{\gamma}{2}}h\|_{L^2_v} \\ 
&\quad  \times \Big(\int_{\mathbb{R}^6\times\mathbb{S}^2} |v-v_*|^{2+\gamma}b(\cos\theta)|g_*|\big(\langle v\rangle^{ l -2}f-\langle v'\rangle^{ l -2}f'\big)^2\,d\sigma\, dv_*\, dv\Big)^{\frac{1}{2}},
\end{aligned}
\end{equation}
where the last inequality holds since it follows from     \eqref{rech}  that 
\begin{align*}
	&\int_{\mathbb{R}^6\times\mathbb{S}^2} |v-v_*|^{\gamma}b(\cos\theta) \cos^{2 l -2}\frac{\theta}{2}\sin^2\frac{\theta}{2}\langle v_*\rangle^{2} |h'|^2|g_*|\,d\sigma\, dv_*\, dv\\
	&=\int_{\mathbb{R}^6\times\mathbb{S}^2}  |v-v_*|^{\gamma}b(\cos\theta) \cos^{2 l -5-\gamma}\frac{\theta}{2}\sin^2\frac{\theta}{2}\langle v_*\rangle^{2} |h|^2|g_*|\,d\sigma\, dv_*\, dv \\
    &\leq C_l\|\langle v\rangle^{2+\gamma}g\|_{L^1_v}\|\langle v\rangle^{\frac{\gamma}{2}}h\|^2_{L^2_v}.
\end{align*}
Now we estimate the last factor in \eqref{j321}.  Using the  identity
\begin{equation}\label{idt}
    (F-F')^2=2F(F-F')+(F'^2-F^2),
\end{equation}
 we write 
 \begin{equation}
 	\label{eq2gamma}
 	\begin{aligned}
&\int_{\mathbb{R}^6\times\mathbb{S}^2} |v-v_*|^{2+\gamma}b(\cos\theta)|g_*|\big(\langle v\rangle^{ l -2}f-\langle v'\rangle^{ l -2}f'\big)^2\,d\sigma\, dv_*\, dv \\
&=2\int_{\mathbb{R}^6\times\mathbb{S}^2}   |v-v_*|^{2+\gamma } b(\cos\theta)   |g_*|\langle v\rangle^{ l -2}f \Big(\langle v\rangle^{ l -2}f -\langle v'\rangle^{ l -2}f'\Big)\,d\sigma\, dv_*\, dv\\
&\quad+\int_{\mathbb{R}^6\times\mathbb{S}^2}  |v-v_*|^{2+\gamma}b(\cos\theta) |g_*|\Big(\langle v'\rangle^{2 l -4}|f'|^2-\langle v\rangle^{2 l -4}|f|^2\Big)\,d\sigma\, dv_*\, dv\\
&\leq 2\int_{\mathbb{R}^6\times\mathbb{S}^2}   |v-v_*|^{\gamma} b(\cos\theta)  |v-v_*|^2 |g_*|\langle v\rangle^{ l -2}f \Big(\langle v\rangle^{ l -2}f -\langle v'\rangle^{ l -2}f'\Big)\,d\sigma\, dv_*\, dv\\
&\quad+C_l \|\langle v\rangle^{2+\gamma}g\|_{L^1_v} \|\langle v\rangle^{ l +\frac{\gamma}{2}}f\|^2_{L^2_v},
\end{aligned}
 \end{equation}
the last inequality using the cancellation lemma  \eqref{cance}.  Moreover, using the change of variables in \eqref{prepost} gives
\begin{align*}
	& \int_{\mathbb{R}^6\times\mathbb{S}^2}   |v-v_*|^{\gamma } b(\cos\theta) |v-v_*|^{2} |g_*|\langle v\rangle^{ l -2}f \Big(\langle v\rangle^{ l -2}f -\langle v'\rangle^{ l -2}f'\Big)\,d\sigma\, dv_*\, dv\\
	&=\int_{\mathbb{R}^6\times\mathbb{S}^2}   |v-v_*|^{\gamma } b(\cos\theta) \big(|v_*|^{2}+ |v|^2-2v\cdot v_*\big) |g_*|\langle v\rangle^{ l -2}f \Big(\langle v\rangle^{ l -2}f -\langle v'\rangle^{ l -2}f'\Big)\,d\sigma\, dv_*\, dv\\
	&= -\big (Q(|v|^2|g|,\langle v\rangle^{ l -2}f),\, \langle v\rangle^{ l -2}f\big)_{L^2_v}-\big(Q(|g|,|v|^2\langle v\rangle^{ l -2}f),\, \langle v\rangle^{ l -2}f\big)_{L^2_v}\\
&\quad +2\sum_{i=1}^3 \big (Q(v_i|g|, v_i\langle v\rangle^{ l -2}f),\, \langle v\rangle^{ l -2}f \big )_{L^2_v}\\
&\leq C_l \big(\|\langle v\rangle^{2+\gamma+2s}g\|_{L^1_v}+\|\langle v\rangle^{2}g\|_{L^2_v}\big) \|\comi{D_v}^s\langle v\rangle^{ l +\frac{\gamma}{2}} f\|^2_{L^2_v} \leq C_l\|\langle v\rangle^{6+\gamma}g\|_{L^2_v}  \|\comi{D_v}^s\langle v\rangle^{ l +\frac{\gamma}{2}} f\|^2_{L^2_v},
\end{align*}
where the last inequality follows from \eqref{uppbound2} as well as \eqref{fre}. Now substituting the above estimate into \eqref{eq2gamma}, we obtain
\begin{multline*}
    \int_{\mathbb{R}^6\times\mathbb{S}^2} |v-v_*|^{2+\gamma}b(\cos\theta)|g_*|\big(\langle v\rangle^{ l -2}f-\langle v'\rangle^{ l -2}f'\big)^2\, d\sigma  \,dv_* \, dv \\  
     \leq C_l \|\langle v\rangle^{6+\gamma}g\|_{L^2_v} \|\comi{D_v}^s\langle v\rangle^{l+\frac{\gamma}{2}} f\|^2_{L^2_v},
\end{multline*}
which with \eqref{j321} yields 
\begin{align*}
|J_{3,2}|\leq C_l \|\langle v\rangle^{6+\gamma}g\|_{L^2_v}\|\comi{D_v}^s\langle v\rangle^{l+\frac{\gamma}{2}} f\|_{L^2_v}\|\langle v\rangle^{\frac{\gamma}{2}}h\|_{L^2_v}.
\end{align*}
Combining  this and   estimate \eqref{dj31} with \eqref{j3j31}, we conclude 
\begin{equation*}
	|J_3|\leq C_l\|\langle v\rangle^{6+\gamma}g\|_{L^2_v}\|\comi{D_v}^s\langle v\rangle^{l+\frac{\gamma}{2}} f\|_{L^2_v}\|\langle v\rangle^{\frac{\gamma}{2}}h\|_{L^2_v}.
\end{equation*}
 This,  with   the upper bounds of $J_1, J_2,J_4$ and $J_5$    in \eqref{j1}, \eqref{j2}, \eqref{j4} and \eqref{j5},  yields
 \begin{align*}
&\Big | (\langle v\rangle^ l  Q(g,f), h)_{L^2_v}-( Q(g,\langle v\rangle^ l  f), h)_{L^2_v}\Big|\leq \sum_{m=1}^5 |J_m|\\
&\leq C_l \|\langle v\rangle^{6+\gamma}g\|_{L^2_v}\|\langle v\rangle^{\frac{\gamma}{2}}h\|_{L^2_v}\|\comi{D_v}^s \langle v\rangle^{l+\frac{\gamma}{2}}f\|_{L^2_v}+	2 \boldsymbol{\omega}_{l-2-\gamma} \|f\|_{L^1_v} \|\langle v\rangle^{l+\frac{\gamma}{2}}g\|_{L^2_v} \|\langle v\rangle^{ \frac{\gamma}{2}}h\|_{L^2_v}\\
 		&\quad +C_l  \|\langle v\rangle^{ 1+\gamma}f\|_{L^1_v}\|\langle v\rangle^{l+\frac{\gamma}{2}}g\|_{L^2_v}  \| h\|_{L^2_v}+C_l  \|\langle v\rangle^{ 1+\gamma}f\|_{L^1_v}\|\langle v\rangle^{l}g\|_{L^2_v}  \| \comi v^{\frac{\gamma}{2}}h\|_{L^2_v}.
 \end{align*} 
Thus the first estimate in Lemma \ref{commutator} follows.  This completes the  proof of  Lemma \ref{commutator}.   
\end{proof}

\begin{lemma}[Commutator estimate]\label{lem:coer}
Let $g\geq 0.$   Then for any $ l \geq \frac{13}{2}+\gamma$,   there exists  a constant $C_{ l }>0$ depending  on $ l $, such that
 \begin{equation*}
 	\begin{aligned}
 		&  \big(\langle v\rangle^ l  Q(g,f), \comi v^l f)_{L^2_v}-( Q(g,\langle v\rangle^ l  f), \comi v^l f\big)_{L^2_v} \\
    &\leq -\frac{1}{2}\Big( Q(g,\langle v\rangle^{ l }f), \langle v\rangle^{ l } f\Big)_{L^2_{v}}+ 2^{1+\gamma}\big(\boldsymbol{\lambda}_l+\boldsymbol{A}_{\gamma,s}\big) \| g\|_{L^1_v}   \|\langle v\rangle^{l+\frac{\gamma}{2}}f\|_{L^2_v}^2+ \|\langle v\rangle^{l+ \frac{\gamma}{2}}f\|_{L^2_v}^2\\
        &\quad+C_l \|\langle v\rangle^{6+\gamma}f\|_{L^2_v}^2\|\langle v\rangle^{l+\frac{\gamma}{2}}g\|_{L^2_v}^2+C_{ l }\Big(1+\|\langle v\rangle^{6+\gamma}g\|_{L^2_v}^2\Big)  \|\langle v\rangle^{l}f\|_{L^2_v}^2,
 	\end{aligned}
 \end{equation*}
 recalling $\boldsymbol{\lambda}_l $ and $\boldsymbol{A}_{\gamma,s}$ are the constants defined in \eqref{lamome} and \eqref{agams}, respectively.
\end{lemma}
\begin{proof}
Letting $h=\comi v^l f$ in
  \eqref{j} and then using \eqref{crj1}, \eqref{j2}, \eqref{j4} and \eqref{j5+} for $h= \comi v^l f$, we obtain 
\begin{equation}\label{j+}
\begin{aligned}
& (\langle v\rangle^ l  Q(g,f), \comi v^l f)_{L^2_v}-( Q(g,\langle v\rangle^ l  f), \comi v^lf)_{L^2_v} \\
&\leq 2^{1+\gamma} \boldsymbol{\lambda}_l \| g\|_{L^1_v} \|\langle v\rangle^{l+\frac{\gamma}{2}}f\|_{L^2_v}^2+C_l \|\langle v\rangle^{6+\gamma}g\|_{L^2_v}\|\langle v\rangle^{l+\frac{\gamma}{2}}f\|_{L^2_v}  \|\langle v\rangle^{l}f\|_{L^2_v} \\
&\quad+C_l\|\comi v^{6+\gamma} f\|_{L^2_v} \|\langle v\rangle^{l+\frac{\gamma}{2}}g\|_{L^2_v} \|\langle v\rangle^{ l+\frac{\gamma}{2}}f\|_{L^2_v}\\
       &\quad+ l \int_{\mathbb{R}^6\times\mathbb{S}^2} |v-v_*|^{\gamma+1}b(\cos\theta)\langle v\rangle^{ l -2}(v\cdot \kappa^\perp) \sin\frac{\theta}{2}\cos^{l -1}\frac{\theta}{2}\,g_*f  \comi {v'}^l f' \,d \sigma\, dv_*\, dv.
\end{aligned}
\end{equation}
To estimate the last term, observing that  \eqref{sdf} implies   
\begin{equation*}
    \int_{\mathbb{R}^6\times\mathbb{S}^2} |v-v_*|^{\gamma+1}b(\cos\theta)\langle v\rangle^{ l -2}(v\cdot \kappa^\perp) \sin\frac{\theta}{2}\cos^{l -1}\frac{\theta}{2}\,g_*f  \comi {v}^l f \,d \sigma\, dv\, dv_*=0,
\end{equation*}
we have, for $g\geq 0,$ 
\begin{equation}\label{tj3}
\begin{aligned}
 &l \int_{\mathbb{R}^6\times\mathbb{S}^2} |v-v_*|^{\gamma+1}b(\cos\theta)\langle v\rangle^{ l -2}(v\cdot \kappa^\perp) \sin\frac{\theta}{2}\cos^{l -1}\frac{\theta}{2}\,g_*f  \comi {v'}^l f' \,d \sigma\, dv_*\, dv\\
 &=l \int_{\mathbb{R}^6\times\mathbb{S}^2} |v-v_*|^{\gamma+1}b(\cos\theta)\langle v\rangle^{ l -2}(v\cdot \kappa^\perp) \sin\frac{\theta}{2}\cos^{l -1}\frac{\theta}{2}g_*f \left (\langle v'\rangle^ l  f'-\langle v\rangle^ l  f \right)\,d\sigma  dv_* dv\\
&\leq \frac{1}{4}\int_{\mathbb{R}^6\times\mathbb{S}^2} |v-v_*|^{\gamma}b(\cos\theta)  g_*\big(\langle v'\rangle^ l  f'-\langle v\rangle^ l  f \big)^2 \,d\sigma\, d v_*\, d v  \\
&\quad +C_{ l }\int_{\mathbb{R}^6\times\mathbb{S}^2} |v-v_*|^{\gamma+2}b(\cos\theta) \langle v\rangle^{2 l -4} |v_*\cdot \kappa^\perp|^2\sin^2\frac{\theta}{2}\cos^{2(l -1)}\frac{\theta}{2} \,g_* f^2  \,d\sigma\, d v_*\, d v,
\end{aligned}
\end{equation}
where the first inequality uses the fact that $v\cdot\kappa^{\perp}=v_*\cdot\kappa^{\perp}.$
For the first term on the right-hand side of \eqref{tj3}, 
using the identity  \eqref{idt} again, one has
\begin{equation}
    \label{d+}
\begin{aligned}
&\frac14\int_{\mathbb{R}^6\times\mathbb{S}^2} |v-v_*|^{\gamma}b(\cos\theta)g_*\big(\langle v\rangle^{ l}f-\langle v'\rangle^{ l }f'\big)^2\,d\sigma\, dv_*\, dv \\
&=\frac{1}{2}\int_{\mathbb{R}^6\times\mathbb{S}^2}   |v-v_*|^{\gamma } b(\cos\theta)   g_*\langle v\rangle^{ l }f \Big(\langle v\rangle^{ l }f -\langle v'\rangle^{ l }f'\Big)\,d\sigma\, dv_*\, dv\\
&\quad+\int_{\mathbb{R}^6\times\mathbb{S}^2}  |v-v_*|^{\gamma}b(\cos\theta) g_*\Big(\langle v'\rangle^{2 l }|f'|^2-\langle v\rangle^{2 l }f^2\Big)\,d\sigma\, dv_*\, dv\\
&\leq -\frac12\big(Q(g,\, \comi v^{l}f), \, \comi v^{l}f)\big) +2^\gamma  \boldsymbol{A}_{\gamma,s} \| g\|_{L^1_v} \|\langle v\rangle^{ l+\frac{\gamma}{2}}f\|^2_{L^2_v}+C_l \| \comi v^\gamma g\|_{L^1_v} \|\langle v\rangle^{ l}f\|^2_{L^2_v},
\end{aligned}
\end{equation}
where, in the last line, we used the pre-postcollisional change of variables \eqref{prepost} together with  \eqref{elementaryinequality} and  the cancellation lemma  \eqref{cance}. 
Here $\boldsymbol{A}_{\gamma,s}$ is the constant defined in \eqref{agams}.   Moreover,
   following the argument in the proof of \eqref{j51},  we have the upper bound of the last term in \eqref{tj3} as below: for any $\eps>0,$
\begin{equation*}
  \begin{aligned}
   &C_l \int_{\mathbb{R}^6\times\mathbb{S}^2} |v-v_*|^{\gamma+2}b(\cos\theta) \langle v\rangle^{2 l -4} |v_*\cdot \kappa^\perp|^2\sin^2\frac{\theta}{2}\cos^{2(l -1)}\frac{\theta}{2} \,g_* f^2  \,d\sigma\, d v_*\, d v\\
  &   \leq  \eps  \| g\|_{L^1_v}\| \langle v\rangle^{ l +\frac{\gamma}{2} }f\|^2_{L^2_v} + C_{\eps,l} \|\langle v\rangle^{6+\gamma} g\|_{L^2_v}\| \langle v\rangle^{ l}f\|^2_{L^2_v}.
    \end{aligned}
\end{equation*}
Letting $\eps=2^\gamma\boldsymbol{A}_{\gamma,s}$ in 
the above estimate and then 
   substituting it and \eqref{d+} into \eqref{tj3}, we conclude
\begin{align*}
       &l \int_{\mathbb{R}^6\times\mathbb{S}^2} |v-v_*|^{\gamma+1}b(\cos\theta)\langle v\rangle^{ l -2}(v\cdot \kappa^\perp) \sin\frac{\theta}{2}\cos^{l -1}\frac{\theta}{2}\,g_*f  \comi {v'}^l f' \,d \sigma\, dv\, dv_*\\
& \leq -\frac{1}{2}\Big( Q(g,\langle v\rangle^{ l }f), \langle v\rangle^{ l } f\Big)_{L^2_{v}}+2^{1+\gamma}\boldsymbol{A}_{\gamma,s}\| g \|_{L^1_v}\|\langle v\rangle^{ l +\frac{\gamma}{2}}f \|^2_{L^2_v}  +C_{ l }\|\langle v\rangle^{6+\gamma } g\|_{L^2_v} \|\langle v\rangle^{ l}f\|^2_{L^2_v},
\end{align*}
 which with \eqref{j+}  yields
 \begin{equation*}
 	\begin{aligned}
& (\langle v\rangle^ l  Q(g,f), \comi v^l f)_{L^2_v}-( Q(g,\langle v\rangle^ l  f), \comi v^lf)_{L^2_v} \\
&\leq  -\frac{1}{2}\Big( Q(g,\langle v\rangle^{ l }f), \langle v\rangle^{ l } f\Big)_{L^2_{v}}+2^{1+\gamma}\big( \boldsymbol{\lambda}_l+\boldsymbol{A}_{\gamma,s}\big) \| g\|_{L^1_v} \|\langle v\rangle^{l+\frac{\gamma}{2}}f\|_{L^2_v}^2\\
&\quad+ C_{ l }\|\langle v\rangle^{6+\gamma } g\|_{L^2_v}\| \langle v\rangle^{ l}f\|^2_{L^2_v}+C_l \|\langle v\rangle^{6+\gamma}g\|_{L^2_v}\|\langle v\rangle^{l+\frac{\gamma}{2}}f\|_{L^2_v}  \|\langle v\rangle^{l}f\|_{L^2_v} \\
&\quad+C_l\|\comi v^{6+\gamma} f\|_{L^2_v} \|\langle v\rangle^{l+\frac{\gamma}{2}}g\|_{L^2_v} \|\langle v\rangle^{ l+\frac{\gamma}{2}}f\|_{L^2_v}.
\end{aligned}
 \end{equation*}
 This yields the assertion   in 
 Lemma \ref{lem:coer}. The proof is completed.
 \end{proof}
 
Based on the commutator estimates established above, we now proceed to   prove the main result of this part.     
  
\begin{proof}
	[Proof of Proposition \ref{upbounded}]   Noting  the identity 
	\begin{equation}\label{gfh}
		\big(\langle v\rangle^{ l } Q(g,f),\  h\big)_{L^2_{v}}  
	=\big( Q(g,\langle v\rangle^{ l }f),\  h\big)_{L^2_{v}}+ \big (\langle v\rangle^ l  Q(g,f)- Q\big(g,\langle v\rangle^ l  f),\ h\big)_{L^2_v} ,
	\end{equation}
	we use  \eqref{uppbound} and \eqref{sa} to conclude 
	\begin{equation*}
		\begin{aligned}
		&	|\big(\langle v\rangle^{ l } Q(g,f), h\big)_{L^2_{v}}|\leq  C \big(\|\langle v\rangle^{2s+\gamma}g\|_{L^1_v}+\|g\|_{L^2_v}\big)\normm{\comi v^l f}\left(\normm{h}+\|\langle v\rangle^{ s+\frac{\gamma}{2}}h\|_{L^2_v}\right)\\
		&\qquad\quad+ C_l \|\langle v\rangle^{6+\gamma}f\|_{L^2_v}   \|\langle v\rangle^{l+\frac{\gamma}{2}}g\|_{L^2_v}  \|\langle v\rangle^{ \frac{\gamma}{2}}h\|_{L^2_v}  +C_{ l } \|\langle v\rangle^{6+\gamma}g\|_{L^2_v}  \norm{\comi v^{ \frac{\gamma}{2}}h}_{L_v^2} \normm{\langle v\rangle^{l}f} .
		\end{aligned}
	\end{equation*} 
	This with \eqref{fre} gives the desired estimate  \eqref{upbound-v}.

Letting  $h=\comi v^l f$ in  identity \eqref{gfh}    and then using \eqref{lower+} and   \eqref{sa}, we obtain 
		\begin{equation*}
		\begin{aligned}
		&	 \big(\langle v\rangle^{ l } Q(g,f), \comi v^l f \big)_{L^2_{v}} \leq  -c_0\normm{\comi v^lf}^2+C_0\norm{\comi v^{l+s+\frac{\gamma}{2}}f}_{L_v^2}^2 \\
		&\quad  +   C_l \|\langle v\rangle^{6+\gamma}f\|_{L^2_v}   \|\langle v\rangle^{l+\frac{\gamma}{2}}g\|_{L^2_v}  \|\langle v\rangle^{l+ \frac{\gamma}{2}} f\|_{L^2_v}  +C_{ l } \|\langle v\rangle^{6+\gamma}g\|_{L^2_v}  \norm{\comi v^{ l+\frac{\gamma}{2}} f}_{L_v^2} \normm{\langle v\rangle^{l}f},
		\end{aligned}
	\end{equation*} 
	which with the fact that $\|\langle v\rangle^{l+ \frac{\gamma}{2}} f\|_{L^2_v} \leq \normm {\comi v^lf}$ implies   estimate \eqref{coer-v} as desired.  
	
	Finally, letting  $h=\comi v^l f$ in  identity \eqref{gfh}    and then using   Lemma \ref{lem:coer}, we get 		\begin{multline*}
		  \big(\langle v\rangle^{ l } Q(g,f), \comi v^l f \big)_{L^2_{v}} 
		 \leq  \frac{1}{2}\Big( Q(g,\langle v\rangle^{ l }f), \langle v\rangle^{ l } f\Big)_{L^2_{v}}+2^{1+\gamma} \big(\boldsymbol{\lambda}_l+\boldsymbol{A}_{\gamma,s}\big) \| g\|_{L^1_v}   \|\langle v\rangle^{l+\frac{\gamma}{2}}f\|_{L^2_v}^2\\
		 + \|\langle v\rangle^{l+ \frac{\gamma}{2}}f\|_{L^2_v}^2
         +C_l \|\langle v\rangle^{6+\gamma}f\|_{L^2_v}^2\|\langle v\rangle^{l+\frac{\gamma}{2}}g\|_{L^2_v}^2+C_{ l }\Big(1+\|\langle v\rangle^{6+\gamma}g\|_{L^2_v}^2\Big)   \|\langle v\rangle^{l}f\|_{L^2_v}^2.
		\end{multline*}
This,  together with \eqref{lower+2} as well as $C_0\geq 1,$  yields estimate \eqref{+coercivity-com}. 	
	The proof of Proposition \ref{upbounded} is thus completed.  
 \end{proof}


\section{Hypoelliptic estimates}\label{app:sub}

This and the next sections are devoted to deriving an {\it \`a priori} estimate for  the following linear Boltzmann equation 
\begin{equation}\label{eq-lin}
	\partial_tf + v\cdot \partial_x f
		= Q(g,f),  \quad f|_{t=0}=f_{{in}},
\end{equation}
with the same initial data $f_{{in}}\ge0$ as in the nonlinear Cauchy problem  \eqref{eq-1}.  In this section, we establish a hypoellipitc estimate for \eqref{eq-lin}, which is essential for controlling the loss-of-weight terms. In the next section we will  deduce the  {\it \`a priori} estimate based on the hypoelliptic estimate.

 Before stating the main result of this part, we introduce a  new function space  which  corresponds to the hypoelliptic estimate 
 in the spatial variable  $x\in\mathbb{T}^3$.
 
\begin{definition}
	\label{defxln}
	The function space $\mathcal{Z}_\ell$ is defined by
	\begin{equation*}
		\mathcal{Z}_\ell=\big\{h ;\ \norm{h}_{\mathcal{Z}_\ell}<+\infty\big\},
	\end{equation*}
	where
	 \begin{equation}\label{zn}
		\|h\|^2_{{\mathcal{Z}}_{\ell}} =  \sum_{|\alpha|=3} \norm{\langle v\rangle^{\frac{\gamma}{2(1+2s)}+\ell-\rho|\alpha|}\comi{D_x}^{\frac{s}{1+2s}}  \partial^{\alpha}_{x}h}^2_{L^2_{x, v}} 
  	    +  \norm{\langle v\rangle^{\frac{\gamma}{2(1+2s)}+\ell}\comi{D_x}^{\frac{s}{1+2s}}   h}^2_{L^2_{x, v}}\, ,
	\end{equation}
 recalling  $\comi{D_x}^{\frac{s}{1+2s}}$ is the Fourier multiplier defined by \eqref{fxfv}.  
  \end{definition} 
  
\begin{theorem}\label{hypoellip}
Let $0<c_0\leq 1$ be the constant  in the coercivity estimate \eqref{lower+}, and let $\mathscr{X}_\ell, \mathcal{Y}_\ell$ and $\mathcal{Z}_\ell$ 
be the spaces given in Definitions \ref{xel} and \ref{defxln}.  Then there exists a constant $a_0>0,$ depending only on $c_0, \ell$ and the parameters $s,\gamma$ in \eqref{kern},  such that
\begin{equation}\label{kesub}
		\begin{aligned}
		& \Big(\sup_{t\leq T}\norm{f}_{\mathscr{X}_{\ell}} \Big)^2 + \frac{c_0}{2} \int_0^T\norm{ f}_{\mathcal{Y}_{\ell}}^2 dt +  
        \frac{a_0}{2} \int_0^T \norm{ f}_{\mathcal{Z}_{\ell}}^2 dt\\
	&\leq 2\norm{f_{in}}_{\mathscr{X}_{\ell}}^2+ \frac{c_0 }{16\boldsymbol{L}} \sup_{t\leq T}\norm{g}_{\mathscr{X}_{\ell}} 
    \int_0^T \| f\|_{{\mathcal{Y}}_{\ell}} ^2dt\\
    &\quad +C \Big( 1+\sup_{t\leq T}\norm{g}_{\mathscr{X}_{\ell}} \Big)^2\int_0^T\norm{\comi{D_v}^s\comi v^{\boldsymbol{\ell_1}+\frac{\gamma}{2}}f}_{L^2_{x, v}}^2dt  
\\
	&\quad+C \Big( 1+\sup_{t\leq T}\norm{g}_{\mathscr{X}_{\ell}} \Big)^2\int_0^T\norm{\comi v^{s+\frac{\gamma}{2}}f}_{\mathscr{X}_{\ell}}^2dt  
    +C  \Big(\sup_{t\leq T}  \|f\|_{\mathscr{X}_{\ell}}\Big)^2  \int_0^T  \|\langle v\rangle^{s+ \frac{\gamma}{2}}g\|_{\mathscr{X}_{\ell}}^2dt,
		\end{aligned}
	\end{equation}
where $\boldsymbol{L}$ is  defined by  \eqref{A-0},  and the constant  $\boldsymbol{\ell_1}$ is  specified  as in \eqref{loc-l}. 
\end{theorem}

\begin{remark}
The second  and  third terms  on the right-hand side of \eqref{kesub} provide a refinement of the following 
\begin{equation*}
C\Big(1+\sup_{t\leq T}\norm{g}_{\mathscr{X}_{\ell}}\Big)^2  \int_0^T \| f\|_{\mathcal{Y}_{\ell}} ^2dt. 
\end{equation*}
We will use this refined estimate to overcome the lack of smallness assumption on the initial data. Further details can be found at   Section \ref{s:priori}. The last two terms in  estimate \eqref{kesub}  account for the loss of weights. As shown in  Proposition \ref{crulem},   we may control these weight-loss terms in terms of  the moment  and  the hypoelliptic estimates. 
\end{remark}
  
\subsection{Upper bound of trilinear terms}
 To prove Theorem \ref{hypoellip}, we begin by establishing the following upper bound for the trilinear terms.
 
\begin{lemma} \label{techlem}
Let $\ell,\boldsymbol{\ell_1}$ be the constants given in \eqref{rhod}, and let $\alpha\in\mathbb Z_+^3$ be a  given multi-index  satisfying that $\abs\alpha= 3$. Then for any $\varepsilon>0$ it holds that
\begin{align*}
	 &\sum_{\substack{\beta\leq \alpha\\ \abs\beta\geq 1}}\big |\big(\langle v\rangle^{\ell-\rho|\alpha|}Q(\partial^{\beta}_xg, \partial_x^{\alpha-\beta}f), \  \langle v\rangle^{\ell-\rho|\alpha|}\partial^{\alpha}_x h\big)_{L^2_{x, v}}\big| \leq   \eps \|g\|_{\mathscr{X}_\ell} \big( \norm{  f}_{\mathcal{Y}_{\ell}}^2+ \norm{h}_{\mathcal{Y}_{\ell}}^2\big)\\
    &\qquad\qquad +C_\eps \|g\|_{\mathscr{X}_\ell} \Big(\normm{\comi v^{\boldsymbol{\ell_1}-s}f}_{L_x^2}^2+\|\langle v\rangle^{ s+\frac{\gamma}{2}}h\|_{\mathscr{X}_{\ell}}^2  \Big)  + C  \| f\|_{\mathscr{X}_{\ell}}  \norm{\comi v^{\frac{\gamma}{2}}g}_{\mathscr{X}_{\ell}}  \norm{\comi v^{\frac{\gamma}{2}}h}_{\mathscr{X}_{\ell}}.
\end{align*}
Meanwhile, 
\begin{align*}
	&\big |\big(\langle v\rangle^{\ell-\rho|\alpha|}Q(g, \partial_x^{\alpha}f), \  
    \langle v\rangle^{\ell-\rho|\alpha|}\partial^{\alpha}_x h\big)_{L^2_{x, v}}\big| \\
	& \qquad \leq C    \| g\|_{\mathscr{X}_{\ell}} \big( \norm{  f}_{\mathcal{Y}_{\ell}}^2+  \norm{h}_{\mathcal{Y}_{\ell}}^2\big)
    +C  \| g\|_{\mathscr{X}_{\ell}}   \|\langle v\rangle^{ s+\frac{\gamma}{2}}h\|_{\mathscr{X}_{\ell}}^2 + C  \| f\|_{\mathscr{X}_{\ell}}  \norm{\comi v^{\frac{\gamma}{2}}g}_{\mathscr{X}_{\ell}}  \norm{\comi v^{\frac{\gamma}{2}}h}_{\mathscr{X}_{\ell}}.
\end{align*}
Consequently, 
\begin{align*}
	&\big |\big(\langle v\rangle^{\ell-\rho|\alpha|}\partial_x^\alpha Q(g, f), \  \langle v\rangle^{\ell-\rho|\alpha|}\partial^{\alpha}_x h\big)_{L^2_{x, v}}\big| \\
	& \qquad \leq C    \| g\|_{\mathscr{X}_{\ell}} \big( \norm{  f}_{\mathcal{Y}_{\ell}}^2+  \norm{h}_{\mathcal{Y}_{\ell}}^2\big)+C  \| g\|_{\mathscr{X}_{\ell}}   \|\langle v\rangle^{ s+\frac{\gamma}{2}}h\|_{\mathscr{X}_{\ell}}^2 + C  \| f\|_{\mathscr{X}_{\ell}}  \norm{\comi v^{\frac{\gamma}{2}}g}_{\mathscr{X}_{\ell}}  \norm{\comi v^{\frac{\gamma}{2}}h}_{\mathscr{X}_{\ell}}.
\end{align*}
\end{lemma}

\begin{proof} 
In the proof, we assume $\alpha$ is any given multi-index with $\abs\alpha=3.$ 
 For any  $\tilde\beta\in\mathbb Z_+^3$ with $|\tilde\beta|\leq 2,$  we have that, for any $\eps>0,$
\begin{equation}\label{inter2+}
  \normm{ \langle v\rangle^{\ell-3\rho}\partial^{\tilde \beta}_xf }_{L^2_{x}} \leq \eps  \norm{f}_{\mathcal{Y}_\ell}   +C_\eps\normm{\langle v\rangle^{\boldsymbol{\ell_1}-s}f}_{L^2_{x}},
\end{equation}
which holds  since it  follows from  
  the interpolation inequality as well as  \eqref{rhod} that
 \begin{multline*}
    \normm{ \langle v\rangle^{\ell-3\rho}\partial^{\tilde \beta}_xf }_{L^2_{x}}  \leq  \eps \sum_{|\beta|=3}\normm{ \langle v\rangle^{\ell-\rho|\beta|}\partial^{\beta}_xf }_{L^2_{x}}+C_\eps\normm{ \langle v\rangle^{\ell-3\rho}f }_{L^2_{x}}\\
    \leq \eps  \norm{f}_{\mathcal{Y}_\ell}  +C_\eps\normm{\langle v\rangle^{\ell-3\rho}f}_{L^2_{x}}\leq \eps  \norm{f}_{\mathcal{Y}_\ell} +\eps\normm{\comi v^\ell f}_{L_x^2} +C_\eps\normm{\langle v\rangle^{\boldsymbol{\ell_1}-s}f}_{L^2_{x}}.  
 \end{multline*} 
  Now we split the sum as
\begin{equation}\label{dsa}
	\sum_{\substack{\beta\leq \alpha\\ \abs\beta\geq 1}}\big |\big(\langle v\rangle^{\ell-\rho|\alpha|}Q(\partial^{\beta}_xg, \partial_x^{\alpha-\beta}f), \  \langle v\rangle^{\ell-\rho|\alpha|}\partial^{\alpha}_x h\big)_{L^2_{x, v}}\big| \leq K_1+K_2+K_3,
\end{equation}
where 
\begin{equation}\label{i123}
	\left\{
	\begin{aligned}
	 	K_1&=  \sum_{\substack{\beta\leq \alpha, \abs\beta=1}}\big |\big(\langle v\rangle^{\ell-\rho|\alpha|}Q(\partial^{\beta}_xg, \partial_x^{\alpha-\beta}f), \  \langle v\rangle^{\ell-\rho|\alpha|}\partial^{\alpha}_x h\big)_{L^2_{x, v}}\big|,\\
	K_2& =  \sum_{\substack{\beta\leq \alpha, \abs\beta=2}}\big |\big(\langle v\rangle^{\ell-\rho|\alpha|}Q(\partial^{\beta}_xg, \partial_x^{\alpha-\beta}f), \  \langle v\rangle^{\ell-\rho|\alpha|}\partial^{\alpha}_x h\big)_{L^2_{x, v}}\big|,\\
	K_3&=\sum_{\substack{\beta\leq \alpha, \abs\beta=3}}\big |\big(\langle v\rangle^{\ell-\rho|\alpha|}Q(\partial^{\beta}_xg, \partial_x^{\alpha-\beta}f), \  \langle v\rangle^{\ell-\rho|\alpha|}\partial^{\alpha}_x h\big)_{L^2_{x, v}}\big|.
	\end{aligned}
\right.
\end{equation}
To estimate $K_1$, 
	using \eqref{upbound-v} in Proposition \ref{upbounded}  yields, for $\abs\alpha=3$ and any $\eps>0,$
\begin{equation*}
\begin{aligned}
	K_1 
	 \leq &   C   \sum_{\substack{\beta\leq \alpha, \abs\beta=1}} \|\langle v\rangle^{6+\gamma} \partial_x^\beta g\|_{L^{\infty}_{x}L^2_v} \normm{ \langle v\rangle^{\ell-\rho\abs\alpha}\partial_x^{\alpha-\beta} f}_{L_x^2} \\
	 &\qquad\qquad\qquad\times\Big( \normm{ \langle v\rangle^{\ell-\rho\abs\alpha}\partial_x^\alpha h}_{L_x^2}+ \| \langle v\rangle^{\ell-\rho\abs\alpha+ s+\frac{\gamma}{2}}\partial_x^\alpha h\|_{L^2_{x, v}}\Big)\\
	& + C  \sum_{\substack{\beta\leq \alpha, \abs\beta=1}} \|\langle v\rangle^{6+\gamma} \partial_x^{\alpha-\beta} f  \|_{L^2}   \|\langle v\rangle^{\ell-\rho\abs\alpha+\frac{\gamma}{2}}\partial_x^\beta g\|_{L_{x}^\infty L_v^2}  \|\langle v\rangle^{\ell-\rho\abs\alpha+\frac{\gamma}{2}}\partial_x^{\alpha} h\|_{L^2_{x, v}} \\
	\leq &    C  \|g\|_{\mathscr{X}_\ell} \Big(\eps\norm{  f}_{\mathcal{Y}_{\ell}}+C_\eps \normm{\langle v\rangle^{\boldsymbol{\ell_1}-s}f}_{L_x^2}\Big)\Big( \norm{h}_{\mathcal{Y}_{\ell}}+ \|\langle v\rangle^{ s+\frac{\gamma}{2}}h\|_{\mathscr{X}_{\ell}}  \Big) \\
	&+  C   \| f\|_{\mathscr{X}_\ell} \norm{\comi v^{\frac{\gamma}{2}}g}_{\mathscr{X}_{\ell}}  \|\langle v\rangle^{\frac{\gamma}{2}}h\|_{\mathscr{X}_{\ell}},
\end{aligned}
\end{equation*}
where the last inequality follows from \eqref{inter2+} as well as \eqref{rhod}   which implies 
\begin{equation}
	\label{xl}
	 \|\langle v\rangle^{6+\gamma}   h\|_{H^{3}_{x}L^2_v} \leq C \| h\|_{\mathscr{X}_{\ell}}.
\end{equation}
Thus, for any $\varepsilon>0,$ 
 \begin{equation*}
 \begin{aligned}
 K_1 \leq & \eps \|g\|_{\mathscr{X}_\ell} \big( \norm{  f}_{\mathcal{Y}_{\ell}}^2+ \norm{h}_{\mathcal{Y}_{\ell}}^2\big)+C_\eps \|g\|_{\mathscr{X}_\ell} \Big(\normm{\comi v^{\boldsymbol{\ell_1}-s}f}_{L_x^2}^2+\|\langle v\rangle^{ s+\frac{\gamma}{2}}h\|_{\mathscr{X}_{\ell}}^2  \Big) \\
	&+ C  \| f\|_{\mathscr{X}_{\ell}}  \norm{\comi v^{\frac{\gamma}{2}}g}_{\mathscr{X}_{\ell}}  \norm{\comi v^{\frac{\gamma}{2}}h}_{\mathscr{X}_{\ell}}.
 	\end{aligned}
 \end{equation*} 
  An analogous argument yields the same upper bound for $K_2$ and $K_3$ in \eqref{i123}; the details are omitted for brevity.   
  Substituting these estimates into \eqref{dsa} yields the first assertion in Lemma \ref{techlem}.
  
The second assertion follows an argument similar to that for $K_1$, and here the estimate \eqref{inter2+} is not applicable.   Instead, using \eqref{upbound-v} in Proposition \ref{upbounded} we obtain 
\begin{equation*}
\begin{aligned}
	&\big |\big(\langle v\rangle^{\ell-\rho|\alpha|}Q(g, \partial_x^{\alpha}f), \  
    \langle v\rangle^{\ell-\rho|\alpha|}\partial^{\alpha}_x h\big)_{L^2_{x, v}}\big|\\
		&\leq    C   \|\langle v\rangle^{6+\gamma}  g\|_{L^{\infty}_{x}L^2_v} \normm{ \langle v\rangle^{\ell-\rho\abs\alpha}\partial_x^{\alpha} f}_{L_x^2}  \Big( \normm{ \langle v\rangle^{\ell-\rho\abs\alpha}\partial_x^\alpha h}_{L_x^2}+ \| \langle v\rangle^{\ell-\rho\abs\alpha+ s+\frac{\gamma}{2}}\partial_x^\alpha h\|_{L^2_{x, v}}\Big)\\
	& \quad + C \|\langle v\rangle^{6+\gamma} \partial_x^{\alpha } f  \|_{L^2_{x, v}}   \|\langle v\rangle^{\ell-\rho\abs\alpha+\frac{\gamma}{2}}  g\|_{L_{x}^\infty L_v^2}  \|\langle v\rangle^{\ell-\rho\abs\alpha+\frac{\gamma}{2}}\partial_x^{\alpha} h\|_{L^2_{x, v}} \\
	&\leq   C   \|g\|_{\mathscr{X}_{\ell}} \norm{  f}_{\mathcal{Y}_{\ell}}\big( \norm{h}_{\mathcal{Y}_{\ell}}+ \|\langle v\rangle^{ s+\frac{\gamma}{2}}h\|_{\mathscr{X}_{\ell}}  \big)  +  C  \|f\|_{\mathscr{X}_{\ell}}  \norm{\comi v^{\frac{\gamma}{2}}g}_{\mathscr{X}_{\ell}} \|\langle v\rangle^{\frac{\gamma}{2}}h\|_{\mathscr{X}_{\ell}}.
\end{aligned}
\end{equation*}
 This yields the second assertion in  Lemma \ref{techlem}. The last assertion follows directly from the first two together with Leibniz's formula. This completes the proof of  Lemma \ref{techlem}.
\end{proof}

\subsection{Coercivity estimate in the velocity variable}

In this part, we apply Proposition \ref{upbounded} to    deduce  a coercivity estimate for velocity $v$.

\begin{proposition}\label{energy}
Suppose that the hypothesis of Theorem \ref{hypoellip}  is fulfilled.  Then for any $\eps>0,$ we have
	\begin{equation*}
		\begin{aligned}
		&   \big(\sup_{t\leq T}\norm{f}_{\mathscr{X}_{\ell}}\big)^2  +  \frac{c_0}{2}  \int_0^T \norm{ f}_{\mathcal{Y}_{\ell}}^2dt  \\
	&\leq  \norm{f_{in}}_{\mathscr{X}_{\ell}} ^2+\eps\, \sup_{t\leq T}\norm{g}_{\mathscr{X}_{\ell}}  \int_0^T \| f\|_{\mathcal{Y}_{\ell}} ^2dt 
    +C_\eps\Big( 1+\sup_{t\leq T}\norm{g}_{\mathscr{X}_{\ell}} \Big)^2  \int_0^T\norm{\comi{D_v}^s\comi v^{\boldsymbol{\ell_1}+\frac{\gamma}{2}}f}_{L^2_{x, v}}^2dt\\
	&\quad+ C_\eps\Big( 1+\sup_{t\leq T}\norm{g}_{\mathscr{X}_{\ell}} \Big)^2\int_0^T\norm{\comi v^{s+\frac{\gamma}{2}}f}_{\mathscr{X}_{\ell}}^2dt  
    +C  \Big(\sup_{t\leq T}  \|f\|_{\mathscr{X}_{\ell}}^2\Big)  \int_0^T  \|\langle v\rangle^{s+ \frac{\gamma}{2}}g\|_{\mathscr{X}_{\ell}}^2dt,
		\end{aligned},
	\end{equation*}
    recalling  $\ell,\boldsymbol{\ell_1}$ are given in \eqref{rhod}.
\end{proposition}

\begin{proof}
Taking the $\mathscr{X}_{\ell}$-product with $f$ on both sides of the linear Boltzmann equation \eqref{eq-lin}, we obtain 
\begin{equation} \label{equa}
\frac12\frac{d}{dt}\norm{f}_{\mathscr{X}_{\ell}}^2 
 =\sum_{\abs\alpha=3}\big(\comi v^{\ell-\rho\abs\alpha}\partial_x^{\alpha}Q(g, f),\ \comi v^{\ell-\rho\abs\alpha}\partial_x^{\alpha}f\big)_{L^2_{x, v}} 
 +\big(\comi v^{\ell} Q(g,f),\ \comi v^{\ell } f\big)_{L^2_{x, v}}\, ,
\end{equation}
where we used the fact that
\begin{align*}
	\inner{v\cdot\partial_x f,   f}_{\mathscr{X}_{\ell}}
	=\sum_{\abs\alpha=3}\big( v\cdot\partial_x\comi v^{\ell-\rho\abs\alpha}\partial_x^{\alpha} f ,\ \comi v^{\ell-\rho\abs\alpha}\partial_x^{\alpha}f\big)_{L^2_{x, v}} 
 + \big(v\cdot\partial_x\comi v^{\ell}  f,\ \comi v^{\ell } f\big)_{L^2_{x, v}}=0.
\end{align*}
It follows from the estimate \eqref{coer-v} in Proposition \ref{upbounded} that
\begin{multline*}
	 \big(\comi v^{\ell} Q(g,f),\ \comi v^{\ell } f)_{L^2_{x, v}}\leq -\frac{c_0}{2}\normm{\comi v^\ell f}_{L_x^2}^2+2C_0\norm{\comi v^{\ell+s+\frac{\gamma}{2}}f}_{L^2_{x, v}}^2\\
	 + C\Big( \|\langle v\rangle^{6+\gamma} g\|_{L^{\infty}_{x}L^2_v}^2  \| \langle v\rangle^{\ell+\frac{\gamma}{2}} f\|_{L^2_{x, v}}^2 +    \|\langle v\rangle^{6+\gamma}f\|_{L_x^{\infty}L^2_v}^2\| \langle v\rangle^{\ell+\frac{\gamma}{2}}g\|_{L^2_{x, v}}^2\Big).  
\end{multline*}
This,  together with   \eqref{xl}  as well as the definition of $\norm{\cdot}_{\mathscr{X}_{\ell}}$ in \eqref{xn},  yields
\begin{equation}\label{dec19}
	\begin{aligned}
  &\big(\comi v^{\ell} Q(g,f),\ \comi v^{\ell } f\big)_{L^2_{x, v}} \\
  &\qquad \leq -\frac{c_0}{2}  \normm{\comi v^\ell f}_{L_x^2}^2 +C\inner{1+\| g\|_{\mathscr{X}_{\ell}}^2}\norm{\comi v^{s+\frac{\gamma}{2}}f}_{\mathscr{X}_{\ell}}^2 +C  \|f\|_{\mathscr{X}_{\ell}}^2\|\langle v\rangle^{\frac{\gamma}{2}}g\|_{\mathscr{X}_{\ell}}^2.
	\end{aligned}
\end{equation}
For the first term on the right-hand side of \eqref{equa}, we write
\begin{equation}\label{firstterm}
	\sum_{\abs\alpha=3}\big(\comi v^{\ell-\rho\abs\alpha}\partial_x^{\alpha}Q(g,f),\ \comi v^{\ell-\rho\abs\alpha}\partial_x^{\alpha}f\big)_{L^2_{x, v}}=S_1+S_2,
\end{equation}
where
\begin{equation*}
\left\{
	\begin{aligned}
		S_1&=\sum_{\abs\alpha=3}\big(\comi v^{\ell-\rho\abs\alpha}Q(g,\partial_x^{\alpha}f),\ \comi v^{\ell-\rho\abs\alpha}\partial_x^{\alpha}f\big)_{L^2_{x, v}}\\
		S_2&=\sum_{\abs\alpha=3}\ \sum_{ \beta\leq \alpha,\,\abs\beta\geq 1}\binom{\alpha}{\beta} \big(\comi v^{\ell-\rho\abs\alpha}Q(\partial_x^{\beta}g, \partial_x^{\alpha-\beta}f),\ \comi v^{\ell-\rho\abs\alpha}\partial_x^{\alpha}f\big)_{L^2_{x, v}}.
	\end{aligned}
\right.
\end{equation*}
Following a similar argument as in \eqref{dec19}  yields  
\begin{align*}
	 S_1 
	\leq   -\frac{c_0}{2}\sum_{\abs\alpha=3} \normm{\comi v^{\ell-\rho\abs\alpha}\partial_x^\alpha f}_{L_x^2}^2 
	+C\inner{1+\| g\|_{\mathscr{X}_{\ell}}^2}\norm{\comi v^{s+\frac{\gamma}{2}}f}_{\mathscr{X}_{\ell}}^2 +C  \|f\|_{{X}_{\ell}}^2\|\langle v\rangle^{\frac{\gamma}{2}}g\|_{\mathscr{X}_{\ell}}^2.
\end{align*}
Using the first assertion in Lemma \ref{techlem}  we conclude that, for any $\eps>0$,
\begin{align*}
	  S_2    \leq & \eps \|g\|_{\mathscr{X}_\ell}   \norm{  f}_{\mathcal{Y}_{\ell}}^2 +C_\eps \|g\|_{\mathscr{X}_\ell} \Big(\normm{\comi v^{\boldsymbol{\ell_1}-s}f}_{L_x^2}^2+\|\langle v\rangle^{ s+\frac{\gamma}{2}}f\|_{\mathscr{X}_{\ell}}^2  \Big) \\
	  & +  C\|f\|_{\mathscr{X}_{\ell}}\|\langle v\rangle^{\frac{\gamma}{2}}g\|_{\mathscr{X}_{\ell}}  \|\langle v\rangle^{\frac{\gamma}{2}}f\|_{\mathscr{X}_{\ell}}\\
	  \leq & \eps \|g\|_{\mathscr{X}_\ell}   \norm{  f}_{\mathcal{Y}_{\ell}}^2+C_\eps \|g\|_{\mathscr{X}_\ell} \Big(\normm{\comi v^{\boldsymbol{\ell_1}-s}f}_{L_x^2}^2+\|\langle v\rangle^{ s+\frac{\gamma}{2}}f\|_{\mathscr{X}_{\ell}}^2  \Big) \\
	  & +\frac{c_0}{4}\|\langle v\rangle^{ \frac{\gamma}{2}}f\|_{\mathscr{X}_{\ell}}^2+C  \|f\|_{\mathscr{X}_{\ell}}^2  \|\langle v\rangle^{ \frac{\gamma}{2}}g\|_{\mathscr{X}_{\ell}}^2.
\end{align*}
Substituting these estimates into  \eqref{firstterm}, one has, for any $\eps>0,$
\begin{equation*}
	\begin{aligned}
	&	\sum_{\abs\alpha=3}\big(\comi v^{\ell-\rho\abs\alpha}\partial_x^{\alpha}Q(g,f),\ \comi v^{\ell-\rho\abs\alpha}\partial_x^{\alpha}f\big)_{L^2_{x, v}}\\
	& \leq -\frac{c_0}{2} \sum_{\abs\alpha=3} \normm{\comi v^{\ell-\rho\abs\alpha}\partial_x^\alpha f}_{L_x^2}^2+\eps \|g\|_{\mathscr{X}_\ell}   \norm{  f}_{\mathcal{Y}_{\ell}}^2+\frac{c_0}{4} \|\langle v\rangle^{ \frac{\gamma}{2}}f\|_{\mathscr{X}_{\ell}}^2\\
    &\quad+C_\eps\big(1+ \|g\|_{\mathscr{X}_\ell}^2\big) \Big(\normm{\comi v^{\boldsymbol{\ell_1}-s}f}_{L_x^2}^2+\|\langle v\rangle^{ s+\frac{\gamma}{2}}f\|_{\mathscr{X}_{\ell}}^2  \Big)  +C     \|f\|_{\mathscr{X}_{\ell}}^2 \|\langle v\rangle^{ \frac{\gamma}{2}}g\|_{\mathscr{X}_{\ell}}^2,
	\end{aligned}
\end{equation*}
which, together with  \eqref{dec19}  as well as definition of $\norm{\cdot}_{\mathcal{Y}_{\ell}}$ in \eqref{yn},  yields 
\begin{equation*}
\begin{aligned}
&\sum_{\abs\alpha=3}\big(\comi v^{\ell-\rho\abs\alpha}\partial_x^{\alpha}Q(g,f),\ \comi v^{\ell-\rho\abs\alpha}\partial_x^{\alpha}f\big)_{L^2_{x, v}}+ \big(\comi v^{\ell} Q(g,f),\ \comi v^{\ell } f\big)_{L^2_{x, v}} 
 \\
 &\leq- \frac{c_0}{2}\norm{f}_{\mathcal{Y}_{\ell}}^2+\eps \|g\|_{\mathscr{X}_\ell}   \norm{  f}_{\mathcal{Y}_{\ell}}^2+ \frac{c_0}{4} \|\langle v\rangle^{ \frac{\gamma}{2}}f\|_{\mathscr{X}_{\ell}}^2 \\
    &\quad+C_\eps \big(1+ \|g\|_{\mathscr{X}_\ell}^2\big)  \Big(\normm{\comi v^{\boldsymbol{\ell_1}-s}f}_{L_x^2}^2+\|\langle v\rangle^{ s+\frac{\gamma}{2}}f\|_{\mathscr{X}_{\ell}}^2  \Big)  +C   \|f\|_{\mathscr{X}_{\ell}}^2\|\langle v\rangle^{ \frac{\gamma}{2}}g\|_{\mathscr{X}_{\ell}}^2.
\end{aligned}
\end{equation*}
Combining this with  \eqref{equa} and observing   $\|\langle v\rangle^{ \frac{\gamma}{2}}f\|_{\mathscr{X}_{\ell}}\leq   \norm{f}_{\mathcal{Y}_{\ell}}$, we obtain 
 \begin{equation*}
 \begin{aligned}
\frac12\frac{d}{dt}\norm{f}_{\mathscr{X}_{\ell}}^2 &\leq- \frac{c_0}{4}\norm{f}_{\mathcal{Y}_{\ell}}^2+\eps \|g\|_{\mathscr{X}_\ell}   \norm{  f}_{\mathcal{Y}_{\ell}}^2    \\
    &\quad+C_\eps \big(1+ \|g\|_{\mathscr{X}_\ell}^2\big)  \Big(\normm{\comi v^{\boldsymbol{\ell_1}-s}f}_{L_x^2}^2+\|\langle v\rangle^{ s+\frac{\gamma}{2}}f\|_{\mathscr{X}_{\ell}}^2  \Big)  +C  \|f\|_{\mathscr{X}_{\ell}}^2\|\langle v\rangle^{ \frac{\gamma}{2}}g\|_{\mathscr{X}_{\ell}}^2.
\end{aligned}
\end{equation*}
Moreover,  it follows from \eqref{upbd}  that 
\begin{equation*}
    \normm{\comi v^{\boldsymbol{\ell_1}-s}f}_{L_x^2}\leq C \norm{\comi{D_v}^s \comi v^{s+\frac{\gamma}{2}+\boldsymbol{\ell_1}-s}f}_{L^2_{x, v}}\leq C\norm{\comi{D_v}^s\comi v^{\boldsymbol{\ell_1}+\frac{\gamma}{2}}f}_{L^2_{x, v}}.
\end{equation*}
Hence, combining these estimates yields, for any $\eps>0,$
\begin{equation*}
\begin{aligned}
  \frac12\frac{d}{dt}\norm{f}_{\mathscr{X}_{\ell}}^2\leq & - \frac{c_0}{4}\norm{f}_{\mathcal{Y}_{\ell}}^2+\eps \|g\|_{\mathscr{X}_\ell}   \norm{  f}_{\mathcal{Y}_{\ell}}^2  +C_\eps\big(1+ \|g\|_{\mathscr{X}_\ell}^2\big) \norm{\comi{D_v}^s\comi v^{\boldsymbol{\ell_1}+\frac{\gamma}{2}}f}_{L^2_{x, v}}\\
    & +C_\eps \big(1+ \|g\|_{\mathscr{X}_\ell}^2\big) \|\langle v\rangle^{ s+\frac{\gamma}{2}}f\|_{\mathscr{X}_{\ell}}^2   +C \|f\|_{\mathscr{X}_{\ell}}^2   \|\langle v\rangle^{ \frac{\gamma}{2}}g\|_{\mathscr{X}_{\ell}}^2.
\end{aligned}
\end{equation*}
 Then integrating over $[0,t] $ for any $t\in[0,T]$  we obtain 
the assertion in Proposition \ref{energy}, completing the proof. 
\end{proof}

\subsection{Hypoelliptic estimate for the spatial variable}

This section is devoted to proving the following weighted hypoelliptic estimate in the spatial variable $x$.  
   
\begin{proposition}\label{wsub} 
    There exists   an operator $P: L^2(\mathbb{T}^3\times\mathbb{R}^3)\ \to\ L^2(\mathbb{T}^3\times\mathbb{R}^3) 
$ satisfying
\begin{equation}\label{bofp}
 \norm{Ph}_{L^2_{x, v}}\leq C\norm{h}_{L^2_{x, v}},\quad  	\normm{P h}_{L_x^2}\leq C \normm h_{L_x^2}, 
\end{equation}
such that  
	\begin{multline*}
 	\int_0^T \norm{\comi v^{\frac{\gamma}{2(1+2s)}} \comi {D_x}^{\frac{s}{1+2s}} h}_{L^2_{x, v}}^2 dt\\
        \leq  C\Big(\sup_{t\leq T}\norm{h}_{L^2_{x, v}}\Big)^2 
         +C\int_0^T  \normm{h}_{L_x^2}^2dt+  \int_0^T\big| \big(  \inner{\partial_th+v\cdot\partial_xh},\  Ph\big)_{L^2_{x, v}}\big|	dt.
 \end{multline*}
\end{proposition}

To prove this Proposition, we adopt the strategy from \cite{MR3950012},  which is based on a global pseudo-differential calculus. To illustrate the methodology with minimal use of pseudo-differential calculus, we first present a complete proof of a hypoelliptic estimate without weights, using only Fourier multiplier techniques. The weighted estimate then follows by a similar argument, with the Fourier multiplier replaced by a general pseudo-differential operator.

\begin{proposition}[Hypoellipticity  without weights]\label{subesw}
There exists a  bounded operator $\tilde P: L^2(\mathbb{T}^3\times\mathbb{R}^3)\ \to\ L^2(\mathbb{T}^3\times\mathbb{R}^3)
$ satisfying
\begin{equation}\label{condition2}
\norm{\tilde Ph}_{L^2_{x, v}}\leq 2\norm{h}_{L^2_{x, v}},\quad 	\normm{\tilde P h}_{L^2_x}\leq C \normm h _{L^2_x}, 
\end{equation}
such that  
\begin{equation*}
\begin{split}
	&  \int_0^T \norm{\comi {D_x}^{\frac{s}{1+2s}} h  }_{L^2_{x, v}}^2 dt \\
	&\qquad\qquad \leq 2 \big(\sup_{t\leq T}\norm{h}_{L^2_{x, v}}\big)^2+C \int_0^T \norm{\comi{D_v}^sh }_{L^2_{x, v}}^2dt+ \int_0^T \big| \big( (\partial_t+v\cdot\partial_x)h,\  \tilde Ph\big)_{L^2_{x, v}}\big| dt.
	  \end{split}
\end{equation*}
\end{proposition}

\begin{proof}
Recall $(k,\eta)\in\mathbb Z^3\times\mathbb R^3$ is the Fourier dual variable of $(x,v)\in\mathbb T^3\times\mathbb R^3.$ 
For each $k\in\mathbb Z^3$, define $\varphi_k(\eta)$ by setting 
\begin{equation}\label{lake}
\varphi_k(\eta):=\frac{2k\cdot\eta}{\comi k^{\frac{2+2s}{1+2s}}}\chi
\bigg(\frac{\comi \eta}{\comi k^{\frac{1}{1+2s}}} \bigg),
\end{equation}
where $\chi=\chi(r) \in C_0^{\infty}(\mathbb{R};[0,1])$ such that $\chi=1$ on
$[-1,1]$ and supp $\chi\subset [-2,2].$
Let $\varphi_k(D_v)$ denote the Fourier multiplier associated with the symbol $\varphi_k(\eta)$, defined by
\begin{equation*}
	\mathcal F_v (\varphi_k(D_v) h)(\eta)=\varphi_k(\eta)(\mathcal F_v h) (\eta).
\end{equation*}
Since $|\varphi_k(\eta)| \leq 2$ for all $\eta \in \mathbb{R}^3$,  Plancherel's theorem   yields  \begin{equation}\label{planc}
	 \forall\, k\in\mathbb Z^3,\quad  \norm{\varphi_k(D_v)h}_{L_v^2}\leq    2 \norm{  h}_{L_v^2}.
\end{equation}
Furthermore, the symbol  $\varphi_k(\eta)$ satisfies uniform derivative bounds with respect to $k$, namely,  for any multi-index $\alpha \in \mathbb{Z}_+^3$, there exists a constant $C_\alpha$, depending only on $\alpha$ but independent of $k$, such that 
\begin{equation}\label{unb}
\forall \ k\in\mathbb Z^3,\  \forall\ \eta\in\mathbb R^3,\quad |\partial_\eta^\alpha \varphi_k(\eta)|\leq C_\alpha.
\end{equation}
This yields 
\begin{equation}\label{uppb}
 \forall\, k\in\mathbb Z^3,\quad     \normm{\varphi_k(D_v) h} \leq C  \normm{ h},
\end{equation}
where, here and below, we denote by $C$ a generic constant {\it independent of} $k.$  
The proof of 
\eqref{uppb}  relies on the commutator estimate  between the Fourier multiplier  $\varphi_k(D_v)$ and the fractional Laplacian on sphere $(-\Delta_{\mathbb S^2})^{s/2},$  which requires some techniques related to the pseudo-differential calculus. Interested readers may refer to  \cite{MR3950012} for the commutator estimate and here we omit the proof to avoid the technicality induced by the   pseudo-differential calculus.

Performing a Fourier transform  with respect to $x$ variable gives
\begin{equation*}
	\mathcal F_x (\partial_th+v\cdot\partial_x h)=\inner{\partial_t +i(v\cdot k) } \mathcal F_x h,
\end{equation*}
and thus
\begin{equation}\label{Four-eq}
 i(v\cdot k)  \mathcal F_x h=  -\partial_t  \mathcal F_x h+	\mathcal F_x (\partial_th+v\cdot\partial_x h).
\end{equation}
Since $\varphi_k(D_v)$ is self-adjoint in $L_v^2$ and commutes with time derivatives,  then 
we  take $L^2_v$-product with $\varphi_k(D_v)\mathcal F_x h$ on both sides, to
 obtain  
\begin{equation}\label{lr}
\begin{aligned}
    &\text{Re}\ \big( i(v\cdot k)\mathcal F_x{h},\  \varphi_k(D_v)\mathcal F_x{h} \big)_{L^2_v} \\
    &=- \frac12\frac{d}{dt}    \big(\mathcal F_x{h}, \varphi_k(D_v)\mathcal F_x{h}\big)_{L^2_v} 
+ \text{Re}\ \big( \mathcal F_x (\partial_th+v\cdot\partial_x h),\  \varphi_k(D_v)\mathcal F_x{h} \big)_{L^2_v}.
\end{aligned}
\end{equation} 
 On the other hand,  
 \begin{equation}\label{nose}
 \begin{aligned}
     &\text{Re}\ \big( i(v\cdot k)\mathcal F_x{h},\  \varphi_k(D_v)\mathcal F_x{h} \big)_{L^2_v}\\
     &=-\text{Re}\ \big(  (k\cdot\partial_\eta )\mathcal F_{x,v} h,\  \varphi_k(\eta)\mathcal F_{x,v}{h} \big)_{L^2_\eta} =\frac12\sum_{1\leq j\leq 3}\big(k_j   (\partial_{\eta_j}\varphi_k(\eta))     \mathcal F_{x,v} h,\  \mathcal F_{x,v}{h} \big)_{L^2_\eta}. 
 \end{aligned}
 \end{equation}
  Using \eqref{lake},  a direct computation gives
  \begin{equation}\label{esti}
 \begin{aligned}
   & \frac12 \sum_{1\leq j\leq 3}  k_j   (\partial_{\eta_j}\varphi_k(\eta)) \\
   &= \frac{ |k|^2}{\comi k^{\frac{2+2s}{1+2s}}}\chi
\bigg(\frac{\comi \eta}{\comi k^{\frac{1}{1+2s}}} \bigg)
    +\sum_{1\leq j\leq 3}  k_j \frac{k\cdot\eta}{\comi k^{\frac{2+2s}{1+2s}}}\frac{d\chi}{dr}  
\bigg(\frac{\comi\eta}{\comi k^{\frac{1}{1+2s}}} \bigg)\frac{\eta_j}{\comi \eta\comi k^{\frac{1}{1+2s}}}\\
   & \geq  \frac{ |k|^2}{\comi k^{\frac{2+2s}{1+2s}}}\chi
\bigg(\frac{\comi \eta}{\comi k^{\frac{1}{1+2s}}} \bigg)-C\comi \eta^{2s}\\
& \geq \frac{ |k|^2}{\comi k^{\frac{2+2s}{1+2s}}}- \frac{ |k|^2}{\comi k^{\frac{2+2s}{1+2s}}}\bigg[1-\chi
\bigg(\frac{\comi\eta}{\comi k^{\frac{1}{1+2s}}} \bigg)\bigg]-C\comi \eta^{2s}\geq  \comi k^{\frac{2s}{1+2s}}-C\comi \eta^{2s}, 
 \end{aligned}
 \end{equation} 
 where the first inequality  holds since $\comi\eta\approx \comi k^{\frac{1}{1+2s}}$ on the support of 
 \begin{equation*}
 	\frac{d\chi}{dr} 
\bigg(\frac{\comi\eta}{\comi k^{\frac{1}{1+2s}}} \bigg),
 \end{equation*} 
 and last one follows from the fact that $\comi \eta\gtrsim \comi k^{\frac{1}{1+2s}}$ on the support of 
 \begin{equation*}
 	1-\chi
\bigg(\frac{\comi\eta}{\comi k^{\frac{1}{1+2s}}} \bigg).
 \end{equation*}
 As a result, substituting \eqref{esti} into \eqref{nose} yields   
 \begin{align*}
     \text{Re}\ \big( i(v\cdot k)\mathcal F_x{h},\  \varphi_k(D_v)\mathcal F_x{h} \big)_{L^2_v}&\geq    \norm{\comi k^{\frac{s}{1+2s}}\mathcal F_x{h} }_{L^2_v}^2-C \norm{\comi{D_v}^s\mathcal F_x{h} }_{L^2_v}^2.
 \end{align*}
 Combining this estimate with   \eqref{lr} we obtain
 \begin{equation}\label{B5}
     \begin{aligned}
       &   \norm{\comi k^{\frac{s}{1+2s}}\mathcal F_x{h}(t,k) }_{L^2_v}^2\leq - \frac12\frac{d}{dt}    \big(\mathcal F_x{h}, \varphi_k(D_v)\mathcal F_x{h}\big)_{L^2_v}\\
       &\qquad\qquad \qquad +C  \norm{\comi{D_v}^s\mathcal F_x{h} }_{L^2_v}^2+ \text{Re}\ \big( \mathcal F_x\inner{\partial_th+v\cdot\partial_xh},\  \varphi_k(D_v)\mathcal F_x{h} \big)_{L^2_v}.
     \end{aligned}
 \end{equation}
 Observing that the constant $C$ in \eqref{B5} is independent of $k$, we sum \eqref{B5} over all $k \in \mathbb{Z}^3$ and apply Parseval's identity to get
 \begin{equation*}
 	\begin{split}
 		 \norm{\comi {D_x}^{\frac{s}{1+2s}}h  }_{L^2_{x, v}}^2 
          \leq -  \frac12\frac{d}{dt}    \big(h,  \tilde Ph\big)_{L^2_{x, v}}+C  \norm{\comi{D_v}^sh }_{L^2_{x, v}}^2+   \big| \big( (\partial_t+v\cdot\partial_x)h,\  \tilde Ph\big)_{L^2_{x, v}}\big|,
 	\end{split}
 \end{equation*}
 where the operator $\tilde P$ is defined by
 \begin{equation*}
 	\mathcal F_{x,v} (\tilde P h)=\varphi_k(\eta)\mathcal F_{x,v} h
 \end{equation*}
 with $\varphi_k(\eta)$ given in \eqref{lake}.  
 It follows from \eqref{planc} and \eqref{uppb} that $\tilde P$ satisfies condition   \eqref{condition2}.  Integrating the above inequality over $[0,T]$ and using \eqref{condition2}, we obtain  the assertion in Proposition \ref{subesw},  completing 
 the proof. 
\end{proof}

\begin{proof}[Sketch of the proof of Proposition \ref{wsub}]  To establish the weighted hypoelliptic estimate in Proposition \ref{wsub}, 
we follow the argument in the proof of Proposition  \ref{subesw}. The key difference is that we replace the Fourier multiplier used there with a pseudo-differential operator.  Precisely, corresponding to $\varphi_k(\eta)$ in \eqref{lake}, we define the symbol $\psi_{k}(v,\eta)$  by
\begin{equation}
	\label{newsy}
\psi_{k}(v,\eta)=\frac{2\comi v^{\frac{\gamma}{1+2s}} k\cdot\eta}{\comi k^{\frac{2+2s}{1+2s}}}\chi
\bigg(\frac{\comi v^{\frac{\gamma}{1+2s}} \comi \eta}{\comi k^{\frac{1}{1+2s}}} \bigg),
\end{equation}
recalling  $\chi=\chi(r) \in C_0^{\infty}(\mathbb{R};[0,1])$ such that $\chi=1$ on
$[-1,1]$ and supp $\chi\subset [-2,2].$  Furthermore, associated with $\psi_{k}$, we define the operator $\psi_{k}^{\rm Wick}$ by setting 
\begin{equation}\label{ww}
	\psi_{k}^{\rm Wick} h(v)=\int_{\mathbb{R}^6} e^{i(v-\tilde v)\eta} A_{\psi_{k}}\Big(\frac{v+\tilde v}{2},\eta\Big ) h(\tilde v) d\tilde vd\eta,
\end{equation}
where 
\begin{equation*}
	A_{\psi_{k}}(v,\eta)= \frac{1}{8\pi^{6}} \psi_{k}*e^{- \abs v^2- \abs \eta^2}= \frac{1}{8\pi^{6}}\int_{\mathbb{R}^6}\psi_{k}(v-z,\eta-\zeta) e^{- \abs z^2- \abs \zeta^2}dzd\zeta.
\end{equation*}
The operator $\psi_{k}^{\rm Wick}$ is called the Wick quantization of the symbol $\psi_{k}(v,\eta).$  Interested readers may refer to \cite{MR2599384}  for the  detailed discussion of the Wick quantization.  Compared with other classical quantizations of symbols,  a typical  feature of the Wick quantization is its preservation of positivity: namely, if a symbol $q(v,\eta)$ is nonnegative on $\mathbb{R}^{6}$, then its Wick quantization $q^{\mathrm{Wick}}$ is a nonnegative operator. That is,
\begin{equation}\label{positi}
  q(v,\eta)\geq 0 ~~\,\textrm{for all}~ (v,\eta)\in\mathbb R^{6} ~\ {\rm implies}
~~\,q^{\rm Wick}\geq 0.
\end{equation}
Moreover,   $q^{\rm Wick}$ is self-adjoint and bounded on $L^2$ if $q$ is a bounded real-valued symbol, and  the commutator between two Wick quantizations satisfies that
\begin{equation}\label{cot}
[q_1^{\rm Wick}, \ q_2^{\rm Wick}]	=  \frac{1}{i}\big\{q_1, q_2\big\} ^{\rm Wick},
\end{equation}
where the notation $\set{q_1,q_2}$ denotes the Poisson
bracket defined by
\begin{equation}\label{11051505}
  \big\{q_1,~q_2\big\}=\sum_{1\leq j\leq 3}\bigg(\frac{\partial q_1}{\partial\eta_j} \frac{\partial q_2}{\partial
v_j}- \frac{\partial q_1}{\partial v_j} \frac{\partial q_2}{\partial \eta_j}\bigg).
\end{equation}
Similar to \eqref{unb}, we can verify directly that  for any multi-indices $\alpha,\beta \in \mathbb{Z}_+^3$, there exists a constant $C_{\alpha,\beta}$, depending only on $\alpha,\beta,$ but independent of $k$, such that 
\begin{equation*}
\forall \ k\in\mathbb Z^3,\  \forall\ (v,\eta)\in\mathbb R^3\times\mathbb R^3, \quad |\partial_v^{\alpha}\partial_\eta^\beta \psi_{k}(v,\eta)|\leq C_{\alpha,\beta}.
\end{equation*}
This yields that, similar to \eqref{planc} and \eqref{uppb},
\begin{equation}\label{+uppb}
\forall\  k\in\mathbb Z^3,\quad  \norm{\psi_{k}^{\rm Wick} h}_{L_v^2}\leq C   \norm{ h}_{L_v^2},\quad    \normm{\psi_{k}^{\rm Wick} h} \leq C  \normm{ h},
\end{equation}
where,  here and below, the constant $C$ is {\it independent of } $k.$

Since $\psi_{k}^{\mathrm{Wick}}$ is self-adjoint on $L_v^2$, we proceed analogously to  \eqref{lr} by taking the $L^2_v$ inner product of both sides of  \eqref{Four-eq}  with $\psi_{k}^{\mathrm{Wick}}\mathcal{F}_x h$, which yields
\begin{equation}\label{lr2}
\begin{aligned}
    &\text{Re}\ \big( i(v\cdot k)\mathcal F_x{h},\  \psi_{k}^{\rm Wick}\mathcal F_x{h} \big)_{L^2_v} \\
    &=- \frac12\frac{d}{dt}    \big(\mathcal F_x{h}, \psi_{k}^{\rm Wick}\mathcal F_x{h}\big)_{L^2_v} 
+ \text{Re}\ \big( \mathcal F_x (\partial_th+v\cdot\partial_x h),\ \psi_{k}^{\rm Wick}\mathcal F_x{h} \big)_{L^2_v}.
\end{aligned}
\end{equation} 
Observe that by \eqref{ww},
 \[
    v\cdot k= \inner{v\cdot k}^{\rm Wick}.
  \]
 Thus
 \begin{multline*}
 	 \text{Re}\ \big( i(v\cdot k)\mathcal F_x{h},\  \psi_{k}^{\rm Wick}\mathcal F_x{h} \big)_{L^2_v}=\text{Re}\ \big( i(v\cdot k)^{\rm Wick}\mathcal F_x{h},\  \psi_{k}^{\rm Wick}\mathcal F_x{h} \big)_{L^2_v}\\
 	 =\frac{1}{2}\big( i[\psi_{k}^{\rm Wick},\ (v\cdot k)^{\rm Wick}]\mathcal F_x{h},\  \mathcal F_x{h} \big)_{L^2_v}=\frac{1}{2}\big( \big\{\psi_{k},\  v\cdot k\big\}^{\rm Wick}\mathcal F_x{h},\  \mathcal F_x{h} \big)_{L^2_v},
 	\end{multline*}
 	where the  second identity follows from the fact that $\psi_{k}^{\rm Wick}$ is self-adjoint on $L_v^2$ and the last one uses \eqref{cot}. 
Using  \eqref{11051505} and \eqref{newsy}, we follow the same argument as that in \eqref{esti} to compute
\begin{equation*}
	\begin{aligned}
   &\frac12\big\{\psi_{k},\  v\cdot k\big\}=  \sum_{1\leq j\leq 3} k_j   (\partial_{\eta_j}\psi_{k}(\eta)) \\
   &= \frac{ \comi v^{\frac{\gamma}{1+2s}} |k|^2}{\comi k^{\frac{2+2s}{1+2s}}}\chi
\bigg(\frac{\comi v^{\frac{\gamma}{1+2s}} \comi \eta}{\comi k^{\frac{1}{1+2s}}} \bigg)
    +\sum_{j}  k_j \frac{\comi v^{\frac{\gamma}{1+2s}}  k\cdot\eta}{\comi k^{\frac{2+2s}{1+2s}}}\frac{d\chi}{dr} 
\bigg(\frac{\comi v^{\frac{\gamma}{1+2s}} \comi\eta}{\comi k^{\frac{1}{1+2s}}} \bigg)\frac{\comi v^{\frac{\gamma}{1+2s}}  \eta_j}{\comi \eta\comi k^{\frac{1}{1+2s}}}\\
   & \geq  \frac{\comi v^{\frac{\gamma}{1+2s}}  |k|^2}{\comi k^{\frac{2+2s}{1+2s}}}\chi
\bigg(\frac{\comi v^{\frac{\gamma}{1+2s}}\comi \eta}{\comi k^{\frac{1}{1+2s}}} \bigg)-C\comi v^{\gamma} \eta^{2s}\\
& \geq \frac{ \comi v^{\frac{\gamma}{1+2s}}|k|^2}{\comi k^{\frac{2+2s}{1+2s}}}- \frac{ \comi v^{\frac{\gamma}{1+2s}}|k|^2}{\comi k^{\frac{2+2s}{1+2s}}}\bigg[1-\chi
\bigg(\frac{\comi v^{\frac{\gamma}{1+2s}}
\comi\eta}{\comi k^{\frac{1}{1+2s}}} \bigg)\bigg]-C\comi v^{\gamma} \comi \eta^{2s}\\
&\geq  \comi v^{\frac{\gamma}{1+2s}} \comi k^{\frac{2s}{1+2s}}-C\comi v^{\gamma} \comi \eta^{2s}.
 \end{aligned}
\end{equation*}
Combining the above estimates and using \eqref{positi} yields that 
\begin{align*}
	&\text{Re}\ \big( i(v\cdot k)\mathcal F_x{h},\  \psi_{k}^{\rm Wick}\mathcal F_x{h} \big)_{L^2_v}=\frac{1}{2}\big( \big\{\psi_{k},\  v\cdot k\big\}^{\rm Wick}\mathcal F_x{h},\  \mathcal F_x{h} \big)_{L^2_v}\\
	&\geq  \big( \big( \comi v^{\frac{\gamma}{1+2s}} \comi k^{\frac{2s}{1+2s}}\big)^{\rm Wick}\mathcal F_x{h},\  \mathcal F_x{h} \big)_{L^2_v}- C\big( \big( \comi v^{\gamma} \comi \eta^{2s}\big)^{\rm Wick}\mathcal F_x{h},\  \mathcal F_x{h} \big)_{L^2_v}\\
	&\geq \norm{ \comi v^{\frac{\gamma}{2(1+2s)}} \comi k^{\frac{s}{1+2s}} \mathcal F_x{h} }_{L^2_v}^2-C\norm{\comi{D_v}^s\comi v^{\frac{\gamma}{2}}\mathcal F_x{h}}_{L_v^2}^2,
\end{align*}
where the last inequality holds since 
 $\big( \comi v^{\frac{\gamma}{1+2s}}  \big)^{\rm Wick}\geq   \comi v^{\frac{\gamma}{1+2s}}  $ in the sence of operators   and 
\begin{align*}
	\big( \comi v^{\gamma} \comi \eta^{2s}\big)^{\rm Wick} 	= \comi v^{\frac{\gamma}{2}}  \comi{D_v}^{s} \underbrace{\comi{D_v}^{-s} \comi v^{-\frac{\gamma}{2}}\big( \comi v^{\gamma} \comi \eta^{2s}\big)^{\rm Wick} \comi v^{-\frac{\gamma}{2}} 
	\comi{D_v}^{-s} }_{\textrm{bounded  on }  L_v^2 \ \textrm{by classical pseudo-differential calculus}}
	\comi{D_v}^s\comi v^{\frac{\gamma}{2}}. 
\end{align*}
Substituting the above estimate into   \eqref{lr2},
we obtain 
\begin{multline*}
           \norm{\comi v^{\frac{\gamma}{2(1+2s)}} \comi k^{\frac{s}{1+2s}} \mathcal F_x{h}(t,k) }_{L^2_v}^2 
        \leq - \frac12\frac{d}{dt}    \big(\mathcal F_x{h}, \ \psi_{k}^{\rm Wick}\mathcal F_x{h}\big)_{L^2_v}\\
        +C  \norm{\comi{D_v}^s\comi v^{\frac{\gamma}{2}}\mathcal F_x{h} }_{L^2_v}^2+ \text{Re}\ \big( \mathcal F_x\inner{\partial_th+v\cdot\partial_xh},\  \psi_{k}^{\rm Wick}\mathcal F_x{h} \big)_{L^2_v}.
     \end{multline*}
As a result,    taking summation for $k\in\mathbb Z^3$ and  using Parseval' Theorem yields 
\begin{multline*}
  \norm{\comi v^{\frac{\gamma}{2(1+2s)}} \comi {D_x}^{\frac{s}{1+2s}} h}_{L^2_{x, v}}^2 
        \leq -\frac12  \frac{d}{dt}    \big(h, \ Ph\big)_{L^2_{x, v}}\\
        +C  \norm{\comi v^{\frac{\gamma}{2}}\comi{D_v}^sh}_{L^2_{x, v}}^2+  \big| \big(  \inner{\partial_th+v\cdot\partial_xh},\  Ph\big)_{L^2_{x, v}}\big|,
     \end{multline*}
     where the operator $P$ is defined as 
     \begin{equation*}
     	\mathcal F_x(Ph)(t,k,v)=( \psi_{k}^{\rm Wick}\mathcal F_xh)(t,k,v).
     \end{equation*}
   Moreover, it follows from  \eqref{+uppb} that  \eqref{bofp} is fulfilled by the operator $P$.  Integrating the above estimate over $[0,T]$, then  using \eqref{bofp} and   
 the fact that $\norm{\comi{D_v}^s\comi v^{\frac{\gamma}{2}}h}_{L^2_{x, v}}\leq \normm{h}_{L_x^2}$, we obtain the desired estimate in Proposition \ref{wsub}, completing the proof. 
\end{proof}

\subsection{Completing the proof of Theorem  \ref{hypoellip}}

Now we  combine the coercivity estimate in $v$ and the hypoelliptic estimate in $x$, to prove   Theorem \ref{hypoellip}.  For any  given $\alpha\in\mathbb Z^3$ with $\abs\alpha=3$ or $0$,    applying  the operator $\langle v\rangle^{\ell-\rho\abs\alpha}\partial_x^\alpha$ to equation \eqref{eq-lin}  yields 
\begin{equation*}
	(\partial_t+v\cdot\partial_x)\comi v^{\ell-\rho\abs\alpha}\partial_x^\alpha f =\comi v^{\ell-\rho\abs\alpha}\partial_x^\alpha Q(g, f).
\end{equation*}
Applying the hypoelliptic estimate in Proposition \ref{wsub} to the above  equation, we obtain
\begin{equation*}
	\begin{aligned}
	&   \int_0^T\norm{\comi v^{\frac{\gamma}{2(1+2s)}} \comi {D_x}^{\frac{s}{1+2s}} \comi v^{\ell-\rho\abs\alpha}\partial_x^\alpha f  }_{L^2_{x, v}}^2 dt \leq   C\Big(\sup_{t\leq T}  \norm{\comi v^{\ell-\rho\abs\alpha}\partial_x^\alpha f}_{L^2_{x, v}}\Big)^2\\
     &   \qquad\quad+C \int_0^T \normm{\comi v^{\ell-\rho\abs\alpha}\partial_x^\alpha f }_{L_x^2}^2dt+ \int_0^T\big| \big( \comi v^{\ell-\rho\abs\alpha}\partial_x^\alpha Q(g, f),\  P\comi v^{\ell-\rho\abs\alpha}\partial_x^\alpha f\big)_{L^2_{x, v}}\big|dt,	
	\end{aligned}
\end{equation*}
where $P: L^2(\mathbb T^3\times\mathbb R^3)\mapsto L^2(\mathbb T^3\times\mathbb R^3)$  is the bounded operator given  in Proposition  \ref{wsub}. 
The validity of the above estimate for $\abs\alpha=3$ or $0$, together with the definition of $\norm{\cdot}_{\mathcal{Y}_{\ell}}$  and $\norm{\cdot}_{\mathcal{Z}_{\ell}}$ in \eqref{yn} and \eqref{zn}, implies
\begin{equation*}
	\begin{aligned}
 \int_0^T\norm{ f  }_{\mathcal{Z}_{\ell}}^2dt \leq	&  C  \Big(\sup_{t\leq T}  \norm{ f}_{\mathscr{X}_{\ell}}\Big)^2+C\int_0^T \norm{ f }_{\mathcal{Y}_{\ell}}^2 dt  +    \int_0^T\big| \big( \comi v^{\ell}  Q(g, f),\  P\comi v^{\ell} f\big)_{L^2_{x, v}}\big|dt\\
	& +    \sum_{\abs\alpha=3}  \int_0^T\big| \big( \comi v^{\ell-\rho\abs\alpha}\partial_x^\alpha Q(g, f), P\comi v^{\ell-\rho\abs\alpha}\partial_x^\alpha f\big)_{L^2_{x, v}}\big|dt. 
	\end{aligned}
\end{equation*} 
Moreover,   using \eqref{bofp} and the last estimate in Lemma \ref{techlem}   gives 
\begin{multline*}
	 \int_0^T\big| \big( \comi v^{\ell}  Q(g, f),\  P\comi v^{\ell} f\big)_{L^2_{x, v}}\big|dt+\sum_{\abs\alpha=3} \big| \big( \comi v^{\ell-\rho\abs\alpha}\partial_x^\alpha Q(g, f),\  P\comi v^{\ell-\rho\abs\alpha}\partial_x^\alpha f\big)_{L^2_{x, v}}\big|\\
 \leq C    \| g\|_{\mathscr{X}_{\ell}}  \norm{  f}_{\mathcal{Y}_{\ell}} ^2  +C   \| g\|_{\mathscr{X}_{\ell}}    \|\langle v\rangle^{ s+\frac{\gamma}{2}}f\|_{\mathscr{X}_{\ell}}^2+  C  \| f\|_{\mathscr{X}_{\ell}}  \norm{\comi v^{\frac{\gamma}{2}}g}_{\mathscr{X}_{\ell}}  \norm{\comi v^{\frac{\gamma}{2}}f}_{\mathscr{X}_{\ell}}.
\end{multline*} 
Combining the above two inequalities, we conclude that
\begin{multline*}
	  \int_0^T\norm{f  }_{\mathcal{Z}_{\ell}}^2 dt   \leq        C\Big(\sup_{t\leq T}  \norm{f}_{\mathscr{X}_{\ell}}\Big)^2 + C  \int_0^T \norm{ f }_{\mathcal{Y}_{\ell}}^2dt+  C     \sup_{t\leq T}\| g\|_{\mathscr{X}_{\ell}} \int_0^T  \norm{  f}_{\mathcal{Y}_{\ell}} ^2  dt
	  \\
	   +  C    \Big(1+ \sup_{t\leq T}\| g\|_{\mathscr{X}_{\ell}} \Big)\int_0^T   \|\langle v\rangle^{ s+\frac{\gamma}{2}}f\|_{\mathscr{X}_{\ell}}^2 dt+ C \Big(\sup_{t\leq T} \| f\|_{\mathscr{X}_{\ell}} \Big)^2 \int_0^T\norm{\comi v^{\frac{\gamma}{2}}g}_{\mathscr{X}_{\ell}}^2dt.	
	\end{multline*}
Letting $C$ be the constant   above   that    depends only on $\gamma,s$ and $\ell,$ and leting $\boldsymbol{L}$ be the constant defined as in \eqref{A-0},  we define    
\begin{equation*}
a_0:=\frac{c_0}{2  C +16C\boldsymbol{L}}, 
\end{equation*}
and then conclude that 
\begin{multline*}
	 \frac{a_0}{4} \int_0^T\norm{f(t)}_{\mathcal{Z}_{\ell}}^2 dt  
	    \leq        \frac1 2\Big( \sup_{t\leq T}  \norm{f}_{\mathscr{X}_{\ell}}^2\Big) +  \frac{c_0}{4} \int_0^T \norm{ f }_{\mathcal{Y}_{\ell}}^2dt+\frac{c_0\sup_{t\leq T}\| g\|_{\mathscr{X}_{\ell}}}{64\boldsymbol{L}}      \int_0^T  \norm{  f}_{\mathcal{Y}_{\ell}} ^2  dt
	  \\
	 +     \Big(1+ \sup_{t\leq T}\| g\|_{\mathscr{X}_{\ell}}\Big) \int_0^T   \|\langle v\rangle^{ s+\frac{\gamma}{2}}f\|_{\mathscr{X}_{\ell}}^2 dt+  \Big( \sup_{t\leq T} \| f\|_{\mathscr{X}_{\ell}}  \Big)^2\int_0^T\norm{\comi v^{\frac{\gamma}{2}}g}_{\mathscr{X}_{\ell}} ^2dt,	
	\end{multline*}
 which  with Proposition \ref{energy} yields
 \begin{equation*}
		\begin{aligned}
		& \Big(\sup_{t\leq T}\norm{f}_{\mathscr{X}_{\ell}} \Big)^2 + \frac{c_0}{2} \int_0^T   \norm{ f}_{\mathcal{Y}_{\ell}}^2 dt +  \frac{a_0}{2}   \int_0^T  \norm{ f}_{\mathcal{Z}_{\ell}}^2 dt\\
	&\leq 2\norm{f_{in}}_{\mathscr{X}_{\ell}}^2+ \frac{c_0 }{16\boldsymbol{L}} \sup_{t\leq T}\norm{g}_{\mathscr{X}_{\ell}}  \int_0^T \| f\|_{\mathcal{Y}_{\ell}} ^2dt+C \Big[ 1+\sup_{t\leq T}\norm{g}_{\mathscr{X}_{\ell}} \Big]^2\int_0^T\norm{\comi{D_v}^s\comi v^{\boldsymbol{\ell_1}+\frac{\gamma}{2}}f}_{L^2_{x, v}}^2dt  
\\
	&\quad+C \Big( 1+\sup_{t\leq T}\norm{g}_{\mathscr{X}_{\ell}} \Big)^2\int_0^T\norm{\comi v^{s+\frac{\gamma}{2}}f}_{\mathscr{X}_{\ell}}^2dt  +C  \Big(\sup_{t\leq T}  \|f\|_{\mathscr{X}_{\ell}}\Big)^2  \int_0^T  \|\langle v\rangle^{s+ \frac{\gamma}{2}}g\|_{\mathscr{X}_{\ell}}^2dt.
		\end{aligned}
	\end{equation*}
This completes the proof of  Theorem \ref{hypoellip}.  


  \section{\emph{\`A priori} estimates}\label{s:priori}

This section is devoted to establishing an  {\it  \`a priori } estimate for  the linear Boltzmann equation  
\begin{equation}\label{eq-linearc}
	\partial_tf + v\cdot \partial_x f
		= Q(g,f),  \quad f|_{t=0}=f_{{in}},
\end{equation}
where $g\geq 0$ is a given function,  and the initial datum $f_{in}\geq 0$  comes from the nonlinear Cauchy problem \eqref{eq-1}, satisfying  \eqref{finite} and \eqref{initial-datum}.  The argument presented here provides    an  $\epsilon$-uniform estimate  for the following   
regularized  equation (with $0<\epsilon\ll 1$),     
\begin{equation}\label{appbol}
	 \big(\partial_t  +v\cdot\partial_x\big)f_\epsilon+\epsilon \comi v^{s+2\gamma} f_\epsilon=Q(g, f_\epsilon),\quad f_\epsilon|_{t=0}=f_{in},
\end{equation}  
 whose existence is well established. 
 
\begin{definition}\label{defspace}
Let $\ell_0$ and $\ell$ be given in \eqref{rhod}, and let $\boldsymbol{L}$ be defined as in \eqref{A-0}.   
For each given $T>0$ we define $\mathcal{M}_T(\boldsymbol{L})$ by 
\begin{equation*}
\begin{aligned}
  \mathcal{M}_T(\boldsymbol{L}):=\Big\{h ; \  \   \sup_{t\leq T}\|\langle v\rangle^{\ell_0} h\|_{L^1_{x, v}}+\frac{m_0\boldsymbol{\lambda}_{\ell_0}}{4^{1+\gamma}}\int_0^T\|\langle v\rangle^{\ell_0+\gamma} h\|_{L^1_{x, v}}dt\leq \boldsymbol{L}
\Big\},
\end{aligned}
\end{equation*}
and \begin{equation*}
\begin{split}
  \mathcal{E}_{T} (\boldsymbol{L}):=\Big\{ h;  \ \    \Big(\sup_{t\leq T}	\|h\|_{\mathscr{X}_{\ell}}\Big)^2+  c_0  \int_0^T  \norm{h}^2_{\mathcal{Y}_{\ell}}dt+a_0\int_0^T\norm{h}^2_{\mathcal{Z}_{\ell}}dt  \leq \boldsymbol{L}^2
 \Big\},
\end{split}
\end{equation*}
where $m_0$, $\boldsymbol{\lambda}_{\ell_0}$ are the parameters  defined in \eqref{finite} and \eqref{lamome}, and $c_0$ is given in  \eqref{lower+} and $a_0$ is the constant specified  in Theorem \ref{hypoellip}.
\end{definition}

 \begin{theorem}\label{thm:ae}
	Suppose that the hypothesis of Theorem \ref{local-exsitence} is fulfilled. Then there exist a time $T_0>0$ and a constant $ \delta_0>0$,   depending only on $s,\gamma$, such that the following property holds. Let  $g\geq 0$ satisfy condition $\boldsymbol{(H)}$ and  belong  to $\mathcal{M}_{T_0}(\boldsymbol{L})\cap \mathcal{E}_{T_0}(\boldsymbol{L})$. If $f\in L^\infty([0,T_0], \mathscr{X}_\ell)\cap L^2([0,T_0], \mathcal{Y}_\ell)$ is any solution to the linear  Cauchy problem \eqref{eq-linearc},  satisfying  that $f\in C^\infty([\tau, T_0]\times \mathbb T_x^3\times \mathbb R_v^3)$ for any  $0<\tau\leq T_0$ and that   
  \begin{equation}\label{smal}
     \bigg( \int_0^{T_0}\norm{ \comi {D_v}^s\comi v^{ \boldsymbol{\ell_1} +\frac{\gamma}{2}} f}_{L^2_{x, v}}^2 dt\bigg)^{\frac12}\leq \delta_0
  \end{equation}
  with $\boldsymbol{\ell_1}$  given in \eqref{rhod}, then $f\geq 0$ satisfies condition $\boldsymbol{(H)}$ and lies in $ \mathcal{M}_{T_0}(\boldsymbol{L})\cap \mathcal{E}_{T_0}(\boldsymbol{L}).$
\end{theorem}

\begin{remark} 
We impose the $C^\infty$-smoothness in $(t,x,v)$  for $t>0$,   to ensure that the subsequent estimates are mathematically rigorous, not merely formal. However, this assumption can be removed,  since we will work not with \eqref{eq-linearc} directly but with its regularized counterpart \eqref{appbol},  
 which  admits a well-established  instantaneous regularization effect.  See  Section \ref{s:main} for details.   
\end{remark}

The remainder of this section is devoted to the proof of Theorem \ref{thm:ae}.  We will proceed  to establish required {\em non-negativity, moment estimates  and energy estimates}.

\begin{lemma}[Technical lemma]\label{lem:te}
Suppose $g\geq 0$ satisfies condition $\boldsymbol{(H)}$. Then there exists a constant $0<\tilde c_0\leq \frac14,$ depending only on $c_0, C_0, m_0$ amd $M_0 $, such that 
	\begin{equation}\label{fras}
		\big(\comi v^{ \boldsymbol{\ell_1} }Q(g, h), \langle v\rangle^{ \boldsymbol{\ell_1} }h\big)_{L^2_{x, v}}
	 \leq -\tilde c_0 \norm{\comi{D_v}^s\comi v^{ \boldsymbol{\ell_1} +\frac{\gamma}{2}} h}_{L^2_{x, v}}^2    +C\big(1+\norm{g}_{\mathscr{X}_\ell}^2\big) \|\langle v\rangle^{  \boldsymbol{\ell_1} }h\|_{L^2_{x, v}}^2.
	\end{equation}
\end{lemma}

\begin{proof}
 By 
  the pre-postcollisional change of variables  \eqref{prepost},
\begin{align*}
 &\big(\comi v^{ \boldsymbol{\ell_1}}Q(g, h), \langle v\rangle^{ \boldsymbol{\ell_1}}h\big)_{L^2_{x, v}}  
  =\int_{\mathbb{T}^3\times\mathbb{R}^6\times\mathbb{S}^2} |v-v_*|^\gamma b(\cos\theta)   \big(g_*'h'-g_*h\big)\langle v\rangle^{2 \boldsymbol{\ell_1}}h \,\,d\sigma\, dv_*\, dv\,dx\\
 & \qquad\qquad\qquad\qquad=\int_{\mathbb{T}^3\times\mathbb{R}^6\times\mathbb{S}^2}  |v-v_*|^\gamma b(\cos\theta)   g_*\Big(hh '\langle v'\rangle^{2 \boldsymbol{\ell_1}}-\langle v\rangle^{2 \boldsymbol{\ell_1}}|h|^2\Big)\,d\sigma\, dv_*\, dv\,dx.
\end{align*}
 This, with the fact that $  g\ge0$ and
 $ |h\, h'|\leq \frac{1}{2}\big(|h|^2+|h'|^2\big),$ 
yields 
\begin{equation}\label{spt}
 \begin{aligned}
  &\big(\langle v\rangle^{ \boldsymbol{\ell_1} }Q(g,h), \langle v\rangle^{ \boldsymbol{\ell_1} }h\big)_{L^2_{x, v}}\\
  &\leq   \frac{1}{2}\int_{\mathbb{T}^3\times\mathbb{R}^6\times\mathbb{S}^2}  |v-v_*|^\gamma b(\cos\theta)  g_*|h|^2\big(\langle v'\rangle^{2 \boldsymbol{\ell_1}}-\langle v\rangle^{2 \boldsymbol{\ell_1}}\big)\,d\sigma\, dv_*\, dv\,dx\\
 &\quad +\frac{1}{2}\int_{\mathbb{T}^3\times\mathbb{R}^6\times\mathbb{S}^2}  |v-v_*|^\gamma b(\cos\theta) g_*\big(| h'|^2\langle v'\rangle^{2 \boldsymbol{\ell_1}}-|h|^2\langle v\rangle^{2 \boldsymbol{\ell_1}}\big)d\sigma dv_* dv dx. 
 \end{aligned}
 \end{equation}
 For the first term on the right-hand side, we use Lemma \ref{lem:cru} and  the estimate 
\begin{equation}\label{H++++}
 \frac{m_0}{2}\leq \norm{g(t,x,\cdot)}_{L_v^1}\leq 2M_0,
\end{equation} 
 which holds because 
$g$ satisfies condition $\boldsymbol{(H)}$. This yields 
\begin{align*}
	& \int_{\mathbb{T}^3\times\mathbb{R}^6\times\mathbb{S}^2}  |v-v_*|^\gamma b(\cos\theta)  g_*|h|^2\big(\langle v'\rangle^{2 \boldsymbol{\ell_1}}-\langle v\rangle^{2 \boldsymbol{\ell_1}}\big)\,d\sigma\, dv_*\, dv\,dx
	 \\
       & \leq - \frac{m_0\boldsymbol{\lambda}_ {2 \boldsymbol{\ell_1}}}{4^{1+ \gamma}}   \norm{\comi v^{ \boldsymbol{\ell_1} +\frac{\gamma}{2}}h}_{L^2_{x, v}}^2+2\boldsymbol{\omega}_{ 2 \boldsymbol{\ell_1} }\| h\|_{L^2_{x, v}}^2\|\comi v^{2 \boldsymbol{\ell_1} + \gamma }g\|_{L_x^\infty L_v^1} \\
       &\quad+C\|\comi v^{2+\frac{\gamma}{2}}h\|_{L^2}^2\|\comi v^{2  \boldsymbol{\ell_1}}g\|_{L_x^\infty L_v^1}+C\|\comi v^{4+\gamma}g\|_{L_x^\infty L^1_v}\|\comi v^{ \boldsymbol{\ell_1}}h\|_{L^2_{x, v}}^2\\
       & \leq - \frac{m_0 \boldsymbol{\lambda}_ {2 \boldsymbol{\ell_1}}}{4^{1+\gamma}}   \norm{\comi v^{ \boldsymbol{\ell_1} +\frac{\gamma}{2}}h}_{L^2_{x, v}}^2+C\norm{g}_{\mathscr{X}_\ell}\norm{\comi v^{   \boldsymbol{\ell_1}}h}_{L^2_{x, v}}^2,
\end{align*}
where the last inequality holds since it follows from    \eqref{rhod}   that 
$$
\norm{h}_{L_{x,v}^2}\leq\norm{\langle v\rangle^{2+\frac{\gamma}{2}}h}_{L^2_{x, v}}\leq \|\langle v\rangle^{ \boldsymbol{\ell_1}  }h\|_{L^2_{x, v}} 
$$
and  
\begin{equation}\label{gconh}
\begin{aligned}
	 \|\langle v\rangle^{4+\gamma}g\|_{L_x^\infty L^1_v}&\leq \|\langle v\rangle^{ 2 \boldsymbol{\ell_1} +\gamma}g\|_{L_x^\infty L_v^1} \leq C\|\langle v\rangle^{ 2 \boldsymbol{\ell_1} +\gamma+ 2}g\|_{H_x^2 L_v^2}\leq C\|\langle v\rangle^{\ell-3\rho}g\|_{H_x^2 L_v^2}\\
    &\leq  C\sum_{\abs\alpha=3}\|\langle v\rangle^{\ell-\rho\abs\alpha}\partial_x^\alpha g\|_{L^2_{x, v}}+C\|\langle v\rangle^{\ell}g\|_{L^2}\leq C\norm{g}_{\mathscr{X}_\ell}.
    \end{aligned}
\end{equation}
For the last term on the right-hand side of \eqref{spt},  it follows from the cancellation lemma \eqref{cance} that 
\begin{align*}
	&\int_{\mathbb{T}^3\times\mathbb{R}^6\times\mathbb{S}^2}  |v-v_*|^\gamma b(\cos\theta) g_*\big(| h'|^2\langle v'\rangle^{2 \boldsymbol{\ell_1}}-|h|^2\langle v\rangle^{2 \boldsymbol{\ell_1}}\big)\,d\sigma\, dv_*\, dv\,dx\\
	&=\boldsymbol{A}_{\gamma,s}\int_{\mathbb{T}^3\times\mathbb{R}^6}   |v-v_*|^\gamma  g_* h^2\langle v\rangle^{2 \boldsymbol{\ell_1}}   dv_*\,dv\,dx\\
	&\leq 2^\gamma\boldsymbol{A}_{\gamma,s}\norm{g}_{L_x^\infty L_v^1} \norm{\comi v^{ \boldsymbol{\ell_1}+\frac{\gamma}{2}}h}_{L^2_{x, v}}^2+2^\gamma\boldsymbol{A}_{\gamma,s}\norm{\comi v^{\gamma} g}_{L_x^\infty L_v^1} \norm{\comi v^{ \boldsymbol{\ell_1}}h}_{L^2_{x, v}}^2\\
	&\leq 2^{1+\gamma}\boldsymbol{A}_{\gamma,s}M_0 \norm{\comi v^{ \boldsymbol{\ell_1}+\frac{\gamma}{2}}h}_{L^2_{x, v}}^2+C\norm{g}_{\mathscr{X}_\ell} \norm{\comi v^{ \boldsymbol{\ell_1}}h}_{L^2_{x, v}}^2\\
    &\leq \frac{m_0}{8}\frac{ \boldsymbol{\lambda}_ {2 \boldsymbol{\ell_1}}}{4^{\gamma}}   \norm{\comi v^{ \boldsymbol{\ell_1} +\frac{\gamma}{2}}h}_{L^2_{x, v}}^2+C\norm{g}_{\mathscr{X}_\ell} \norm{\comi v^{ \boldsymbol{\ell_1}}h}_{L^2_{x, v}}^2,
\end{align*}
where the second inequality uses condition $\boldsymbol{(H)}$  and \eqref{H++++}, and the last one follows  from  \eqref{loc-l}. 
Combining these estimates above with \eqref{spt}, we get
  that 
\begin{equation}\label{m2}
	   (\langle v\rangle^{ \boldsymbol{\ell_1} }Q(g,h), \langle v\rangle^{ \boldsymbol{\ell_1} }h)_{L^2_{x, v}} 
	 \leq -\frac{m_0\boldsymbol{\lambda}_ {2 \boldsymbol{\ell_1}}}{2^{3+2\gamma}}\norm{\comi v^{ \boldsymbol{\ell_1} +\frac{\gamma}{2}}h}_{L^2_{x, v}}^2 +C\norm{g}_{\mathscr{X}_\ell}\norm{\comi v^{ \boldsymbol{\ell_1} }h}_{L^2_{x, v}}^2.
\end{equation} 
On the other hand, it follows from
 \eqref{+coercivity-com} that  
\begin{equation}\label{m1}
	\begin{aligned}
		&\big(\comi v^{ \boldsymbol{\ell_1} }Q(g, h), \langle v\rangle^{ \boldsymbol{\ell_1} }h\big)_{L^2_{x, v}}  \\
        &\leq  -\frac{c_0}{2}\norm{\comi{D_v}^s\comi v^{ \boldsymbol{\ell_1} +\frac{\gamma}{2}} h}_{L^2_{x, v}}^2+ 2C_0 \norm{\comi v^{ \boldsymbol{\ell_1} +\frac{\gamma}{2}}h}_{L^2_{x, v}}^2+2^{1+\gamma} \big(\boldsymbol{\lambda}_{ \boldsymbol{\ell_1} }+\boldsymbol{A}_{\gamma,s}\big) \| g\|_{L_x^\infty L^1_v}   \|\langle v\rangle^{ \boldsymbol{\ell_1} +\frac{\gamma}{2}}h\|_{L^2_{x, v}}^2\\
		&\quad +C \|\langle v\rangle^{6+\gamma}h\|_{L^2_{x, v}}^2\|\langle v\rangle^{ \boldsymbol{\ell_1} +\frac{\gamma}{2}}g\|_{L_x^\infty L^2_v}^2  +C\Big(1+\|\langle v\rangle^{6+\gamma}g\|_{L_x^\infty L^2_v}^2\Big)   \|\langle v\rangle^{ \boldsymbol{\ell_1} }h\|_{L^2_{x, v}}^2\\
&\leq -\frac{c_0}{2}\norm{\comi{D_v}^s\comi v^{ \boldsymbol{\ell_1} +\frac{\gamma}{2}} h}_{L^2_{x, v}}^2+ 2^{1+\gamma} \Big( C_0 +4M_0 \boldsymbol{\lambda}_{2 \boldsymbol{\ell_1} }\Big)\norm{\comi v^{ \boldsymbol{\ell_1} +\frac{\gamma}{2}}  h}_{L^2_{x, v}}^2 \\
&\quad+  C\big(1+\norm{g}_{\mathscr{X}_\ell}^2\big)    \|\langle v\rangle^{ \boldsymbol{\ell_1} }h\|_{L^2_{x, v}}^2,
	\end{aligned}
\end{equation}
where the last inequality follows from \eqref{gconh} and \eqref{loc-l} as well as the fact that $\boldsymbol{\lambda}_{ \boldsymbol{\ell_1} }\leq \boldsymbol{\lambda}_{2 \boldsymbol{\ell_1}}$ and $\| g\|_{L_x^\infty L^1_v} \leq 2M_0$.   
Define $\tau_0>0 $ by setting
\begin{eqnarray*}
	\tau_0 2^{1+\gamma} \Big(C_0 +4M_0 \boldsymbol{\lambda}_{2 \boldsymbol{\ell_1}}\Big)=  \frac{m_0\boldsymbol{\lambda}_ {2 \boldsymbol{\ell_1}}}{2^{3+2\gamma}}.
\end{eqnarray*}
Consequently, 
multiplying  \eqref{m1}  by 
$\tau_0$ and then take summation  with estimate \eqref{m2}, we obtain 
\begin{align*}
	 \big(1+\tau_0\big)\big (\langle v\rangle^{ \boldsymbol{\ell_1} }Q(g,h),\  \langle v\rangle^{ \boldsymbol{\ell_1} }h\big)_{L^2_{x, v}} 
	 \leq -\frac{\tau_0c_0}{2}\norm{\comi{D_v}^s\comi v^{ \boldsymbol{\ell_1} +\frac{\gamma}{2}} h}_{L^2_{x, v}}^2   +C \big(1+\norm{g}_{\mathscr{X}_\ell}^2\big)\norm{\comi v^{ \boldsymbol{\ell_1} }h}_{L^2_{x, v}}^2.
\end{align*}
This  yields  estimate \eqref{fras} in Lemma \ref{lem:te}.
\end{proof}

\begin{proposition}[Non-negativity]\label{nonnegative}
Under the hypothesis of Theorem \ref{thm:ae}, we have $f\ge0$ for $t\in[0,T_0].$ 
\end{proposition}

\begin{proof}
   Write $f=f_{+}+f_{-}$ with
  \begin{equation*}
  	f_{+}:=\max\{f,0\}\ \textrm{ and }\  f_{-}:=-\max\{-f,0\}.
  \end{equation*}
 Since the solution $f$ is  smooth in all variables $(t,x,v)$ for $t>0$,  then, for any $\beta\in\mathbb Z_+^3,$
   \begin{equation}\label{mp}
  	(\partial_tf_+) f_-=0 \  \textrm{ and }\  (\partial_{x}^\beta  f_+) f_-=0, \   a.e.  \textrm{ on }    \mathbb T^3\times \mathbb R^3.
  \end{equation}
  As a result,
   we multiply   equation \eqref{eq-linearc} by $\langle v\rangle^{{2 \boldsymbol{\ell_1}}}f_-$ and then integrate over $\mathbb{T}^3\times \mathbb{R}^3$, to get 
  \begin{equation}\label{fminusv}
  \begin{split}
 &\big(\comi v^{ \boldsymbol{\ell_1}}Q(g,\ f),\   \langle v\rangle^{{ \boldsymbol{\ell_1}}}f_-\big)_{L^2_{x, v}}= \big(\comi v^{ \boldsymbol{\ell_1}} (\partial_t  +v\cdot\partial_x )f,\,   \langle v\rangle^{{ \boldsymbol{\ell_1}}}f_-\big)_{L^2_{x, v}} \\
  &=	\Big(\comi v^{ \boldsymbol{\ell_1}} \partial_t  (f_{+}+f_{-})+\comi v^{ \boldsymbol{\ell_1}} v\cdot\partial_x  (f_{+}+f_{-}),\   \langle v\rangle^{{ \boldsymbol{\ell_1}}}f_-\Big)_{L^2_{x, v}}=\frac{1}{2}\frac{d}{dt}\|\langle v\rangle^{ \boldsymbol{\ell_1}}f_-\|^2_{L^2_{x, v}}.
  \end{split}
  \end{equation}
On the one hand,   write
  \begin{equation}\label{qfg}
  \begin{aligned}
  		\big(\comi v^{ \boldsymbol{\ell_1}}Q(g,\ f),    \langle v\rangle^{{ \boldsymbol{\ell_1}}}f_-\big)_{L^2_{x, v}} 
  		&= \big(\comi v^{ \boldsymbol{\ell_1}}Q(g,\ f_+),    \langle v\rangle^{{ \boldsymbol{\ell_1}}}f_-\big)_{L^2_{x, v}} + \big(\comi v^{ \boldsymbol{\ell_1}}Q(g,\ f_{-}),   \langle v\rangle^{{ \boldsymbol{\ell_1}}}f_-\big)_{L^2_{x, v}}.
  \end{aligned}
  \end{equation}
  For the first term on the right-hand side,
 \begin{equation*}
 \begin{aligned}
   \big(\comi v^{ \boldsymbol{\ell_1}}Q(g,\ f_+),    \langle v\rangle^{{ \boldsymbol{\ell_1}}}f_-\big)_{L^2_{x, v}}  
 &=\int_{\mathbb{T}^3\times\mathbb{R}^6\times\mathbb{S}^2}   |v-v_*|^\gamma b(\cos\theta) (g_*'(f_+)'-g_*f_+)\langle v\rangle^{2{ \boldsymbol{\ell_1}}}f_-d\sigma  dv_*  dv dx\\
 &\leq \int_{\mathbb{T}^3\times\mathbb{R}^6\times\mathbb{S}^2}  |v-v_*|^\gamma b(\cos\theta)   g_*'(f_+)' \langle v\rangle^{2{ \boldsymbol{\ell_1}}}f_-\,d\sigma\, dv_*\, dv\,dx\leq 0,
 \end{aligned}
 \end{equation*}
the last line using \eqref{mp} and  the fact that $g_*'\ge0, (f_+)'  \geq 0 $ and $f_{-}\leq 0.$ On the other hand, using Lemma \ref{lem:te} yields
\begin{equation*}
 \big(\comi v^{ \boldsymbol{\ell_1}}Q(g,\ f_-),\   \langle v\rangle^{{ \boldsymbol{\ell_1}}}f_-\big)_{L^2_{x, v}}   	\leq  C\boldsymbol{L}^{2} \|\langle v\rangle^{  \boldsymbol{\ell_1} }f_-\|_{L^2_{x, v}}^2.
\end{equation*}
Combining the above two estimates with \eqref{qfg} we conclude 
\begin{equation*}
    \big(\comi v^{ \boldsymbol{\ell_1}}Q(g,\ f),\   \langle v\rangle^{{ \boldsymbol{\ell_1}}}f_-\big)_{L^2_{x, v}}\leq   C\boldsymbol{L}^{2} \|\langle v\rangle^{  \boldsymbol{\ell_1} }f_-\|_{L^2_{x, v}}^2,
\end{equation*}
which, with 
   \eqref{fminusv}, yields 
\begin{equation*}
\frac{d}{dt}\|\comi v^{ \boldsymbol{\ell_1}}f_-\|^2_{L^2_{x, v}}\leq C\boldsymbol{L}^{2}\|\comi v^{ \boldsymbol{\ell_1}}f_-\|^2_{L^2_{x, v}}
\end{equation*}
Thus,
\begin{equation*}
\forall\ 0\leq t\leq T_0,\quad	\|\comi v^{ \boldsymbol{\ell_1}}f_-(t)\|^2_{L^2_{x, v}}\leq e^{C\boldsymbol{L}^{2}T_0}\|\comi v^{ \boldsymbol{\ell_1}}(f_{in})_-\|^2_{L^2_{x, v}}= 0,
\end{equation*}
the last inequality follows from the fact that $f_{in}\geq 0.$  Thus
     $f_- (t)=0$ for $t\in[0,T_0]$, and hence $f\geq0$. The proof of Proposition \ref{nonnegative} is completed.
\end{proof}

\begin{proposition}[Moment estimate]\label{lem:mo}
 Suppose   the hypothesis of Theorem \ref{thm:ae} holds. Then 
	\begin{multline*}
	   \sup_{t\le T_0}\|\langle v\rangle^{\ell_0}f\|_{L^1_{x, v}} +\frac{m_0\boldsymbol{\lambda}_{\ell_0}}{4^{2+\gamma}} \int_0^{T_0} \|\langle v\rangle^{\ell_0+\gamma}f\|_{L^1_{x, v}}dt\\
 \leq e^{C \boldsymbol{L}  T_0} \|\langle v\rangle^{\ell_0}f_{in}\|_{L^1_{x, v}}+  e^{C\boldsymbol{L} T_0}\Big(\frac18\boldsymbol{L} +C \boldsymbol{L} T_0    \sup_{t\leq T_0}\norm{f}_{\mathscr{X}_{\ell}} \Big). 
	\end{multline*}
\end{proposition}

\begin{proof}
Note that 
$$
\int_{\mathbb{T}^3\times\mathbb{R}^3}  v\cdot\partial_x f (t,x,v)  \langle v\rangle^{\ell_0} dxdv=0.
$$
Then multiplying the linear Boltzmann equation \eqref{eq-linearc} by $\langle v\rangle^{\ell_0}$ and integrating over $\mathbb{T}^3_x\times \mathbb{R}^3_v$, we obtain
\begin{align*}
  \frac{d}{dt}\int_{\mathbb{T}^3\times\mathbb{R}^3} f(t,x,v)\langle v\rangle^{\ell_0}dxdv =\int_{\mathbb{T}^3\times\mathbb{R}^3} Q(g,\ f) \langle v\rangle^{\ell_0}dxdv.
\end{align*}
Observe $f,g\geq 0$ by Proposition \ref{nonnegative} and the assumption in Theorem \ref{thm:ae}. Thus,
combining the above identity with Proposition \ref{momentQ} yields
\begin{equation*}
\begin{aligned}
 &\frac{d}{dt}\|\langle v\rangle^{\ell_0}f\|_{L^1_{x, v}}+\frac{m_0\boldsymbol{\lambda}_{\ell_0}}{4^{\gamma  +1}}\|\langle v\rangle^{\ell_0+\gamma}f\|_{L^1_{x, v}} \leq   2\boldsymbol{\omega}_{\ell_0}\|f\|_{L^{\infty}_xL^1_v}\|\langle v\rangle^{ \ell_0 +\gamma}g\|_{L^1_{x, v}}\\
 &\qquad\qquad\qquad\qquad\qquad+C \norm{\comi v^{4+\gamma}g}_{L^{\infty}_xL^1_v} \norm{\comi v^{ \ell_0 } f}_{L^1_{x, v}}  +C \|\langle v\rangle^{4+\gamma}f\|_{L^{\infty}_xL^1_v} \|\langle v\rangle^{ \ell_0}g\|_{L^1_{x, v}}.
\end{aligned}
\end{equation*}
It follows from assumption 
$g\in\mathcal M_{T_0}(\boldsymbol{L}) \cap \mathcal E_{T_0}(\boldsymbol{L})$ 
that 
\begin{equation}\label{gl}
	 \sup_{t\leq T_0}\|\langle v\rangle^{\ell_0} g\|_{L^1_{x, v}} +\frac{m_0\boldsymbol{\lambda}_{\ell_0}}{4^{1+\gamma}}\int_0^{T_0}\|\langle v\rangle^{\ell_0+\gamma} g\|_{L^1_{x, v}}dt \leq\boldsymbol{L}, \ \ \textrm{ and }\ \ \sup_{t\leq T_0} \norm{g}_{\mathscr{X}_{\ell}} \leq  \boldsymbol{L}.
\end{equation}
Thus,  
\begin{equation}\label{melo}
\begin{aligned}
 &\frac{d}{dt}\|\langle v\rangle^{\ell_0}f\|_{L^1_{x, v}}+ \frac{m_0\boldsymbol{\lambda}_{\ell_0}}{4^{1+\gamma  }} \|\langle v\rangle^{\ell_0+\gamma}f\|_{L^1_{x, v}} \\
 &\qquad\qquad \leq   2\boldsymbol{\omega}_{\ell_0}\|f\|_{L^{\infty}_xL^1_v}\|\langle v\rangle^{ \ell_0 +\gamma}g\|_{L^1_{x, v}}+C \boldsymbol{L}\norm{\comi v^{ \ell_0 } f}_{L^1_{x, v}}  +C\boldsymbol{L} \| f\|_{\mathscr{X}_\ell}.
\end{aligned}
\end{equation}
To estimate the first term on the right-hand side, we integrate the linear Boltzmann equation \eqref{eq-linearc} for $v\in\mathbb R^3$  to obtain that 
$$
\partial_t \int_{\mathbb{R}^3_v}f(t,x,v)dv+\partial_x\cdot\int_{\mathbb{R}^3_v} vf(t,x,v)dv=0.
$$
Thus, for any $t\in[0,T_0],$
\begin{equation}\label{lb}
\begin{aligned}
 \Big|\int_{\mathbb{R}^3_v} f(t,x,v)dv-\int_{\mathbb{R}^3_v}f_{in}(x,v)dv\Big|\leq &\int_0^t \Big|\partial_x\cdot\int_{\mathbb{R}^3_v} vf(r,x,v)dv\Big| dr \\
 \leq & C\int_0^t  \|\langle v\rangle^{3} f(r)\|_{ H^3_xL^2_v}dr\leq C\, T_0\ \sup_{t\leq T_0}\norm{f}_{\mathscr{X}_\ell},
\end{aligned}
\end{equation}
which, with assumption \eqref{finite} as well as the fact $f\geq 0$,  implies
\begin{equation*}
\sup_{t\leq T_0} \norm{f}_{L_x^\infty L_v^1}\leq M_0+ C\, T_0\ \sup_{t\leq T_0}\norm{f}_{\mathscr{X}_\ell}.
\end{equation*}
As a result, combining the above estimate with \eqref{melo}, we get
\begin{multline*}
    \frac{d}{dt}\|\langle v\rangle^{\ell_0}f\|_{L^1_{x, v}}+\frac{m_0\boldsymbol{\lambda}_{\ell_0}}{4^{1+\gamma}}\|\langle v\rangle^{\ell_0+\gamma}f\|_{L^1_{x, v}}\\
      \leq 2   \boldsymbol{\omega}_{\ell_0} \Big(M_0+ C\, T_0\ \sup_{t\leq T_0}\norm{f}_{\mathscr{X}_\ell}\Big) \|\langle v\rangle^{\ell_0+\gamma}g\|_{L^1_{x, v}} 
 +C \boldsymbol{L} \|\langle v\rangle^{\ell_0}f\|_{L^1_{x, v}}+C \boldsymbol{L} \norm{f}_{\mathscr{X}_{\ell}} .
\end{multline*}
 This with  Gronwall's inequality    implies
 \begin{equation*}
\begin{aligned}
&\sup_{t\le T_0}\|\langle v\rangle^{\ell_0}f\|_{L^1_{x, v}} +\frac{m_0\boldsymbol{\lambda}_{\ell_0}}{4^{1+\gamma}} \int_0^{T_0} \|\langle v\rangle^{\ell_0+\gamma}f\|_{L^1_{x, v}}dt\\
&\leq e^{C\boldsymbol{L} T_0} \|\langle v\rangle^{\ell_0}f_{in}\|_{L^1_{x, v}} + C e^{C\boldsymbol{L} T_0} \boldsymbol{L} T_0  \sup_{t\leq T_0}\norm{f}_{\mathscr{X}_{\ell}} \\
& \quad +  2e^{C\boldsymbol{L} T_0}  \Big (M_0+ C\, T_0\ \sup_{t\leq T_0}\norm{f}_{\mathscr{X}_\ell}\Big) \boldsymbol{\omega}_{\ell_0} \int_0^{T_0}\|\langle v\rangle^{\ell_0+\gamma}g\|_{L^1_{x, v}}dt.
\end{aligned}
\end{equation*}
Moreover, it follows from    \eqref{gl} and \eqref{loc-l} with the fact $\ell_0>\boldsymbol{\ell_1}$ that 
\begin{equation*}
	\boldsymbol{\omega}_{\ell_0}   \int_0^T\|\langle v\rangle^{\ell_0+\gamma}g\|_{L^1_{x, v}}dt \leq  \frac{4^{1+\gamma}\boldsymbol{\omega}_{\ell_0}}{m_0\boldsymbol{\lambda}_{\ell_0}}\boldsymbol{L}\leq\frac{4^{1+\gamma}\boldsymbol{\omega}_{\boldsymbol{\ell_1}}}{m_0\boldsymbol{\lambda}_{\boldsymbol{\ell_1}}}\boldsymbol{L}\leq \frac{1}{64M_0}\boldsymbol{L}.
\end{equation*}
 Thus, combining the two above estimates,  we obtain the assertion in Proposition \ref{lem:mo}, completing the proof.   
\end{proof}

\begin{proposition}[Control of weight-loss terms]\label{crulem}
 For any $\varepsilon>0$ it holds that
	\begin{equation*}
		\|\langle v\rangle^{ s+\frac{\gamma}{2}}h\|_{\mathscr{X}_{\ell}}^2 \leq \varepsilon\norm{h}_{\mathcal{Y}_{\ell}}^2+ \varepsilon\norm{h}_{\mathcal{Z}_{\ell}}^2+ C\varepsilon^{-\frac{3+6s}{s}}\Big(\|\langle v\rangle^{\ell_0}h\|_{L^1_{x, v}}^2+\|h\|^2_{\mathscr{X}_{\ell}}\Big),
	\end{equation*}
	recalling $\ell_0=\ell+3+7s+2\gamma.$
\end{proposition}

\begin{proof}This is just an immediate consequence of Lemmas \ref{embedding} and \ref{inter-inequality}.
	Recalled that
$$
\|\langle v\rangle^{ s+\frac{\gamma}{2}}h\|^2_{\mathscr{X}_{\ell}}=\sum_{|\alpha|=3}\|\langle v\rangle^{\ell-\rho |\alpha|+ s+\frac{\gamma}{2}} \partial^{\alpha}_xh\|^2_{L^2_{x, v}}+\|\langle v\rangle^{\ell+ s+\frac{\gamma}{2}} h\|^2_{L^2_{x, v}}.
$$
For any $\varepsilon>0,$  applying  Lemma \ref{inter-inequality}  with $\tau=\ell-3\rho $   gives  
\begin{align*}
 \sum_{|\alpha|=3}\|\langle v\rangle^{\ell-\rho |\alpha|+ s+\frac{\gamma}{2}} \partial^{\alpha}_xh\|^2_{L^2_{x, v}} 
&\leq \varepsilon\sum_{|\alpha|=3}\|\langle v\rangle^{\ell-\rho|\alpha|+ \frac{\gamma}{2(1+2s)}}\langle D_x\rangle^{\frac{s}{1+2s}}\partial^{\alpha}_x h\|^2_{L^2}+C\varepsilon^{-\frac{3+6s}{s}}\|\langle v\rangle^{\ell }h\|^2_{L^2_{x, v}}\\
&\leq \varepsilon\norm{h}_{\mathcal{Z}_{\ell}} ^2+ C\varepsilon^{-\frac{3+6s}{s}}\| h\|^2_{\mathscr{X}_\ell},
\end{align*}
where the last inequality follows from the definition of $\norm{\cdot}_{\mathcal{Z}_{\ell}}$ in \eqref{zn}. Consequently, for any $\varepsilon>0,$
\begin{equation*}
	\|\langle v\rangle^{ s+\frac{\gamma}{2}}h\|_{\mathscr{X}_{\ell}}^2\leq \varepsilon\norm{h}_{\mathcal{Z}_{\ell}}^2+C\varepsilon^{-\frac{3+6s}{s}}\| h\|_{\mathscr{X}_\ell}^2+\|\langle v\rangle^{\ell+ s+\frac{\gamma}{2}} h\|^2_{L^2_{x, v}}.
\end{equation*}
To estimate the last term, we use  Lemma \ref{embedding} to obtain that, for any $ \varepsilon>0,$
\begin{align*}
 \|\langle v\rangle^{\ell+s+\frac{\gamma}{2}}h\|_{L^2_{x, v}}^2 &\leq  \varepsilon \|\langle v\rangle^{\ell+\frac{\gamma}{2(1+2s)}}\langle D_x\rangle^{\frac{s}{1+2s}}   h\|_{L^2_{x, v}}^2+ \varepsilon\normm{\comi v^{\ell} h}_{L^2_x}^2+C {\varepsilon}^{-\frac{3+6s}{s}}\|\langle v\rangle^{\ell+3+7s+2\gamma}h\|_{L^1_{x, v}}^2\\
 &\leq  \varepsilon \norm{h}_{\mathcal{Z}_\ell}^2+ \varepsilon\norm{  h}_{\mathcal{Y}_\ell}^2+C {\varepsilon}^{-\frac{3+6s}{s}}\|\langle v\rangle^{\ell+3+7s+2\gamma}h\|_{L^1_{x, v}}^2.
\end{align*}
Noting  $\ell_0=\ell+3+7s+2\gamma,$  we combine the two  estimates above to obtain  the assertion in Proposition \ref{crulem}. The proof is completed. 
\end{proof}

\begin{proposition}[Energy estimate]\label{lem:energy}
Suppose the hypothesis of Theorem \ref{thm:ae} holds. Then 
\begin{multline*}
 3 \big(\sup_{t\leq T_0}\norm{f}_{\mathscr{X}_{\ell}}\big)^2  + c_0    \int_0^{T_0}   \norm{ f}_{\mathcal{Y}_{\ell}}^2 dt + a_0   \int_0^{T_0}  \norm{ f}_{\mathcal{Z}_{\ell}}^2 dt  \\
	\leq 8\norm{f_{in}}_{\mathscr{X}_{\ell}}^2+C\boldsymbol{L}^2\delta_0^2 +C   \boldsymbol{L} ^{\frac{6+12s}{s}} T_0\Big(\sup_{t\leq T_0}\norm{\comi v^{\ell_0} f}_{L^1_{x, v}}\Big)^2+  C   \boldsymbol{L} ^{\frac{6+14s}{s}} T_0 \Big(\sup_{t\leq T_0}  \|f\|_{\mathscr{X}_{\ell}}\Big)^2, 
 \end{multline*}
 where $a_0$ is the constant determined in Theorem \ref{hypoellip}, and $\delta_0$ is the constant in \eqref{smal}.
\end{proposition}

\begin{proof}
By Theorem \ref{hypoellip}, we have 
\begin{equation*} 
		\begin{aligned}
		& \Big(\sup_{t\leq T_0}\norm{f}_{\mathscr{X}_{\ell}} \Big)^2 + \frac{c_0}{2} \int_0^{T_0}   \norm{ f}_{\mathcal{Y}_{\ell}}^2 dt +  \frac{a_0}{2}   \int_0^{T_0}  \norm{ f}_{\mathcal{Z}_{\ell}}^2 dt\\
	&\leq 2\norm{f_{in}}_{\mathscr{X}_{\ell}}^2+ \frac{c_0 }{16\boldsymbol{L}} \sup_{t\leq T_0}\norm{g}_{\mathscr{X}_{\ell}}  \int_0^{T_0} \| f\|_{\mathcal{Y}_{\ell}} ^2dt+C  ( 1+\sup_{t\leq T_0}\norm{g}_{\mathscr{X}_{\ell}}  )^2\int_0^{T_0}\norm{\comi{D_v}^s\comi v^{ \boldsymbol{\ell_1} +\frac{\gamma}{2}}f}_{L^2_{x, v}}^2dt  
\\
	&\quad+C \Big( 1+\sup_{t\leq T_0}\norm{g}_{\mathscr{X}_{\ell}} \Big)^2\int_0^{T_0}\norm{\comi v^{s+\frac{\gamma}{2}}f}_{\mathscr{X}_{\ell}}^2dt  +C  \Big(\sup_{t\leq T_0}  \|f\|_{\mathscr{X}_{\ell}}\Big)^2  \int_0^{T_0}  \|\langle v\rangle^{s+ \frac{\gamma}{2}}g\|_{\mathscr{X}_{\ell}}^2dt,
		\end{aligned}
	\end{equation*}
	which with assumption \eqref{smal} and  the condition $g\in\mathcal E_{T_0}(\boldsymbol{L})$ yields  
\begin{equation}\label{fxl}
		\begin{aligned}
		 &\big(\sup_{t\leq T_0}\norm{f}_{\mathscr{X}_{\ell}}\big)^2  +\frac{c_0}{2}   \int_0^{T_0}   \norm{ f}_{\mathcal{Y}_{\ell}}^2 dt + \frac{a_0}{2}   \int_0^{T_0} \norm{ f}_{\mathcal{Z}_{\ell}}^2 dt\\
		 & \leq 2\norm{f_{in}}_{\mathscr{X}_{\ell}}^2+ \frac{c_0 }{16 }  \int_0^{T_0} \| f\|_{\mathcal{Y}_{\ell}} ^2dt+C\boldsymbol{L}^2\delta_0^2 \\
		 &\quad+C \boldsymbol{L}^2\int_0^{T_0}\norm{\comi v^{s+\frac{\gamma}{2}}f}_{\mathscr{X}_{\ell}}^2dt  +C  \Big(\sup_{t\leq T_0}  \|f\|_{\mathscr{X}_{\ell}}\Big)^2  \int_0^{T_0}  \|\langle v\rangle^{s+ \frac{\gamma}{2}}g\|_{\mathscr{X}_{\ell}}^2dt. 	
		\end{aligned}
	\end{equation}
 Moreover, it follows from Proposition \ref{crulem} that, for any $\eps>0,$
 \begin{align*}
     &\int_0^{T_0}\norm{\comi v^{s+\frac{\gamma}{2}}f}_{\mathscr{X}_{\ell}}^2dt\\
     &\leq \eps    \int_0^{T_0}   \big(c_0\norm{ f}_{\mathcal{Y}_{\ell}}^2  + a_0  \norm{ f}_{\mathcal{Z}_{\ell}}^2\big) dt+C  \eps^{-\frac{3+6s}{s}}  \Big(\int_0^{T_0}   \norm{\comi v^{\ell_0} f}_{L^1}^2 dt+\int^{T_0}_0\norm{f}_{\mathscr{X}_{\ell}}^2dt\Big)\\
     &\leq \eps    \int_0^{T_0}   \big(c_0\norm{ f}_{\mathcal{Y}_{\ell}}^2  +   a_0\norm{ f}_{\mathcal{Z}_{\ell}}^2\big) dt+C  \eps^{-\frac{3+6s}{s}}  T_0\Big[\big(\sup_{t\leq T_0}\norm{\comi v^{\ell_0} f}_{L^1}\big)^2+\big(\sup_{t\leq T_0}\norm{f}_{\mathscr{X}_{\ell}}\big)^2\Big], 
 \end{align*}
 and
 \begin{align*}
     &\int_0^{T_0}\norm{\comi v^{s+\frac{\gamma}{2}}g}_{\mathscr{X}_{\ell}}^2dt\\
     &\leq \eps   \int_0^{T_0}   \big(c_0\norm{ g}_{\mathcal{Y}_{\ell}}^2  +   a_0\norm{ g}_{\mathcal{Z}_{\ell}}^2\big) dt+C  \eps^{-\frac{3+6s}{s}}  {T_0}\Big[\sup_{t\leq T_0}\norm{\comi v^{\ell_0} g}_{L^1_{x, v}}+\sup_{t\leq T_0}\norm{g}_{\mathscr{X}_{\ell}}\Big]^2\\
     &\leq \eps   \boldsymbol{L}^2   +C  \eps^{-\frac{3+6s}{s}}  T_0 \boldsymbol{L}^2  
 \end{align*}
 for $g\in\mathcal M_{T_0}(\boldsymbol{L})\cap\mathcal E_{T_0}(\boldsymbol{L}).$
 Substituting these estimates into \eqref{fxl} we conclude that
 \begin{equation*}
 	\begin{aligned}
		 &\big(\sup_{t\leq T_0}\norm{f}_{\mathscr{X}_{\ell}}\big)^2  +\frac{c_0}{2}   \int_0^{T_0}   \norm{ f}_{\mathcal{Y}_{\ell}}^2 dt + \frac{a_0}{2}   \int_0^{T_0} \norm{ f}_{\mathcal{Z}_{\ell}}^2 dt\\
		 & \leq 2\norm{f_{in}}_{\mathscr{X}_{\ell}}^2+ \frac{c_0}{4}  \int_0^{T_0} \| f\|_{\mathcal{Y}_{\ell}} ^2dt+ \frac{a_0 }{4}  \int_0^{T_0} \| f\|_{\mathcal{Z}_{\ell}} ^2dt+C\boldsymbol{L}^2\delta_0^2 \\
		 &\quad +C   \boldsymbol{L} ^{\frac{6+12s}{s}} T_0\Big(\sup_{t\leq T_0}\norm{\comi v^{\ell_0} f}_{L^1_{x, v}}\Big)^2+\Big(\frac14+ C   \boldsymbol{L} ^{\frac{6+14s}{s}} T_0\Big)\Big(\sup_{t\leq T_0}  \|f\|_{\mathscr{X}_{\ell}}\Big)^2.	
		 \end{aligned}
 \end{equation*}
 This  yields the desired estimate in Proposition \ref{lem:energy}. The    proof is completed. 
  \end{proof}

\begin{proof}
	[Completing the proof of Theorem \ref{thm:ae}] We recall the estimates in Propositions \ref{lem:mo} and \ref{lem:energy}. Namely, there exists a constant $C_*,$ depending only on $\ell$, $s$ and $\gamma $, such that
   \begin{multline*}
	   \sup_{t\le T_0}\|\langle v\rangle^{\ell_0}f\|_{L^1_{x, v}} +\frac{m_0\boldsymbol{\lambda}_{\ell_0}}{4^{1+\gamma}} \int_0^{T_0} \|\langle v\rangle^{\ell_0+\gamma}f\|_{L^1_{x, v}}dt\\
 \leq e^{C_* \boldsymbol{L}  T_0} \|\langle v\rangle^{\ell_0}f_{in}\|_{L^1_{x, v}}+  e^{C_* \boldsymbol{L}  T_0}\bigg(\frac18\boldsymbol{L}   +C_*  \boldsymbol{L} T_0    \sup_{t\leq T_0}\norm{f}_{\mathscr{X}_{\ell}} \bigg). 
	\end{multline*}
    and
 \begin{multline*}
  3\big(\sup_{t\leq T_0}\norm{f}_{\mathscr{X}_{\ell}}\big)^2  +  c_0   \int_0^{T_0}   \norm{ f}_{\mathcal{Y}_{\ell}}^2 dt + a_0   \int_0^{T_0}  \norm{ f}_{\mathcal{Z}_{\ell}}^2 dt  \\
	\leq 8\norm{f_{in}}_{\mathscr{X}_{\ell}}^2+C_* \boldsymbol{L}^2\delta_0^2 +C_*   \boldsymbol{L} ^{\frac{6+12s}{s}} T_0\Big(\sup_{t\leq T_0}\norm{\comi v^{\ell_0} f}_{L^1_{x, v}}\Big)^2+  C_*   \boldsymbol{L} ^{\frac{6+14s}{s}} T_0 \Big(\sup_{t\leq T_0}  \|f\|_{\mathscr{X}_{\ell}}\Big)^2. 
 \end{multline*}
Now we choose   $T_0$ and $\delta_0$,  such that  
\begin{equation*}
   C_* \delta_0^2+ C_*   \boldsymbol{L} ^{\frac{6+14s}{s}}T_0\leq \frac18.
\end{equation*}
Then  the above estimates reduce to 
\begin{align*}
	\sup_{t\le T_0}\|\langle v\rangle^{\ell_0}f\|_{L^1_{x, v}} +\frac{m_0\boldsymbol{\lambda}_{\ell_0}}{4^{1+\gamma}} \int_0^{T_0} \|\langle v\rangle^{\ell_0+\gamma}f\|_{L^1_{x, v}}dt 
 \leq 2 \|\langle v\rangle^{\ell_0}f_{in}\|_{L^1_{x, v}}+\frac14\boldsymbol{L}  +  \frac14  \sup_{t\leq T_0}\norm{f}_{\mathscr{X}_{\ell}} ,
\end{align*}
    and
    \begin{multline*}
2\big(\sup_{t\leq T_0}\norm{f}_{\mathscr{X}_{\ell}}\big)^2  +c_0 \int_0^{T_0}   \norm{ f}_{\mathcal{Y}_{\ell}}^2 dt 
+ a_0 \int_0^{T_0}  \norm{ f}_{\mathcal{Z}_{\ell}}^2 dt\\
 \leq 8\norm{f_{in}}_{\mathscr{X}_{\ell}}^2+\frac{1}{8} \boldsymbol{L}^2  
	 +\frac{1}{8}\Big(\sup_{t\leq T_0}\norm{\comi v^{\ell_0} f}_{L^1_{x, v}}\Big)^2 . 
 \end{multline*}
 Combining the above two estimates and the definition of $\boldsymbol{L}$ in \eqref{A-0} implies 
  \begin{align*}
 \big(\sup_{t\leq T_0}\norm{f}_{\mathscr{X}_{\ell}}\big)^2  +c_0 \int_0^{T_0}   \norm{ f}_{\mathcal{Y}_{\ell}}^2 dt + a_0   
 \int_0^{T_0}  \norm{ f}_{\mathcal{Z}_{\ell}}^2 dt\leq \boldsymbol{L}^2,
 \end{align*} 
 and
 	 \begin{equation*}
 	 	 \sup_{t\le T_0}\|\comi v^{\ell_0}f\|_{L^1_{x, v}} +\frac{m_0\boldsymbol{\lambda}_{\ell_0}}{4^{1+\gamma}} \int_0^{T_0} \|\langle v\rangle^{\ell_0+\gamma}f\|_{L^1_{x, v}}dt \leq \boldsymbol{L}.
 	 \end{equation*}
 	Hence $f\in \mathcal M_{T_0}(\boldsymbol{L})\cap \mathcal E_{T_0}(\boldsymbol{L}).$ The non-negativity of $f$ is established in Proposition \ref{nonnegative}. Moreover,
 by  \eqref{lb} and the fact that $f \in \mathcal E_{T_0}(\boldsymbol{L}),$
  we have 
  \begin{equation*}
  	\Big|\int_{\mathbb{R}^3_v} f(t,x,v)dv-\int_{\mathbb{R}^3_v}f_{in}(x,v)dv\Big|\leq  C\, T_0\, \sup_{t\leq T_0}\norm{f}_{\mathscr{X}_\ell}\leq C \boldsymbol{L}T_0.
  \end{equation*}
 Thus,	 shrinking $T_0$ if necessary, it follows from the above estimate and \eqref{finite} that  
 	\begin{equation*}
\forall \ (t,x)\in[0,T_0]\times\mathbb T^3,\quad \frac{m_0}{2}\leq 	\int_{\mathbb{R}^3_v} f(t,x,v)dv	\leq 2M_0.
 	\end{equation*} 
 The derivation of the upper bound of entropy and energy is just the same as in \cite[Proposition 4.1]{preprint-JHL}, so we omit it for brevity. This yields $f$ satisfies condition $\boldsymbol{(H)}$,  completing  the proof of Theorem \ref{thm:ae}.  
\end{proof}

 
\section{Proof of the main result}\label{s:main}
This section is devoted to proving 
  Theorem \ref{local-exsitence}.  We will combine  the {\it \`a priori} estimate in Theorem \ref{thm:ae}  with a standard  iteration scheme to construct a solution to the  nonlinear equation \eqref{eq-1}.   The delicate part  is  to find an uniform time $T$ such that the condition  \eqref{smal}  holds for all functions in the iteration procedure.

We first prove the following technical lemma that will be used to establish the existence and uniqueness of solutions. 


\begin{lemma}\label{lem:te2}
	Suppose $f\in L^\infty([0,T], \mathscr{X}_{\ell})\cap L^2([0,T], \mathcal{Y}_{\ell})$ satisfies condition $\boldsymbol{(H)}$. Then, for  any $\varepsilon>0,$  it holds that 
	\begin{align*}
 		&\big| \big(\comi v^{\boldsymbol{\ell_1}}Q( g,\     f ), \  \comi v^{\boldsymbol{\ell_1}} h \big)_{L^2_{x, v}} \big|\\
&\leq \frac{\tilde c_0}{2} \|\comi {D_v}^s \comi v^{\boldsymbol{\ell_1}+\frac{\gamma}{2}}  h \|_{L^2_{x, v}}^2+C\norm{f}_{\mathscr{X}_{\ell}}^2 \|\comi v^{\boldsymbol{\ell_1}} h\|_{L^2_{x, v}}^2+ \frac{\tilde c_0}{64}  \|\langle v\rangle^{\boldsymbol{\ell_1}+\frac{\gamma}{2}}g\|_{L^2_{x, v}}^2 \\
&\quad    +\|\langle v\rangle^{\boldsymbol{\ell_1}} g\|_{L^2_{x, v}}^2\Big(\varepsilon c_0\norm{  f }_{\mathcal{Y}_\ell}^2
+C(c_0\varepsilon)^{-5}   \|\comi {D_v}^s \comi v^{\boldsymbol{\ell_1}+\frac{\gamma}{2}}     f \|_{L^2_{x, v}}^2+C\norm{  f }_{\mathscr{X}_\ell}^2\Big),
 	\end{align*}
    where $c_0\geq \tilde c_0$ are the constants specified  in  \eqref{lower+} and Lemma \ref{lem:te},respectively.
\end{lemma}

  \begin{proof}
   We write
 \begin{multline*}
 	\big(\comi v^{\boldsymbol{\ell_1}}Q( g,\    f ), \  \comi v^{\boldsymbol{\ell_1}} h \big)_{L^2_{x, v}}=\big(Q( g,\ \comi v^{\boldsymbol{\ell_1}}   f ), \  \comi v^{\boldsymbol{\ell_1}} h \big)_{L^2_{x, v}}\\
 	+\big(\comi v^{\boldsymbol{\ell_1}}Q( g,\    f ), \  \comi v^{\boldsymbol{\ell_1}} h \big)_{L^2_{x, v}}-\big(Q( g,\ \comi v^{\boldsymbol{\ell_1}}   f ), \  \comi v^{\boldsymbol{\ell_1}} h \big)_{L^2_{x, v}}.
\end{multline*}
Moreover,  it follows from Lemma \ref{commutator} as well as \eqref{rhod} that
\begin{equation*}
\begin{split}
&\big|\big(\comi v^{\boldsymbol{\ell_1}}Q( g,\    f ), \  \comi v^{\boldsymbol{\ell_1}} h \big)_{L^2_{x, v}}-\big(Q( g,\ \comi v^{\boldsymbol{\ell_1}}   f ), \  \comi v^{\boldsymbol{\ell_1}} h \big)_{L^2_{x, v}} \big|\\
	&\leq  C  \|\langle v\rangle^{6+\gamma}g\|_{L^2_{x, v}}\|\langle v\rangle^{\boldsymbol{\ell_1}+\frac{\gamma}{2}}h\|_{L^2_{x, v}}\|\comi{D_v}^s\langle v\rangle^{\boldsymbol{\ell_1}+\frac{\gamma}{2}} f\|_{L_x^\infty L^2_v}\\
	&\quad +	2  \boldsymbol{\omega}_{\boldsymbol{\ell_1}-2-\gamma} \|f\|_{L_x^\infty L^1_v} \|\langle v\rangle^{\boldsymbol{\ell_1}+\frac{\gamma}{2}}g\|_{L^2_{x, v}} \|\langle v\rangle^{ \boldsymbol{\ell_1}+\frac{\gamma}{2}}h\|_{L^2_{x, v}}\\
	&\quad+C \|\langle v\rangle^{6+ \gamma}f\|_{L_x^\infty L_v^2}\|\langle v\rangle^{\boldsymbol{\ell_1}+\frac{\gamma}{2}}g\|_{L^2_{x, v}}  \|\comi v^{\boldsymbol{\ell_1}} h\|_{L^2_{x, v}} \\
    &\ \ \ \ + C\|\langle v\rangle^{6+\gamma}f\|_{L_x^\infty L^2_v} \|\langle v\rangle^{\boldsymbol{\ell_1}}g\|_{L^2_{x, v}}      \|\langle v\rangle^{ \boldsymbol{\ell_1}+\frac{\gamma}{2}}h\|_{L^2_{x, v}}\\
&\leq  \frac{\tilde c_0}{4}  \| \langle v\rangle^{\boldsymbol{\ell_1}+\frac{\gamma}{2}} h \|_{L^2_{x, v}}^2+\frac{8   \boldsymbol{\omega}_{\boldsymbol{\ell_1}-2-\gamma} ^2 }{\tilde c_0}  \|f\|_{L_x^\infty L^1_v}^2 \|\langle v\rangle^{\boldsymbol{\ell_1}+\frac{\gamma}{2}}g\|_{L^2_{x, v}}^2+\frac{\tilde c_0}{128}\|\langle v\rangle^{\boldsymbol{\ell_1}+\frac{\gamma}{2}}g\|_{L^2_{x, v}}^2\\
&\quad+C\norm{f}_{\mathscr{X}_\ell}^2\Big(\|\comi v^{\boldsymbol{\ell_1}} h\|_{L^2_{x, v}}^2+ \|\langle v\rangle^{\boldsymbol{\ell_1}}g\|_{L^2_{x, v}}^2    \Big)  +C\norm{\comi v^{\boldsymbol{\ell_1}} g}_{L^2_{x, v}}^2\|\comi{D_v}^s  \langle v\rangle^{\boldsymbol{\ell_1}+\frac{\gamma}{2}}  f \|_{L^{\infty}_x L_v^2}^2,
\end{split}
\end{equation*}
where $\tilde c_0$ is the constant given in \eqref{fras}.  
On the other hand, 
applying the  trilinear upper bound \eqref{uppbound2} with  $a_1=2s+\frac{\gamma}{2}$ and $a_2=\frac{\gamma}{2}$, we get 
 \begin{equation*}
 \begin{aligned}
 & \big| \big(Q( g,\  \comi v^{\boldsymbol{\ell_1}}   f ), \  \comi v^{\boldsymbol{\ell_1}} h \big)_{L^2_{x, v}} \big| \\
& \leq  C \Big(\|\langle v\rangle^{\gamma+2s} g\|_{L_x^2L^1_v}+\| g\|_{L^2_{x, v}}\Big)  \|\comi {D_v}^s\comi v^{\boldsymbol{\ell_1}+2s+\frac{\gamma}{2}}    f \|_{L_x^\infty L_v^2}\|\comi {D_v}^s\comi v^{\boldsymbol{\ell_1}+\frac{\gamma}{2}}   h \|_{L^2_{x, v}}\\
& \leq  C \|\langle v\rangle^{\boldsymbol{\ell_1}} g\|_{L^2_{x, v}}  \|\comi {D_v}^s\comi v^{\boldsymbol{\ell_1}+2s+\frac{\gamma}{2}}    f \|_{L_x^\infty L_v^2}\|\comi {D_v}^s\comi v^{\boldsymbol{\ell_1}+\frac{\gamma}{2}}   h \|_{L^2_{x, v}}\\
&\leq  \frac{\tilde c_0}{4}  \|\comi {D_v}^s\comi v^{\boldsymbol{\ell_1}+\frac{\gamma}{2}}   h \|_{L^2_{x, v}}^2+C \|\langle v\rangle^{\boldsymbol{\ell_1}} g\|_{L^2_{x, v}}^2  \|\comi {D_v}^s   \comi v^{\boldsymbol{\ell_1}+2s+\frac{\gamma}{2}}  f \|_{L_x^\infty L_v^2}^2,
 \end{aligned}
 \end{equation*}
where  the second inequality follows from \eqref{rhod}.  Combining these estimates, we get 
 \begin{equation}\label{efgh}
 \begin{aligned}
 & \big| \comi v^{\boldsymbol{\ell_1}} \big(Q( g,\     f ), \  \comi v^{\boldsymbol{\ell_1}} h \big)_{L^2_{x, v}} \big|\\ 
&\leq \frac{\tilde c_0}{2} \|\comi {D_v}^s \comi v^{\boldsymbol{\ell_1}+\frac{\gamma}{2}}  h \|_{L^2_{x, v}}^2  +\bigg(\frac{8   \boldsymbol{\omega}_{\boldsymbol{\ell_1}-2-\gamma} ^2 }{\tilde c_0}  \|f\|_{L_x^\infty L^1_v}^2+\frac{\tilde c_0}{128}\bigg) \|\langle v\rangle^{\boldsymbol{\ell_1}+\frac{\gamma}{2}}g\|_{L^2_{x, v}}^2\\
&\quad +C\norm{f}_{\mathscr{X}_{\ell}}^2\Big(\|\comi v^{\boldsymbol{\ell_1}} h\|_{L^2_{x, v}}^2  +\|\langle v\rangle^{\boldsymbol{\ell_1}}g\|_{L^2_{x, v}}^2    \Big)   +C \|\langle v\rangle^{\boldsymbol{\ell_1}} g\|_{L^2_{x, v}}^2  \|\comi {D_v}^s   \comi v^{\boldsymbol{\ell_1}+2s+\frac{\gamma}{2}}  f \|_{L_x^\infty L_v^2}^2.
 \end{aligned}
 \end{equation}
  Using \eqref{loc-l} and the assumption that $f$ satisfies condition $\boldsymbol{(H)}$, the factor appearing in the second term on the right-hand side of \eqref{efgh} can be bounded as follows: 
 \begin{equation*}
     \frac{ 8  \boldsymbol{\omega}_{\boldsymbol{\ell_1}-2-\gamma} ^2 }{\tilde c_0}  \|f\|_{L_x^\infty L^1_v}^2+\frac{\tilde c_0}{128} \leq  \frac{ 8  \boldsymbol{\omega}_{\boldsymbol{\ell_1}-2-\gamma} ^2 }{\tilde c_0} (2M_0)^2+\frac{\tilde c_0}{128}\leq  \frac{\tilde c_0}{64}.
 \end{equation*}
For the last factor in \eqref{efgh}, we use Lemma \ref{lem:small} to get that,  for any $\varepsilon>0,$  
 \begin{equation*}
 	\begin{aligned}
 	 \|\comi {D_v}^s   \comi v^{\boldsymbol{\ell_1}+2s+\frac{\gamma}{2}}  f \|_{L_x^\infty L_v^2}^2 &\leq 
 	\|\comi {D_v}^s   \comi v^{\boldsymbol{\ell_1}+2s+\frac{\gamma}{2}}  f \|_{H_x^2 L_v^2}^2 \leq \varepsilon \norm{  f }_{\mathcal{Y}_\ell} ^2+C\varepsilon^{-5} \|\comi {D_v}^s\comi v^{\boldsymbol{\ell_1}+\frac{\gamma}{2}}     f \|_{L^2_{x, v}}^2.
 	\end{aligned}
 \end{equation*}
 As a result, combining  the two estimates above with \eqref{efgh}, we obtain the assertion in Lemma \ref{lem:te2}, completing the  proof.  
\end{proof}

\subsection{Iteration scheme}
  
 To clarify the argument,  we first   focus on the iteration procedure,  assuming the existence of a solution to the linear problem.  The key part is to find an uniform time $T$ such that the condition  \eqref{smal}  holds for all functions in the iteration procedure.

 \begin{proposition}\label{lem:itera}
 	 Let $f_{in}$ be the initial datum in equation  \eqref{eq-1}  satisfying the assumptions in Theorem \ref{local-exsitence}.   Suppose that  there exists a time $T_*>0,$ such that 	whenever  $g\geq 0$ satisfies  condition $\boldsymbol{(H)}$ for $0<t\leq T_* $ and belongs to  $ \mathcal{M}_{T_*}(\boldsymbol{L})\cap \mathcal{E}_{T_*}(\boldsymbol{L})$,  
	the linear Boltzmann equation
	\begin{equation}\label{liboltz}
		\partial_t h+ v\cdot \partial_x h
		= Q(g,\, h),  \quad h|_{t=0}=f_{in},
\end{equation}
admits a local solution $h\in L^\infty([0,T_*], \mathscr{X}_\ell)\cap L^2([0,T_*], \mathcal{Y}_\ell),$ which is $C^\infty$-smooth in all variables at $t>0.$ 
 	 Then there exists a time $T\in (0,T_*]$ such that the original  nonlinear Boltzmann equation \eqref{eq-1} has a local  weak solution $f\in L^\infty([0,T]; \mathscr{X}_{\ell})\cap L^2([0,T]; \mathcal{Y}_{\ell})$. Furthermore, $f\geq 0$ satisfies condition $\boldsymbol{(H)}$ and belongs to  $ \mathcal{M}_{T}(\boldsymbol{L})\cap \mathcal{E}_{T}(\boldsymbol{L}).$  In addition, the following estimate holds:
 	\begin{equation}\label{fsmall}
 		 \bigg(\int_0^{T}\norm{\comi {D_v}^s \comi v^{\boldsymbol{\ell_1}+\frac{\gamma}{2}} f}_{L^2_{x, v}}^2 dt\bigg)^{\frac12} \leq  \delta_0.
 	\end{equation}
 	Recall that  $\boldsymbol{L},\delta_0$ are the constants  specified  in \eqref{A-0} and Theorem \ref{thm:ae}, respectively,  and the spaces  $\mathscr{X}_\ell, \mathcal{Y}_\ell$ and $\mathcal M_{T}(\boldsymbol{L}), \mathcal E_{T}(\boldsymbol{L})$ are given in Definitions \ref{xel} and \ref{defspace}.  
 \end{proposition}
 
\begin{proof}
 Let $T_0$ be specified in Theorem \ref{thm:ae} and let $T_*$ be the time given in the assumption of this lemma. Choose   
  $f_0\approx \comi v^{-2\ell_0} $.  Observe $f_0$  satisfies condition $\boldsymbol{(H)}$  and lies in $\mathcal M_{T_*}(\boldsymbol{L})\cap \mathcal E_{T_*}(\boldsymbol{L}).$   By the assumption of this Proposition,  the linear equation
  \begin{equation*}
 	\big(\partial_{t}  +v\cdot\partial_{x} \big)f_{1} 
 	=Q(f_{0}, f_{1}), \quad f_1|_{t=0}=f_{in}.
\end{equation*}
admits a local solution $f_1 \in L^\infty([0,T_*]; \mathscr{X}_\ell)\cap L^2([0,T_*]; \mathcal{Y}_\ell)$ that is smooth in all variables at $t>0$.  Moreover, 
since 
\begin{equation*}
	\lim_{T\rightarrow 0} \bigg(\int_0^T\norm{\comi {D_v}^s\comi v^{\boldsymbol{\ell_1}+\frac{\gamma}{2}}  f_1}_{L^2_{x, v}}^2 dt\bigg)^{\frac12} =0,
\end{equation*}
then
there exists a time $T_1\leq \min\{T_0, T_*\}$ such that
\begin{equation*}
	\bigg(   \int_0^{T_1}\norm{ \comi {D_v}^s\comi v^{\boldsymbol{\ell_1}+\frac{\gamma}{2}} f_1}_{L^2_{x, v}}^2 dt\bigg)^{\frac12} \leq \frac{\delta_0}{4} 
\end{equation*}
with $\delta_0$  the constant specified in Theorem \ref{thm:ae}.  This enables us to apply Theorem \ref{thm:ae} to conclude that $f_1\geq 0$ satisfies condition $\boldsymbol{(H)}$ and moreover 
\begin{equation*}
	f_1\in \mathcal M_{T_1}(\boldsymbol{L})\cap \mathcal E_{T_1}(\boldsymbol{L}).
\end{equation*}
Then by the standard iteration, we can find a sequence of positive
numbers  
\begin{equation*}
	T_* \geq T_1\geq T_2\cdots\geq T_n\geq \cdots,
\end{equation*} 
 such that  for each $  j\geq 1,$   $f_j\geq 0$ solves the linear equation
 \begin{equation*}
 		\big(\partial_{t}  +v\cdot\partial_{x} \big) f_{j} 
 	=Q(f_{j-1}, f_{j}), \quad f_j|_{t=0}=f_{in},
 \end{equation*}
 satisfying
 \begin{equation}\label{ann}
 	\left\{
 	\begin{aligned}
 	&	f_j\geq 0  \textrm{ satisfies the condition } \boldsymbol{(H)}    \textrm{ and lies in }    \mathcal M_{T_j}(\boldsymbol{L})\cap \mathcal E_{T_j}(\boldsymbol{L}),\\
    &f_j \textrm{ is  smooth in all variables at  } t>0,\\
 	&\bigg(\int_0^{T_j}\norm{ \comi {D_v}^s\comi v^{\boldsymbol{\ell_1}+\frac{\gamma}{2}} f_j}_{L^2_{x, v}}^2 dt\bigg)^{\frac12} \leq \frac{\delta_0}{4}.
 	\end{aligned}
 	\right.
 \end{equation}
Choose a large integer $N\geq 5$   such that 
 \begin{equation}\label{n0}
 	 2^{-N} \leq   \frac{\delta_0}{32 \boldsymbol{L}}
 \end{equation}
 and define  
 \begin{equation}\label{T}
 	T:=\min\big\{T_1, T_2,\cdots, T_{N}, 4^{-1}\boldsymbol{L}^{-2} \big\} \leq T_{*}.
 \end{equation}
 Next we will apply the inductive argument to  prove that  the following property
  \begin{equation}\label{thetafn}
  \left\{
\begin{aligned}
&f_j\geq 0  \textrm{ satisfies the condition } \boldsymbol{(H)}  \textrm{ and lies in }    \mathcal M_{T}(\boldsymbol{L})\cap \mathcal E_{T}(\boldsymbol{L}),  \\
&f_j \textrm{ is  smooth in all variables at  } t>0,\\
 & \bigg(\int_0^{T}\norm{ \comi {D_v}^s\comi v^{\boldsymbol{\ell_1}+\frac{\gamma}{2}} f_j}_{L^2_{x, v}}^2 dt\bigg)^{\frac12} \leq \Big(1+\sum_{m=0}^{j}2^{-m}\Big)\frac{\delta_0}{4},
\end{aligned}
\right.
 \end{equation}
 holds for all $j\geq 1.$   

  {\it Step 1 (Initial Step).} By the definition of $T$ in \eqref{T}, the validity of \eqref{thetafn}  for $j\leq N$ just follows from  \eqref{ann}.

 {\it  Step 2 (Inductive Step).} Let $n\geq N$ be given, and suppose \eqref{thetafn} holds true for any $0\le j\leq n$, which yields that   
 \begin{equation}\label{fdns}
\left\{
\begin{aligned}
& f_n\geq 0 \textrm{ satisfies the condition } \boldsymbol{(H)}  \textrm{ and lies in }    \mathcal M_{T}(\boldsymbol{L})\cap \mathcal E_{T}(\boldsymbol{L}), \\
&f_n \textrm{ is  smooth in all variables at  } t>0,\\
& \bigg(\int_0^T\|\comi {D_v}^s  \comi v^{\boldsymbol{\ell_1}+\frac{\gamma}{2}} f_n\|_{L^2_{x, v}}^2dt \bigg)^{\frac12} \leq \Big(1+\sum_{m=0}^{n}2^{-m}\Big)\frac{\delta_0}{4}.	
\end{aligned}
\right. 
\end{equation}
 We will derive the validity of \eqref{thetafn} for $j=n+1.$

 Under the postulated solvability of the linear equation \eqref{liboltz},  the Cauchy problem 
 \begin{equation*}
		\big(\partial_t + v\cdot \partial_x\big) f_{n+1}
		= Q(f_n, f_{n+1}),  \quad f_{n+1}|_{t=0}=f_{in},
\end{equation*}
admits a solution 
 \begin{equation*}
 	f_{n+1}\in  L^\infty([0,T]; \mathscr{X}_{\ell})\cap L^2([0,T]; \mathcal{Y}_{\ell}), 
 \end{equation*}
which is smooth in all variables at  $ t>0$.
Define
   \begin{equation*}
   	\zeta_n=f_{n+1}-f_n,
   \end{equation*} 
which satisfies 
 	  	\begin{equation*}
 		\big(\partial_{t} +v\cdot\partial_{x}\big)\zeta_{n} 
 	=Q(f_n, \ \zeta_{n})+Q(\zeta_{n-1},\ f_{n}),\quad \zeta_{n}|_{t=0}=0.
 	\end{equation*} 
Thus
 \begin{equation}\label{diq}
 	\begin{aligned}
 		 \frac{1}{2}\frac{d}{dt}\norm{\comi v^{\boldsymbol{\ell_1}}\zeta_n}_{L^2_{x, v}}^2= \big (\comi v^{\boldsymbol{\ell_1}}Q(f_n, \ \zeta_{n}),\   \comi v^{\boldsymbol{\ell_1}} \zeta_n\big )_{L^2_{x, v}}+ \big (\comi v^{\boldsymbol{\ell_1}}Q(\zeta_{n-1},\  f_{n}), \  \comi v^{\boldsymbol{\ell_1}}\zeta_n \big )_{L^2_{x, v}}.
 	\end{aligned}
 \end{equation}
Condition \eqref{fdns} enables us to apply Lemma \ref{lem:te}  to obtain that  
 \begin{equation*}
 	\big (\comi v^{\boldsymbol{\ell_1}}Q(f_n, \ \zeta_{n}),\   \comi v^{\boldsymbol{\ell_1}} \zeta_n\big )_{L^2_{x, v}}\leq -\tilde c_0 \norm{\comi{D_v}^s\comi v^{\boldsymbol{\ell_1}+\frac{\gamma}{2}}  \zeta_n}_{L^2_{x, v}}^2    +C\boldsymbol{L}^{2} \|\langle v\rangle^{\boldsymbol{
 	 \ell_1} } \zeta_n\|_{L^2_{x, v}}^2.
 \end{equation*}
 Moreover, by Lemma 
  \ref{lem:te2} we have,  for any $\varepsilon>0,$  
  \begin{align*}
 		&\big| \big (\comi v^{\boldsymbol{\ell_1}}Q(\zeta_{n-1},\  f_{n}), \  \comi v^{\boldsymbol{\ell_1}}\zeta_n \big )_{L^2_{x, v}}\big|\\
&\leq \frac{\tilde c_0}{2} \|\comi {D_v}^s \comi v^{\boldsymbol{\ell_1}+\frac{\gamma}{2}} \zeta_n \|_{L^2_{x, v}}^2+C\boldsymbol{L}^2 \|\comi v^{\boldsymbol{\ell_1}}\zeta_n\|_{L^2_{x, v}}^2+ \frac{\tilde c_0}{64} \|\langle v\rangle^{\boldsymbol{\ell_1}+\frac{\gamma}{2}}\zeta_{n-1}\|_{L^2_{x, v}}^2 \\
&\quad    +\|\langle v\rangle^{\boldsymbol{\ell_1}} \zeta_{n-1}\|_{L^2_{x, v}}^2\Big(\varepsilon c_0 \norm{  f_n }_{\mathcal{Y}_\ell}^2
+C(\varepsilon c_0)^{-5}   \|\comi {D_v}^s \comi v^{\boldsymbol{\ell_1}+\frac{\gamma}{2}}     f_n \|_{L^2_{x, v}}^2+C\boldsymbol{L}^{2}\Big).
 	\end{align*}
 Consequently, combining the two estimates above with \eqref{diq}  yields, for any $\varepsilon>0,$ 
\begin{equation}
	 \label{begin}	
	 \begin{aligned}
 		& \frac{1}{2}\frac{d}{dt}\norm{\comi v^{\boldsymbol{\ell_1}}\zeta_n}_{L^2_{x, v}}^2+\frac{\tilde c_0}{2} \norm{\comi{D_v}^s\comi v^{\boldsymbol{\ell_1}+\frac{\gamma}{2}}   \zeta_n}_{L^2_{x, v}}^2 \\
 		&\leq    C\boldsymbol{L}^{2} \|\langle v\rangle^{\boldsymbol{ \ell_1} } \zeta_n\|_{L^2_{x, v}}^2+ \frac{\tilde c_0}{64}  \|\langle v\rangle^{\boldsymbol{\ell_1}+\frac{\gamma}{2}}\zeta_{n-1}\|_{L^2_{x, v}}^2 \\
 		 & \quad+\|\langle v\rangle^{\boldsymbol{\ell_1}}\zeta_{n-1}\|_{L^2_{x, v}}^2\Big(\varepsilon c_0 \norm{  f_n }_{\mathcal{Y}_\ell}^2
+C(c_0\varepsilon)^{-5}   \|\comi {D_v}^s \comi v^{\boldsymbol{\ell_1}+\frac{\gamma}{2}}     f_n \|_{L^2_{x, v}}^2+C\boldsymbol{L}^{2}\Big),
 	\end{aligned}
\end{equation}
and thus, multiplying the above estimate by $\frac{2}{\tilde c_0}$ and then integrating with  respect to $t,$     
 \begin{equation*}
	\begin{aligned}
	&\sqrt{ \frac{1}{\tilde c_0}} \sup_{t\leq T} \norm{\comi v^{\boldsymbol{\ell_1}}\zeta_n}_{L^2_{x, v}} +  \bigg(   \int_0^T    \norm{ \comi{D_v}^s\comi v^{\boldsymbol{\ell_1}+\frac{\gamma}{2}}  \zeta_n}_{L^2_{x, v}}^2dt \bigg)^{\frac12}\\
	&\leq  \sqrt{ \frac{C \boldsymbol{L}^{2}T}{\tilde c_0} }  \sup_{t\leq T} \norm{\comi v^{\boldsymbol{\ell_1}}\zeta_n}_{L^2_{x, v}}+\frac14\Big(   \int_0^T   \norm{ \comi{D_v}^s\comi v^{\boldsymbol{\ell_1}+\frac{\gamma}{2}}  \zeta_{n-1}}_{L^2_{x, v}}^2dt \Big)^{\frac12}\\
	&\quad+  \sup_{t\leq T}  \|\langle v\rangle^{\boldsymbol{\ell_1}}\zeta_{n-1}\|_{L^2_{x, v}} \bigg( \int_0^T \Big(\varepsilon c_0 \norm{f_n}_{\mathcal{Y}_\ell}^2+C(c_0\varepsilon)^{-5}   \|\comi {D_v}^s\comi v^{\boldsymbol{\ell_1}+\frac{\gamma}{2}}   f_n\|_{L^2_{x, v}}^2+C\boldsymbol{L}^2\Big)dt \bigg)^{\frac12}\\
	&\leq \sqrt{ \frac{C \boldsymbol{L}^{2}T}{\tilde c_0} }  \sup_{t\leq T} \norm{\comi v^{\boldsymbol{\ell_1}}\zeta_n}_{L^2_{x, v}}+\frac14\Big(   \int_0^T   \norm{ \comi{D_v}^s\comi v^{\boldsymbol{\ell_1}+\frac{\gamma}{2}}  \zeta_{n-1}}_{L^2_{x, v}}^2dt \Big)^{\frac12}\\
	&\quad+   \bigg[ \varepsilon \boldsymbol{L}^2   +C (c_0\varepsilon)^{-5} \Big(1+\sum_{m=0}^{n}2^{-m}\Big)\frac{\delta_0}{4}+C\boldsymbol{L}^2T\bigg]^{\frac12}  \sup_{t\leq T}  \|\langle v\rangle^{\boldsymbol{\ell_1}}\zeta_{n-1}\|_{L^2_{x, v}},
	\end{aligned}
\end{equation*}
the last inequality using \eqref{fdns} and \eqref{T}.  
Hence,  letting $\eps=\frac{1}{48\boldsymbol{L}^2}$ in the above estimate  and then   shrinking    $\delta_0$ and $T$    if necessary such that
\begin{equation*}
	 C (c_0\varepsilon)^{-5} \Big(1+\sum_{m=0}^{n}2^{-m}\Big)\frac{\delta_0}{4}\leq \frac{1}{48},\quad C\boldsymbol{L}^2T\leq \frac{1}{48}\ \  \textrm{ and }\ \  \frac{C \boldsymbol{L}^{2}T}{\tilde c_0}\leq 1,
\end{equation*}
 we conclude that, observing $1/\tilde c_0\geq 4,$
\begin{multline*}
	 \sup_{t\leq T} \norm{\comi v^{\boldsymbol{\ell_1}}\zeta_n}_{L^2_{x, v}} +  \Big(  \int_0^T    \norm{ \comi{D_v}^s\comi v^{\boldsymbol{\ell_1}+\frac{\gamma}{2}}  \zeta_{n}}_{L^2_{x, v}}^2dt \Big)^{\frac12}  
	 \\ \leq  
	\frac{1}{4}  \sup_{t\leq T}  \|\langle v\rangle^{\boldsymbol{\ell_1}}\zeta_{n-1}\|_{L^2_{x, v}}+\frac14\Big(    \int_0^T    \norm{ \comi{D_v}^s\comi v^{\boldsymbol{\ell_1}+\frac{\gamma}{2}}  \zeta_{n-1}}_{L^2_{x, v}}^2dt \Big)^{\frac12},
\end{multline*}
which implies
\begin{equation}\label{cse}
\begin{aligned}
	 &\sup_{t\leq T} \norm{\comi v^{\boldsymbol{\ell_1}}\zeta_n}_{L^2_{x, v}} +  \Big(\int_0^T  \norm{\comi{D_v}^s \comi v^{\boldsymbol{\ell_1}+\frac{\gamma}{2}}  \zeta_n}_{L^2_{x, v}}^2dt \Big)^{\frac12}\\
	 &\leq 4^{-n} \bigg[\sup_{t\leq T} \norm{\comi v^{\boldsymbol{\ell_1}}(f_1-f_0)}_{L^2_{x, v}}+ \Big(\int_0^T  \norm{\comi{D_v}^s \comi v^{\boldsymbol{\ell_1}+\frac{\gamma}{2}} (f_1-f_0)}_{L^2_{x, v}}^2dt \Big)^{\frac12}\bigg ]\\
     &\leq \Big(2\boldsymbol{L}+\frac{\delta_0}{2} \Big)4^{-n},
\end{aligned}
\end{equation}
the last inequality using \eqref{ann} as well as \eqref{T}.   
 Consequently, using the above estimate  and  \eqref{fdns} as well as  the estimate
 \begin{multline*}
 \Big(\int_0^T  \norm{\comi{D_v}^s\comi v^{\boldsymbol{\ell_1}+\frac{\gamma}{2}}    f_{n+1}}_{L^2_{x, v}}^2dt \Big)^{\frac12} \\
 	\leq \Big(\int_0^T  \norm{\comi{D_v}^s\comi v^{\boldsymbol{\ell_1}+\frac{\gamma}{2}}   \zeta_n}_{L^2_{x, v}}^2dt \Big)^{\frac12}   	+\Big(\int_0^T  \norm{\comi{D_v}^s\comi v^{\boldsymbol{\ell_1}+\frac{\gamma}{2}} f_n}_{L^2_{x, v}}^2dt \Big)^{\frac12} ,
 \end{multline*}
we obtain 
 	\begin{multline}\label{fn01}
 		 \Big(\int_0^T  \norm{\comi{D_v}^s \comi v^{\boldsymbol{\ell_1}+\frac{\gamma}{2}}   f_{n+1}}_{L^2_{x, v}}^2dt \Big)^{\frac12}\\
 		 \leq \Big(2\boldsymbol{L}+\frac{\delta_0}{2} \Big) 4^{-n}+\Big(1+\sum_{m=0}^{n}2^{-m}\Big)\frac{\delta_0}{4}\leq \Big(1+\sum_{m=0}^{n+1}2^{-m}\Big)\frac{\delta_0}{4},
 	\end{multline}
 	where the last inequality holds since it follows from \eqref{n0}  and the fact $n\geq N\ge5$ that
 	\begin{equation*}
 	\Big(2\boldsymbol{L}+\frac{\delta_0}{2} \Big) 4^{-n}\leq \Big(2\boldsymbol{L}+\frac{\delta_0}{2} \Big) 2^{-n} 2^{-N}\leq  2\boldsymbol{L} 2^{-n}  \frac{\delta_0}{32 \boldsymbol{L}}+2^{-n-1}\frac{\delta_0}{8}\leq  2^{-n-1}\frac{\delta_0}{4}.
 	\end{equation*}
The estimate \eqref{fn01} and condition \eqref{fdns}  enable  us to  apply Theorem \ref{thm:ae} to equation \eqref{liboltz}; this gives  that 
$f_{n+1}\geq 0$ satisfies condition $\boldsymbol{(H)}$ and belongs to $\mathcal M_T(\boldsymbol{L})\cap \mathcal E_T(\boldsymbol{L})$.  

 {\it Step 3 (Conclusion).} We have proved that \eqref{thetafn} holds true for $j=n+1$, hence it is true for all $j\in\mathbb{Z}_+$. 
Combining \eqref{thetafn} with the (sequential) Banach-Alaoglu theorem, we can extract a subsequence of 
$\big\{f_{n}\big\}$, still denoted by $\big\{f_{n}\big\},$ such that 
\begin{equation*}
	f_n\rightharpoonup f \  \textrm{ in }\  L^\infty\inner{[0,T]; \mathscr{X}_{\ell}\cap L_{\ell_0}^1}\cap L^2\inner{[0,T]; \mathcal{Y}_{\ell}\cap L_{\ell_0+\gamma}^1}  
	\end{equation*} 
 with respect to the weak* topology.  
 Furthermore, estimate \eqref{cse}  yields the strong convergence 
$$ 
\sup_{t\leq T} \norm{\comi v^{\boldsymbol{\ell_1}} (f_n-f )}_{L^2_{x, v}} +  \Big(\int_0^T  \norm{\comi{D_v}^s \comi v^{\boldsymbol{\ell_1}+\frac{\gamma}{2}}  (f_n-f)}_{L^2_{x, v}}^2dt \Big)^{\frac12}\rightarrow 0 \, \textrm{ as }\,  n\rightarrow+\infty.
$$ 
The limit $f$ satisfies the original Boltzmann equation \eqref{eq-1}, which is non-negative by the same argument as in Proposition \ref{nonnegative}.  Moreover, property \eqref{thetafn} implies
 $f\in \mathcal M_T(\boldsymbol{L})\cap \mathcal E_T(\boldsymbol{L})$ and satisfies \eqref{fsmall}.  This completes the proof of Proposition \ref{lem:itera}.
\end{proof}
    
\subsection{Proof of Theorem \ref{local-exsitence}}  
 
We will proceed through two steps to prove the existence and the uniqueness stated in  Theorem \ref{local-exsitence}.  

 {\it Step 1 (Existence)}. 
 To overcome the loss of weights, we  
  	consider the following   approximation equation:     
\begin{equation}\label{app+}
	 \big(\partial_t  +v\cdot\partial_x\big)h_\epsilon+\epsilon \comi v^{s+2\gamma} h_\epsilon=Q(g,h_\epsilon),\quad h_\epsilon|_{t=0}=f_{in},
\end{equation}
with given parameter $0<\epsilon\ll 1$.  Note that there is no the problem of loss-of-weights for \eqref{app+}.  Then by  the classical existence and regularization  theory for the linear Boltzmann equation,  we can find some $\epsilon$-dependent lifespan $T_\epsilon$,  such that 	 if  $g\geq 0$ satisfies condition $\boldsymbol{(H)}$ for $0<t\leq  T_\epsilon $ and belongs to  $ \mathcal{M}_{T_\epsilon}(\boldsymbol{L})\cap \mathcal{E}_{T_\epsilon}(\boldsymbol{L})$, then 
	equation \eqref{app+} 
admits a local solution $h_\epsilon\in L^\infty([0,T_\epsilon], \mathscr{X}_\ell)\cap L^2([0,T_\epsilon], \mathcal{Y}_\ell),$ which is $C^\infty$-smooth in all variables at $t>0.$ This enables us to
repeat the iteration procedure in Proposition \ref{lem:itera}, 
to conclude that the Cauchy problem 
\begin{equation}\label{appnon}
 	 \big(\partial_t   +v\cdot\partial_x\big)f_{\epsilon}+\epsilon \comi v^{s+2\gamma} f_{\epsilon} =Q(f_{\epsilon}, f_{\epsilon}),\quad f_{\epsilon}|_{t=0}=f_{in}.
 \end{equation}
admits  a local  weak solution $f_\epsilon \in L^\infty\big([0, \widetilde T]; \mathscr{X}_{\ell}\big)\cap L^2([0,\widetilde T]; \mathcal{Y}_{\ell}) $ for some $\widetilde T\leq T_\epsilon$, such that $f_\epsilon \geq 0$   satisfies condition $\boldsymbol{(H)}$ and belongs to  $ \mathcal{M}_{\widetilde T}(\boldsymbol{L})\cap \mathcal{E}_{\widetilde T}(\boldsymbol{L}).$  Moreover,  
 	\begin{equation*}
 		 \bigg(\int_0^{\widetilde T}\norm{\comi {D_v}^s \comi v^{\boldsymbol{\ell_1}+\frac{\gamma}{2}} f_\epsilon}_{L^2_{x, v}}^2 dt\bigg)^{\frac12} \leq  \delta_0.
 	\end{equation*}
 Following the proof of Theorem \ref{thm:ae}, 
  we  apply the standard continuous induction argument to extend the lifespan   to a time interval $[0,T]$ with $T$ {\it independent of } $\epsilon,$  such that 
  $$
  f_\epsilon\in L^\infty([0,T], \mathscr{X}_\ell)\cap L^2([0,T], \mathcal{Y}_\ell)
  $$ 
  satisfying equation \eqref{appnon}.   Letting $\epsilon\rightarrow 0$ and using the Banach-Alaoglu theorem,  we obtain the existence of solution $f$ to the original nonlinear equation \eqref{eq-1}, satisfying that
  \begin{equation}\label{laseq}
      \left\{
\begin{aligned}
& f\geq 0 \textrm{ satisfies the condition } \boldsymbol{(H)}  \textrm{ and lies in }    \mathcal M_{T}(\boldsymbol{L})\cap \mathcal E_{T}(\boldsymbol{L}), \\
& \bigg(\int_0^T\|\comi {D_v}^s  \comi v^{\boldsymbol{\ell_1}+\frac{\gamma}{2}} f\|_{L^2_{x, v}}^2dt \bigg)^{\frac12} \leq  \delta_0,	
\end{aligned}
\right. 
  \end{equation}
where $\delta_0$ is the constant specified in   Theorem \ref{thm:ae}.  
   
{\it Step 2 (Uniqueness).} Let $f$ be the solution constructed in the previous step with condition  \eqref{laseq} fulfilled, and  let $g\in L^\infty([0,T], \mathscr{X}_{\ell})$ be any non-negative solution to   equation \eqref{eq-1},  satisfying condition $\boldsymbol{(H)}$. 
 Denote $w=g-f$. Then
\begin{equation*}
\partial_t w+v\cdot\nabla_x w=Q(g,w)+Q(w,f), \quad w|_{t=0}=0.
\end{equation*}
Then
similar to  \eqref{begin}, we  use Lemmas \ref{lem:te} and \ref{lem:te2} to obtain that,
  for any $\eps>0,$
\begin{equation*}
\begin{aligned}
 		& \frac{1}{2}\frac{d}{dt}\norm{\comi v^{\boldsymbol{\ell_1}}w}_{L^2_{x, v}}^2= \big (\comi v^{\boldsymbol{\ell_1}}Q(g, \ w),  \  \comi v^{\boldsymbol{\ell_1}} w\big )_{L^2_{x, v}}+ \big (\comi v^{\boldsymbol{\ell_1}}Q(w,\  f),\  \comi v^{\boldsymbol{\ell_1}}w \big )_{L^2_{x, v}}\\
 		&\leq  -\frac{\tilde c_0}{2} \norm{\comi{D_v}^s \comi v^{\boldsymbol{\ell_1}+\frac{\gamma}{2}}  w}_{L^2_{x, v}}^2+  C\big(1+\norm{g}_{\mathscr{X}_\ell}^{2}\big) \|\langle v\rangle^{\boldsymbol{ \ell_1} } w\|_{L^2_{x, v}}^2   + \frac{\tilde c_0}{64} \|\langle v\rangle^{\boldsymbol{\ell_1}+\frac{\gamma}{2}}w\|_{L^2_{x, v}}^2 \\
&\quad    +\|\langle v\rangle^{\boldsymbol{\ell_1}} w\|_{L^2_{x, v}}^2\Big(\eps c_0\norm{ f}_{\mathcal{Y}_\ell}^2
+C(c_0 \eps)^{-5}  \|\comi {D_v}^s \comi v^{\boldsymbol{\ell_1}+\frac{\gamma}{2}}     f\|_{L^2_{x, v}}^2+ C\norm{f}_{\mathscr{X}_\ell}^2\Big).
 		\end{aligned}
\end{equation*}
This  with \eqref{laseq}   yields, for any $\eps>0$ and any $t_1\leq T,$	  
\begin{equation*}
    \begin{aligned}
	  \sup_{t\leq t_1} \norm{\comi v^{\boldsymbol{\ell_1}}w}_{L^2_{x, v}} & \leq   \sup_{t\leq t_1}  \|\langle v\rangle^{\boldsymbol{\ell_1}}w\|_{L^2_{x, v}}\Big( \tilde  C\boldsymbol{L}  +\tilde  C \sup_{t\leq T} \norm{g}_{\mathscr{X}_\ell}   \Big ) \sqrt {t_1} \\
	&\  + \sup_{t\leq t_1}  \|\langle v\rangle^{\boldsymbol{\ell_1}}w\|_{L^2_{x, v}}   \bigg( \int_0^T \Big(\eps c_0 \norm{ f}_{\mathcal{Y}_\ell}^2
+\tilde  C(  \eps c_0)^{-5}  \|\comi {D_v}^s \comi v^{\boldsymbol{\ell_1}+\frac{\gamma}{2}}     f\|_{L^2_{x, v}}^2 \Big)dt \bigg)^{\frac12}\\
& \leq  \bigg[\Big( \tilde  C\boldsymbol{L}  +\tilde  C \sup_{t\leq T}\norm{g}_{\mathscr{X}_\ell}   \Big ) \sqrt{t_1}+ \sqrt{ \varepsilon \boldsymbol{L}^2   +\tilde C (c_0\varepsilon)^{-5}  \delta_0^2}\, \bigg ]\sup_{t\leq t_1}  \|\langle v\rangle^{\boldsymbol{\ell_1}}w\|_{L^2_{x, v}},
	\end{aligned}
\end{equation*}
where $\tilde C$ is a constant depending only on the parameters $s,\gamma$ in \eqref{kern}, the constants $c_0, C_0$ in \eqref{lower+} and the quantities in condition $\boldsymbol{(H)}$. 
Choosing  $\eps$ sufficiently small and shrinking  $\delta_0$ if necessary, we obtain that
\begin{equation*}
\sup_{t\leq t_1}  \|\langle v\rangle^{\boldsymbol{\ell_1}}w\|_{L^2_{x, v}}\leq 	\Big( \tilde  C\boldsymbol{L}  +\tilde  C \sup_{t\leq T}\norm{g}_{\mathscr{X}_\ell}   \Big ) \sqrt{t_1} \sup_{t\leq t_1}  \|\langle v\rangle^{\boldsymbol{\ell_1}}w\|_{L^2_{x, v}}+\frac{1}{4}\sup_{t\leq t_1}  \|\langle v\rangle^{\boldsymbol{\ell_1}}w\|_{L^2_{x, v}},
	\end{equation*}
	which implies
	$w\equiv 0$ on  $[0,t_1]$ with 
	\begin{equation*}
		\sqrt{t_1}:=\frac{1}{4	\big( \tilde C\boldsymbol{L} +\tilde C \sup_{t\leq T} \norm{g}_{\mathscr{X}_\ell}        \big)}.
	\end{equation*}
Repeating the same argument on subsequent intervals yields  $w\equiv 0$ on the interval $[t_1,2t_1]$, and hence on the whole interval $[0,T].$  This proves uniqueness and completes the proof of Theorem \ref{local-exsitence}.

\bigskip
\noindent {\bf Acknowledgements.}
The research of W.-X.Li was supported by Natural Science Foundation of China
(Nos. 12325108, 12131017, 12221001), and the Natural Science Foundation of Hubei
Province (No. 2019CFA007). 
C.-J. Xu was supported by the NSFC (No.12031006, No.12426632) and the Fundamental 
Research Funds for the Central Universities of China.  
H.-G. Li was supported by the Natural Science Foundation 
of Hubei province (No.2025AFB696).

\appendix

\section{Change of variables}\label{Appendix}

We recall   two types of change of variables established  in \cite{MR1765272} (see the proof of  Lemma 1 therein). The first one is 
\begin{equation} \label{rech}
\begin{aligned}
   &\int_{\mathbb{R}^3\times\mathbb{S}^2}  |v-v_*|^\gamma  H(\sin\theta,\cos\theta) f(v') d\sigma dv  \\
   &\ \ \ \ \ \ \ \ \ \ \ \ \ \ \ \ = \int_{\mathbb{R}^3\times\mathbb{S}^2}  |v-v_*|^\gamma  H(\sin\theta, \cos\theta)\frac{1}{\cos^{3+\gamma}\frac{\theta}{2}}f(v) d\sigma dv,
\end{aligned}
\end{equation}
and the second is 
\begin{equation}\label{inch}
\begin{aligned}
   & \int_{\mathbb{R}^3\times\mathbb{S}^2}    |v-v_*|^\gamma f(v')H(\sin\theta,\cos\theta) d\sigma dv_* \\
   &\ \ \ \ \ \ \ \ \ \ \ \ \ \ \ \  = \int_{\mathbb{R}^3\times\mathbb{S}^2}   |v-v_*|^\gamma H(\sin\theta,\cos\theta) \frac{1}{\sin^{3+\gamma}\frac{\theta}{2}}f(v_*) d\sigma dv_*,
\end{aligned}
\end{equation}
 where $\cos\theta=\kappa\cdot\sigma$ with  $\kappa=\frac{v-v_*}{|v-v_*|}$, and $H(\sin\theta,\cos\theta)$ is an arbitrary function of $\sin\theta$ and $\cos\theta$ with compact support in  $
 \theta\in[0, \pi/2].$   The proof of \eqref{rech} can be found in \cite{MR1765272}. For completeness, we provide a proof of \eqref{inch} following the same argument.
 
 \begin{proof}[Proof of \eqref{inch}]
For fixed $v$, recall that $\kappa=\frac{v-v_*}{|v-v_*|}$ and $\cos\theta=\kappa\cdot\sigma,$  and thus $\sin\theta=\sqrt{1-(\kappa\cdot\sigma)^2}.$ Then the integral in \eqref{inch}
can be rewritten as 
\begin{align*}
&\int_{\mathbb{R}^3\times\mathbb{S}^2}  |v-v_*|^\gamma H(\sin\theta,\cos\theta) f(v') d\sigma dv_* \\
&\ \ \ \ \ \ \ \ \ \ \ \ \ \ \ \ =\int_{\mathbb{R}^3\times\mathbb{S}^2}   |v-v_*|^\gamma  H(\sqrt{1-(\kappa\cdot\sigma)^2},\kappa\cdot\sigma) f(v')\,d\sigma dv_*.
\end{align*}
Setting 
\begin{equation*}
   \tilde  \kappa=\frac{v'-v}{|v'-v|},
\end{equation*}
one has
\begin{equation*}
    \tilde\kappa\cdot\sigma=\cos\Big(\frac{\pi}{2}-\frac{\theta}{2}\Big)=\sin\frac{\theta}{2}=\sqrt{\frac{1-\kappa\cdot\sigma}{2}}.
\end{equation*}
As a result,   observing 
$$v'=\frac{v+v_*}{2}+\frac{|v-v_*|}{2}\sigma=v-\frac{v-v_*}{2}+\frac{|v-v_*|}{2}\sigma=v+\frac{|v-v_*|}{2}(\sigma-\kappa),$$
we have
\begin{align*}
| v'-v|=\frac{|v-v_*|}{2}|\sigma-\kappa|=|v-v_*|\sqrt{\frac{1-\kappa\cdot\sigma}{2}}=|v-v_*| \,\tilde \kappa \cdot\sigma.
\end{align*}
For each $\sigma$, we perform the change of variables $v_*\rightarrow v'$, and   
the Jacobian determinant is 
$$\Big| \frac{dv'}{dv_*}\Big|=\Big| \frac{1}{2}\, {\bf I}-\frac{1}{2}\kappa\otimes \sigma\Big|=\frac{1-\kappa\cdot\sigma}{2^3}=\frac{(\tilde \kappa \cdot\sigma)^2}{4}.$$
Note that $\tilde \kappa \cdot\sigma=\sin\frac{\theta}{2}\leq \frac{\sqrt{2}}{2} $ for $\theta\in[0,\pi/2],$ and then we use the above identities  to compute 
\begin{align*}
&\int_{\mathbb{R}^3\times \mathbb{S}^2} |v-v_*|^\gamma  H\Big (\sqrt{1-(\kappa\cdot\sigma)^2},\kappa\cdot\sigma\Big) f(v')\,d\sigma\,dv_*\\
&=\int_{\tilde \kappa \cdot\sigma\le \frac{\sqrt{2}}{2}} \frac{|v'-v|^\gamma}{(\tilde \kappa \cdot\sigma)^{\gamma}} H\Big(2(\tilde\kappa \cdot\sigma)\sqrt{1-(\tilde\kappa \cdot\sigma)^2}, 1-2(\tilde\kappa\cdot\sigma)^2\Big)\frac{4}{(\tilde\kappa\cdot\sigma)^2} f(v')\,d\sigma\,d v'.
\end{align*}
For fixed  $v'$,  we
 write, recalling  that $\tilde\kappa\cdot\sigma=\sin\frac{\theta}{2},$ 
\begin{equation*}
\begin{aligned}
  \sigma&=(\sigma\cdot\tilde\kappa)\tilde\kappa+\tilde\kappa^\perp \abs{\sigma-(\sigma\cdot\tilde\kappa)\tilde\kappa}\\
  &= \tilde\kappa\sin\frac{\theta}{2}+\tilde\kappa^\perp\cos\frac{\theta}{2}=\tilde\kappa\sin\frac{\theta}{2}+\big(\underbrace{\tilde\kappa_1 \cos\phi+\tilde\kappa_2 \sin\phi}_{=\tilde\kappa^\perp}\big)\cos\frac{\theta}{2},   
\end{aligned}
\end{equation*}
where  $\tilde\kappa_1,\tilde\kappa_2\in\mathbb S^2$ such that $\{\tilde\kappa,\tilde\kappa_1,\tilde\kappa_2\}$  forms an orthonormal basis of $\mathbb R^3$ and $\phi\in [0,2\pi]$ is independent on $\tilde\kappa$. In this polar coordinates, we have
\begin{align*}
&\int_{\tilde \kappa \cdot\sigma\le \frac{\sqrt{2}}{2}} \frac{|v'-v|^\gamma}{(\tilde \kappa \cdot\sigma)^{\gamma}} H\Big(2(\tilde\kappa \cdot\sigma)\sqrt{1-(\tilde\kappa \cdot\sigma)^2}, 1-2(\tilde\kappa\cdot\sigma)^2\Big)\frac{4}{(\tilde\kappa\cdot\sigma)^2} f(v')\,d\sigma\,d v'\\
&=2\pi  \int_{\mathbb{R}^3}\int^{\frac{\pi}{2}}_0\frac{|v'-v|^\gamma}{(\sin\frac{\theta}{2})^{\gamma}}  H(\sin\theta,\cos\theta)\frac{4}{\sin^2\frac{\theta}{2}} f(v')\,\frac12\cos\frac{\theta}{2} \,d\theta\,dv'\\
&=2\pi  \int_{\mathbb{R}^3}\int^{\frac{\pi}{2}}_0|v'-v|^\gamma  H(\sin\theta,\cos\theta)\frac{1}{\sin^{3+\gamma}\frac{\theta}{2}} f(v')\,\sin\theta \,d\theta\,dv'\\
&=2\pi  \int_{\mathbb{R}^3}\int^{\frac{\pi}{2}}_0|v_*-v|^\gamma  H(\sin\theta,\cos\theta)\frac{1}{\sin^{3+\gamma}\frac{\theta}{2}} f(v_*)\,\sin\theta \,d\theta\,dv_*\\
&=\int_{\mathbb{R}^3\times \mathbb{S}^2} |v-v_*|^\gamma H(\sin\theta,\cos\theta)\frac{1}{\sin^{3+\gamma}\frac{\theta}{2}}f(v_*)\,d\sigma\,dv_*.
\end{align*}
This completes the proof of equality \eqref{inch}.
 \end{proof}

\end{document}